\documentclass[11pt]{amsart}
\usepackage{geometry}                
\usepackage{graphicx}
\usepackage{ytableau}
\usepackage{amssymb,upgreek}
\usepackage{amsmath,amsthm}
\numberwithin{equation}{section}
\usepackage{esvect}

\newcommand\mynotes[1]{\textcolor{blue}{#1}}

\usepackage{epstopdf}
\usepackage{extarrows}
\usepackage[draft]{pdfcomment}
\usepackage[colorinlistoftodos]{todonotes}
\usepackage{hyperref}
\hypersetup{
pdftex,colorlinks,citecolor=red,linkcolor=blue
}

\usepackage{fancyhdr}

\usepackage{cclicenses}

\newtheorem{thm}{Theorem}[subsection]
\newtheorem{definition}[thm]{Definition}
\newtheorem{prop}[thm]{Proposition}
\newtheorem{lemma}[thm]{Lemma}
\newtheorem{cor}[thm]{Corollary}
\newtheorem{rem}[thm]{Remark}

\theoremstyle{definition}
\newtheorem{ex}[thm]{Example}

\newcommand{\ZZ}{\mathbf{Z}}
\newcommand{\RR}{\mathbf{R}}
\newcommand{\CC}{\mathbf{C}}   
\newcommand{\QQ}{\mathbf{Q}}
\newcommand{\KK}{K_{ur}}

\newcommand{\ha}{\mathcal{H}}
\newcommand{\PH}{\mathbb{P}}
\newcommand{\IH}{\mathbb{I}}
\newcommand{\calr}{\mathcal{R}}
\DeclareMathOperator{\IM}{\mathrm{IM}}

\newcommand{\tih}{\tilde{h}}
\newcommand{\tiH}{\tilde{H}}
\newcommand{\tisfI}{\tilde{\mathsf{I}}}

\newcommand{\mbf}{\mathbf}
\def\fb{\mathfrak{b}}

\newcommand{\respt}{\mathtt{r}}

\newcommand{\AAA}{\mathrm{A}} 
\newcommand{\BBB}{\mathrm{B}}
\newcommand{\CCC}{\mathrm{C}}
\newcommand{\DDD}{\mathrm{D}}

\newcommand{\str}{F_u}
\def\Frob{\ensuremath{\mathrm{Frob}}}

\DeclareMathOperator{\cInd}{\mathrm{c--Ind}}

\DeclareMathOperator{\fdeg}{\mathrm{fdeg}}

\newcommand{\E}{\mathsf{E}}
\newcommand{\D}{\mathsf{D}}

\newcommand{\od}{\mathrm{odd, dist}}
\newcommand{\mult}{\mathrm{mult}}
\newcommand{\cycl}{\mathrm{cycl}}

\def\dem{\ensuremath {\delta_-}}
\def\dep{\ensuremath {\delta_+}}

\def\emm{\ensuremath {m_-}}
\def\emp{\ensuremath {m_+}}

\def\diracdelta{\ensuremath{\updelta}}

\begin{document}

\title{On a uniqueness property of supercuspidal unipotent representations}

\author{Yongqi Feng}
\address[Y.~Feng]
{Department of Mathematics\\Shantou University\\ Shantou, Guangdong, 515063,  China}
\email{yqfeng@stu.edu.cn}

\author{Eric Opdam}
\address[E.~Opdam]
{Korteweg-de Vries Institute for Mathematics\\Universiteit van Amsterdam\\Science Park 105-107\\ 1098 XG Amsterdam, The Netherlands}
\email{e.m.opdam@uva.nl}

\date{\today}
\thanks{This research was partly supported by ERC--advanced grant no.~268105, 
and partly by an NWO VIDI grant no.~639.032.528 (PI: Maarten Solleveld). 
It is a pleasure to thank Mark Reeder and Maarten Solleveld for useful discussions and comments. Yongqi will extend his gratitude to Maarten Solleveld and IMAPP at Radboud University Nijmegen for the support.}

\begin{abstract} 
The formal degree of a unipotent discrete series character of a simple linear algebraic group over a non-archimedean local field 
(in the sense of Lusztig \cite{Lusztig-unirep}), is a rational function of $q$ evaluated at $q=\mbf{q}$, the cardinality of the residue field.
The irreducible factors of this rational function are $q$ and cyclotomic polynomials.
We prove that the formal degree of a \emph{supercuspidal} unipotent representation 
determines its Lusztig-Langlands parameter, up to twisting by weakly unramified 
characters.  
For split exceptional groups this result follows from the work of M.~Reeder \cite{Re}, and for the remaining exceptional cases this is verified in \cite{Fe2}. In the present paper we treat the classical families. 

The main result of this article characterizes unramified Lusztig-Langlands parameters which support a cuspidal local system in terms of formal degrees.
The result implies the uniqueness of so-called cuspidal spectral transfer morphisms 
(as introduced in \cite{Opds}) between unipotent affine Hecke algebras (up to twisting by unramified characters). 
In \cite{Opdl} the essential uniqueness of arbitrary unipotent spectral transfer morphisms was reduced to the cuspidal case.
\end{abstract}

\maketitle
\thispagestyle{fancy}
\rhead{Copyright\copyright 2020\\
\small{ This manuscript version is made available under the CC-BY-NC-ND 4.0 license} \url{http://creativecommons.org/licenses/by-nc-nd/4.0/}}

\tableofcontents

\section{Introduction}

Let $K$ be a non-archimedean local field. This paper is about a very special case of the local Langlands correspondence (LLC) for 
$K$-forms of connected almost simple groups which split over an unramified extension. 
Such a $K$-form is an inner twist $\mbf{G}^u$ of a connected almost simple unramified $K$-group $\mbf{G}$
(where $u\in G_{ad}:=\mbf{G}_{ad}(K_{ur})$ encodes the inner twisting, see paragraph \ref{par:inner} for more details).
We denote by $G^{\str}$ the group of $K$-rational points of $\mbf{G}^u$. 
For the sake of simplicity we assume in the Introduction that $\mbf{G}$ is simply connected, but in the body of the paper there is no such restriction.

The LLC for $G^{\str}$ is a conjectured non-abelian version of local class field theory that is currently 
of great interest in harmonic analysis and automorphic forms. 
Roughly speaking, the LLC predicts that each irreducible complex representation $\pi$ of $G^{\str}$ corresponds to a 
homomorphism called \emph{local Langlands parameter} (LLP) 
\begin{equation*}
\varphi_{\pi}: W_K \times \mbf{SL}_2(\CC) \to {}^LG:=G^\vee \rtimes \langle\theta\rangle, 
\end{equation*}
where $\theta$ is an outer automorphism of $G^\vee$ arising from the action of 
the Weil group $W_K$ on the root datum of $G^\vee$. 
Here $G^\vee:=\mbf{G}^\vee(\CC)$ is a connected complex reductive  
Lie group whose root datum is dual to that of $\mbf{G}$. Notice that ${}^LG$ depends only on $\mbf{G}$, 
not on the inner twist $u$. The set of LLP for $G^{\str}$ is the subset of the set of LLP for $G^F$ 
which satisfy the so-called relevance condition for $G^{F_u}$ (cf.  \cite[8.2(ii)]{Bo}). 
We identify local Langlands parameters which are conjugate under the action of $G^\vee$.

Restricted to discrete series representations, the mapping $\pi \mapsto \varphi_\pi$ should be a finite-to-one surjection onto the set
$\Phi^2(G^{\str})$ of $G^\vee$-conjugacy classes of parameters with finite centralizer; such parameters are likewise called \emph{discrete}. 
We note that $\Phi^2(G^{\str})=\Phi^2(G^{F})$, since the $G^{F_u}$-relevance condition is trivially satisfied in the discrete case. 
The members of the fibre of the surjection $\pi \mapsto \varphi_\pi$
are conjectured to be parameterized by, roughly, the local systems $\sigma_\pi$  
on the $G^\vee$-orbit of $\varphi_\pi$. Thus it is expected that the union of the sets of (equivalence classes of) 
discrete series representations $\pi$ of the various inner forms $G^{\str}$ is in bijection with the set of $G^\vee$-orbits of pairs $(\varphi_\pi,\sigma_\pi)$, where $\varphi_\pi$ is a discrete local Langlands parameter and $\sigma_\pi$ is a local system as above.\\
 
For $\mbf{G} = \mbf{GL}_n$ the LLC has been proved in a strong form \cite{H1,HT}, namely there are invariants ($L$-functions  
and $\epsilon$-factors of pairs) attached independently to representations and parameters, and the map $\pi \mapsto \varphi_\pi$ is uniquely characterized \cite{H2,H3} by agreement of the invariants $L$ and $\epsilon$ on both sides. Using the relation 
with the local Langlands correspondence for $\mbf{SL}_n$ \cite{GK} these results can be extended also to $\mbf{SL}_n$. 
 
For groups other than $\mbf{GL}_n$ or $\mbf{SL}_n$, 
representation-theoretic invariants as robust as $L$-functions and $\epsilon$-factors are not known, 
and it is not even conjectured how one might characterize the map 
$\pi \mapsto \varphi_\pi$ uniquely,  
even when such a mapping has been found. 
However, about twenty years ago it was observed by Reeder \cite{Re-Iwa-sph,Re} that for a discrete series representation $\pi$, 
the formal degree $\textup{fdeg}(\pi)$ may function as a weak-but-useful substitute for the invariants $L(\pi)$ and $\epsilon(\pi)$. 
One obvious advantage of this proposal is that $\fdeg(\pi)$ is a well defined invariant 
for any discrete series representation $\pi$ of $G^{\str}$. 

On the Galois side, the expression conjectured to correspond to $\fdeg(\pi)$ \cite{HII} is essentially the adjoint gamma factor 
$\gamma(\varphi)$ of $\varphi$. More precisely, it is expected that 
\begin{equation}\label{eq:HII}
\fdeg(\pi)=C_{\pi}\gamma(\varphi_\pi), 
\end{equation} 
where $C_{\pi}$ is a nonzero rational number independent of $K$ (and which admits an explicit expression in terms of the local system $\sigma_\pi$). 
The adjoint gamma factor $\gamma(\varphi)$ is a complex number that is computed from the action of $\varphi \big(W_K \times \mbf{SL}_2(\CC) \big)$ on the Lie algebra of $G^\vee$. 
It is known that $\gamma(\varphi)$ is obtained via specialization:
$\gamma(\varphi)=\gamma(\varphi,\mathbf{q})$,  
where $\gamma(\varphi,q)$ is a rational function in an indeterminate $q$, 
and $\mathbf{q}$ denotes the cardinality of the residue field of $K$.
\footnote{We use an additive character of $K$ of order $1$ to normalize 
the Haar measure and the adjoint $\gamma$-factor. This gives a factor $\mbf{q}^{-\textup{dim}(\mbf{G})/2}$ 
in $\gamma(\varphi)$ and in $\textup{fdeg}(\pi)$ compared to \cite{HII}
(where an additive character of order $0$ is used).
In our normalization these numbers are evaluations at $q=\mbf{q}$ 
of rational functions of the indeterminate $q$.}
 
In this paper we restrict our attention to \emph{unipotent} representations
\cite{Lusztig-unirep}. If $\mbf{G}$ is of adjoint type Lusztig has  
constructed a bijection $\pi \leftrightarrow (\varphi_\pi,\sigma_\pi)$ from the set of equivalence classes of unipotent representations to the set of $G^\vee$-orbits of 
enhanced Langlands parameters with $\varphi_\pi$ \emph{unramified}, i.e.~trivial on the inertia subgroup $I_K\subset W_K$. Denote by $\Phi_{ur}^2(G^{F})$ 
the set of $G^\vee$-conjugacy classes of discrete local Langlands parameters $\varphi$ which are unramified.
It follows from Lusztig's work \textit{op.~cit.} that one can define a finite-to-one map 
\begin{equation}\label{eq:lusztig}
\bigsqcup_{\omega = [u] \in H^1(F_u,G_{ad})}\mathrm{Irr}_{upt}^2\,(G^{\str}) \to \Phi_{ur}^2(G^{F}), \quad \pi \mapsto \varphi_\pi
\end{equation}

Reeder \cite{Re} has shown for split exceptional groups  $\mbf{G}^u$ of adjoint type that, if $\pi$ is a unipotent discrete series 
representation of $\mbf{G}^{F_u}$,  
there exists a \emph{unique} discrete unramified Langlands parameter $\varphi_\pi$ and a nonzero rational constant $C_\pi$ such that 

\begin{equation}\label{eq:fdeg=gamma}
\fdeg(\pi, q)=C_{\pi}\gamma(\varphi_\pi, q)
\end{equation}
(an equality of \emph{rational functions} in $q$). 

Using Reeder's results it was verified in \cite{HII} that, accepting 
\eqref{eq:lusztig} as a LLC for unipotent discrete series representations, \eqref{eq:HII} holds for unipotent representations 
of a simply connected, split exceptional group defined over a non-archimedean local field. 
In \cite{Opdl} this result was extended to all unramified simple groups $\mbf{G}^u$ of adjoint type 
and their inner twists. \\
  
The main motivation of this paper is the converse question; whether or not 
the unramified local Langlands parameter $\varphi_\pi$ is uniquely determined by \eqref{eq:fdeg=gamma} in general. \\

It is not hard to settle this question affirmatively for the remaining exceptional groups, see \cite{Fe2}. 
Therefore we will focus on the remaining case of classical $K$-groups in the present paper. By a classical $K$-group we 
mean a $K$-form of a connected almost simple group whose root system is classical 
(see paragraph \ref{par:Class} for more detail). Observe that our notion of a classical group only depends on the 
central isogeny class.
%

We have not been able to answer this question 
in complete generality in this case. Our main result (which we formulate here only for the case of simply connected 
classical groups, for the sake of simplicity; see Theorem \ref{thm:A} for the general case) states that this uniqueness 
result is true when we concentrate on the case of \emph{supercuspidal} unipotent representations:
\begin{thm}\label{thm:main}
Suppose that $\mbf{G}^u$ is an inner twist of a simply connected unramified classical group $\mbf{G}$ over $K$. 
For every supercuspidal unipotent representation $\pi$ of $\mbf{G}^u$, 
there exists a \emph{unique} equivalence class of discrete unramified parameters $\varphi\in \Phi^2_{ur}(G^{F})$ such that 
\eqref{eq:fdeg=gamma} holds as rational functions in $q$ for some constant rational factor $C_{\pi}$.
\end{thm}

The computability of both sides of \eqref{eq:fdeg=gamma} make this uniqueness result quite useful. 
It can be amplified to show the uniqueness of an assignment $\pi\mapsto\varphi_\pi$ 
for tempered unipotent representations which is in compliance 
with the two conjectures on the Plancherel measure in terms of local Langlands parameters as formulated in \cite{HII} (one can find this result in \cite{Opdl} in the form of the essential uniqueness of \emph{spectral transfer maps} to the Iwahori--Hecke algebra of $G^{F}$). \\

We now describe the strategy of the proof of Theorem \ref{thm:main}. 
The  proof consists of a case-by-case analysis. 
For $\mbf{G}$ isogenous to $\mbf{SL}_n$ the result is obvious, since there is essentially only one unipotent discrete local Langlands parameter. 
In the remainder of this introduction we will concentrate on the other classical 
cases.\\

By the work of Lusztig on unipotent characters for finite groups of Lie type, the 
left-hand side of \eqref{eq:fdeg=gamma} is explicitly known for all cuspidal unipotent characters $\pi$ (see e.g.~\cite{Carter}). 
By inspection of the list of degrees of cuspidal unipotent characters of classical groups, a crucial fact is that in the remaining classical cases, $\fdeg(\pi)$ is the reciprocal of a product of \emph{even} cyclotomic polynomials $\Phi_{2n}(q)$, up to multiplication by a power of $q$ and a rational constant. 

Let us call an unramified discrete local Langlands parameter $\varphi$ ``of even degree" if $\gamma(\varphi,q)$ has no odd cyclotomic factors in numerator or denominator. This property turns out to be very distinctive; it reduces the collection of  parameters $\varphi_\pi$ which may possibly satisfy \eqref{eq:fdeg=gamma} dramatically. The first objective in the proof will be the classification of the parameters of even degree, which will be carried out in Section \ref{sec:noodd}. Once this has been established, the remaining part of the proof consists of a direct comparison of $\fdeg(\pi, q)$ and $\gamma(\varphi,q)$ where $\varphi$ runs through the list of eligible parameters. This second step will be carried out in Section \ref{sec:pfmain}.\\
 
We use various technical tools to compute $\gamma(\varphi,q)$ 
for a discrete unramified local Langlands parameter $\varphi$. The starting point, which follows by combining the results of \cite{HO,Re,HII}, is the fundamental fact 
that $\gamma(\varphi,q)$ can be interpreted as the residue of the $\mu$-function of the normalized Iwahori--Hecke algebra of $G^{F}$ at the semisimple element 
$\respt_\varphi=c_\varphi s_\varphi\theta \in G^\vee \theta$, where 
\begin{equation}\label{eq:phi(Frob)}
c_\varphi := \varphi\left(\mathrm{id},\begin{pmatrix} q^{1/2} & 0\\0 & q^{-1/2}\end{pmatrix} \right)\in G^\vee \quad s_\varphi\theta :=\varphi\left(\Frob,\begin{pmatrix} 1 & 0\\ 0 & 1\end{pmatrix}\right)\in G^\vee\theta. 
\end{equation}
Let $S^\vee$ denote a $\theta$-stable maximal torus of $G^\vee$. The group $W_0=W(G^\vee,S^\vee)^\theta$ is isomorphic to the relative Weyl group 
$W(G^{F},S^{F})$ and acts on the torus $T:=S^\vee/(1-\theta)S^\vee$. It is well-known that the $\mathrm{Int}(G^\vee)$-orbits of semisimple elements in $G^\vee\theta$ 
are in bijection with the $W_0$-orbits in $T$ (cf.~\cite[Proposition 6.7]{Bo}). By abuse of notation we will denote the image of $G^\vee\respt_\varphi$ in $T$ by $W_0\respt_\varphi$. Furthermore, it is not difficult to see that $W_0\respt_\varphi$ completely determines $G^\vee\varphi$ (cf.~\cite[Appendix]{Opd3}). 

The points $\respt_\varphi\in T$ are the so-called \emph{residual points} of $T$ (cf.~\cite{OpdMoscow}), a notion that can be characterized completely in terms of the $\mu$-function of the Iwahori--Hecke algebra of $G^F$, which is a $W_0$-invariant rational function defined on $T$
(cf.~\cite[(9) and (13)]{Opds}). The residue of $\mu$ at $\respt_\varphi$ is denoted by 
$\mu^{\{\respt_\varphi\}}(\respt_\varphi)$ (cf.~\cite[3.2.1]{Opds})
and we have the identity (see text around \cite[(38)]{Opdl}):
\begin{equation}\label{eq:residue}
\gamma(\varphi,q)=C\mu^{\{\respt_\varphi\}}(\respt_\varphi)
\end{equation}
where $C$ is some nonzero rational constant. We call $\respt_\varphi$ ``of even degree'' if and only if $\varphi$ is 
of even degree. \\

Recall that $\varphi$ is discrete if and only if the centralizer $C_{G^\vee}(\varphi)$ 
is finite, or equivalently if and only if the connected centralizer $H:=C_{G^\vee}(s_\varphi \theta)^\circ$ is semisimple, and 
$$
u_\varphi:=\varphi\left(\mathrm{id},\begin{pmatrix} 1 & 1\\ 0 & 1\end{pmatrix}\right) \in H
$$ 
is a \emph{distinguished} unipotent element in $H$. In the classical cases we are considering, $H$ is isogenous to a direct product of at most two almost simple classical groups $H_-$ and $H_+$. Thus, the $G^\vee$-orbit of $\varphi$ is completely specified by a pair of partitions 
$\lambda=(\lambda_-, \lambda_+)$ corresponding to distinguished unipotent elements of $H_\pm$. 
We write $\respt_\varphi=(-\respt_{\lambda_-},\respt_{\lambda_+})$ accordingly, 
where $\respt_{\lambda_\pm}=c_{\lambda_\pm}$ is a real residual point 
(i.e.~a residual point with positive coordinates) of the $\mu$-function 
$\mu_{H_\pm}$ associated with the Iwahori--Hecke algebra of $H_\pm$.
An important result towards the classification of residual points of even degree states that 
$(-c_{\lambda_-},c_{\lambda_+})$ is of even degree 
if and only if $\respt_{\lambda_-}$ and $\respt_{\lambda_+}$ are both of even degree. 
Hence we may concentrate on the individual partitions $\lambda_\pm$ at this stage.
\\

The $\mu$-functions which play a role in the present context are $\mu$-functions of an affine Hecke algebra of the form 
$\CCC_l (\delta_-,\delta_+)[q^\fb]$ with parameters (following the conventions of \cite[3.2.1]{Opdl}) $m_+(\alpha)=\fb$ and 
$m_-(\alpha)=0$ if $\alpha$ is a type $\DDD_n$-root, and 
$m_\pm(\beta)=\fb \delta_\pm$ if $\beta$ is a short root in $\BBB_n$. 
The relevant parameter triples $(\delta_-,\delta_+;\fb)$ of the $\mu$-functions $\mu_G$ are: 
$(1/2,1;2), (0,1/2;2), (1/2,1/2;1), (0,1;1), (0,0;1), (1,1;1)$. \\

The classification of the partitions $\lambda_\pm$ such that the corresponding residual point $\respt_{\lambda_\pm}$ of 
$\mu_{H_\pm}$ is of even degree is quite easy for $\delta_\pm=1/2$, but rather difficult for $\delta_\pm\in\{0,1\}$. 
For example the third case $\CCC_l (1/2,1/2)[q]$, corresponding to $\mbf{G}$ of type $\mbf{SO}_{2l+1}$,  is an easy analysis without using any further technical tools, see Proposition \ref{cor:deltaishalf}.
The integral parameter cases seem to be difficult because of the abundance of possible cancellations which take place in the residue computation. In these cases we use the \emph{Spectral Transfer Morphisms} (STMs) developed in \cite{Opds,Opdl} as a technical tool. For any $\delta\in\{0,1\}$ and residual point $\respt_{\lambda}$ 
(depending on a partition $\lambda$ with odd, distinct parts) 
of the $\mu$-function of the Hecke algebra of type $\CCC_l (0,\delta)[q]$, 
using these STMs we find: a parameter $m\in (\ZZ \pm 1/4)_+$, 
a positive integer $d$, a partition $\nu$ with corresponding residual point $\respt_\nu$ for the 
$\mu$-function $\mu_{2;1/4,m}^d$ of 
the Hecke algebra of type $\CCC_d (1/4,m)[q^2]$ such that:
\begin{equation}\label{eq:STM}
\mu_{1;0,\delta}^{l,\{\respt_{\lambda}\}}(\respt_\lambda)=C\mu_{2;1/4,m}^{d,\{\respt_{\nu}\}}(\respt_\nu)
\end{equation}
for some nonzero rational constant $C$. The algorithm to determine $(d,m, \nu)$ in terms of $(l,\delta, \lambda)$ 
and vice versa will be discussed in Section \ref{e-s}.
This enables us to transfer the classification of residual points ``of even degree'' for $\mu=\mu_{1;0, \delta}^l$ to those of 
$\mu^d_{2; 1/4,m}$.  
In the latter case the residual points of even degree are much easier to classify, since the 
parameters $(1/4,m)$ are more ``generic" than $(0,\delta)$. 
This procedure is carried out in Section \ref{sub:oddextra}. \\

The structure of this paper is as follows: We focus on classical 
groups (in the sense of paragraph \ref{par:Class}). In Section \ref{sec:uptrep}, we recall the basics of supercuspidal unipotent representations and explicitly give the formal degrees that we will study. In Section \ref{sec:adgamma}, we will give the discrete unramified local Langlands parameters for the supercuspidal representations described in Section \ref{sec:uptrep}. The main theorem, which states that the equation \eqref{eq:fdeg=gamma} between supercuspidal formal degree and adjoint $\gamma$-factor uniquely characterizes the Langlands correspondence, is also formulated in Section \ref{sec:adgamma}. An important technique to prove the main theorem is the use of  
\emph{spectral transfer morphisms} (STMs). The \emph{existence} results of STMs of \cite{Opdl, Fe2} enable us to associate a certain central character of the Iwahori--Hecke algebra of the quasi-split inner form $G^F$ of $G^{F_u}$, given the formal degree of a supercuspidal unipotent representations of $G^{F_u}$. The desired uniqueness result 
is now reduced to a property of the $\mu$-function of this Iwahori--Hecke algebra at certain central characters.  
The relevant spectral transfer morphisms are explicitly described in Sections \ref{sec:adgamma} and \ref{e-s}. In particular, the extra-special algorithm in Section \ref{e-s} enables us to translate the integral parameters of affine Hecke algebras to generic parameters. In Sections \ref{sec:noodd} and \ref{sec:pfmain} we prove the main theorem, using STMs and the evenness of the supercuspidal formal degrees (of groups isogenous to classical groups). 

\section{Conventions and notations}

\subsection{The group $\mbf{G}$ and its $L$-group}

Throughout this paper, $K$ will be a non-archimedean local field with finite residue field $\frak{F}$, and we fix a separable closure $K_s$ of $K$. 
Let $\KK\subset K_s$ be the maximal unramified extension of $K$. The residue field $\overline{\frak{F}}$ of $\KK$ is an algebraic closure of $\frak{F}$. 
There are isomorphisms of Galois groups 
$\mathrm{Gal}(\KK/K) \simeq \mathrm{Gal}(\overline{\frak{F}}/\frak{F}) \simeq \hat{\ZZ}$. The \emph{geometric Frobenius element} $\Frob$, whose \emph{inverse} induces the automorphism $x \mapsto x^{|\frak{F}|}$ for any $x \in \overline{\frak{F}}$, is a topological generator 
of $\mathrm{Gal}(\KK/K)$. 
Let $I_K =\mathrm{Gal}(K_s/\KK)$ be the inertia subgroup of $\mathrm{Gal}(K_s/K)$ 
and let $W_K$ be the Weil group of $K$.\\
 
Unless otherwise stated, $\mbf{G}$ denotes a connected, absolutely almost simple and unramified linear algebraic group over $K$. 
By this we mean that (1) $\mbf{G}$ is quasi-split over $K$; (2) $\mbf{G}$ splits over $\KK$, and (3) the extension of scalars group $\mbf{G}\times_{K}K_s$ is almost simple. 

Fix a $K$-Borel subgroup $\mbf{B}$ and a maximally $K$-split maximal $K$-torus $\mbf{S}\subset\mbf{B}$ which splits over $\KK$. 
\emph{Unconventionally, we will denote by $\Sigma$ the based root datum of the identity component of the dual $L$-group ${}^LG$ (see below), and let 
$$
\Sigma^\vee=(X^*(\mathbf{S}),\Sigma^\vee_0, \Delta_0^\vee, X_*(\mbf{S}), \Sigma_0, \Delta_0)
$$ 
denote the based root datum of  $(\mathbf{G, B, S})$, where $X^*(\mathbf{S}):=\mathrm{Hom}(\mbf{S}, \mbf{G}_m)$ is the \emph{character} lattice of $\mathbf{S}$, and the choice of the base $\Delta_0^\vee \subset \Sigma_0^\vee$ is compatible with $\mbf{B}$.} Also we have $X_*(\mbf{S}):=\mathrm{Hom}(\mbf{G}_m, \mbf{S})$ as the cocharacter lattice of $\mbf{S}$. Suppose that $Q(\Sigma_0)$ is the $\ZZ$-span of $\Sigma_0$. The quotient $\Omega:= X_*(\mbf{S})/Q(\Sigma_0)$ is a finite abelian group. \\

By construction the Galois group $\mathrm{Gal}(K_s/K)$ acts by automorphisms of the triple $(\mbf{G, B, S})$. 
Since $\mbf{G}$ is split over $\KK$, the induced action on the root datum $\Sigma^\vee$ 
factors through $\mathrm{Gal}(\KK/K)$, and hence is completely determined by the action of the geometric Frobenius element $\Frob$. 
We denote the corresponding automorphism of $\Sigma^\vee$ by $\theta$.

Let $G^\vee$ be a connected complex reductive group with a maximal torus $S^\vee$ and Borel subgroup $B^\vee\supset S^\vee$, 
such that the based root datum $\Sigma$ of 
$(G^\vee, B^\vee, S^\vee)$ is dual to $\Sigma^\vee$.  The action of $\Frob$ on $\Sigma$ will also denoted by $\theta$. 
We fix an ``\'epinglage'' for $(G^\vee,B^\vee,S^\vee)$, and use it to lift the action of $\mathrm{Gal}(K_s/K)$ determined by 
$\theta$ to an action on $G^\vee$.  
We define ${}^L G:= G^\vee \rtimes \langle \theta \rangle$ as the Langlands $L$-group of $\mbf{G}$ over $K$. \\

\subsection{Inner forms and Kottwitz's isomorphism}\label{par:inner}

Let $Z(\mbf{G})\subset \mbf{S}$ be the centre of $\mbf{G}$, and let $\mbf{G}_{ad}:=\mbf{G}/Z(\mbf{G})$ be the adjoint form of $\mbf{G}$. 
Then $\mbf{S}_{ad}:=\mbf{S}/Z(\mbf{G})\subset \mbf{G}_{ad}$ is a maximally $K$-split maximal $K$-torus.  
We write $G, G_{ad}, S, S_{ad}$ for the group of $\KK$-rational points of the groups 
$\mbf{G, G}_{ad}, \mbf{S}, \mbf{S}_{ad}$ respectively. 
We denote the action of $\Frob$ on $G$ or $G_{ad}$ by $F$.

For our purpose we shall consider the various inner forms of $\mbf{G}$ in this paper. Recall that isomorphism classes of inner forms are parameterized by the Galois cohomological set $H^1(K, \mbf{G}_{ad})$. By a theorem of Steinberg, which says that $H^1(\KK, \mbf{G}_{ad})$ is trivial, we obtain a canonical bijection $H^1(K, \mbf{G}_{ad}) \cong H^1(F, G_{ad})$. \\

Kottwitz's isomorphism (cf. \cite{DeRe,Ko}) 
gives a natural bijection between $H^1(F, G_{ad})$ and $\mathrm{Hom}(\pi_0({}^L Z_{ad}), \CC^\times)$, where ${}^LZ_{ad}$ denotes the centre of ${}^LG_{ad}$. In our present setting, the centre $Z_{ad}^\vee$ of $G^\vee_{ad}$ is finite, and there is a canonical isomorphism 
$$
\mathrm{Hom}(Z_{ad}^\vee, \CC^\times) \simeq \Omega_{ad}:=X_*(\mathbf{S}_{ad})/Q(\Sigma_0).
$$ 
One obtains natural bijections  
$$
H^1(F, G_{ad})\simeq \mathrm{Hom}({}^LZ_{ad}, \CC^\times) \simeq \Omega_{ad}/(1-\theta) \Omega_{ad},
$$ 
such that the $K$-quasi-split group $\mbf{G}$ corresponds to the trivial character. 

A cocycle $z\in Z^1(F, G_{ad})$ is completely determined by its image $u:=z(F)\in G_{ad}$ of $F$, 
and its cohomology class $\omega=[z]\in H^1(F, G_{ad})$ is represented by the $F$-twisted 
conjugacy class of $u$. The inner twist of $\mbf{G}$ corresponding to $z$ is denoted by $\mbf{G}^u$. We have $\mbf{G}^u(\KK)=G$, and $\Frob$ acts on $\mbf{G}^u(\KK)$ via the $K$-automorphism 
$F_u:=\mathrm{Int}(u) \circ F \in \mathrm{Aut}_K(G)$ of $G$. 

Following \cite{DeRe} we choose, for each class $\omega\in H^1(F, G_{ad})$, 
an inner twist $F_u$ of $F$ representing $\omega$ as follows. 
In the Bruhat--Tits building of $G_{ad}$, choose a fundamental alcove $C_{ad}$ such that $F(C_{ad}) = C_{ad}$ inside  
the apartment determined by $S_{ad}$, and let $\IH_{ad}\subset G_{ad}$ be the corresponding $F$-stable Iwahori subgroup. 
We have an isomorphism $\Omega_{ad}\simeq N_{G_{ad}}(\IH_{ad})/\IH_{ad}$. In \cite{DeRe} it is shown that one can 
choose the representing cocycle $z\in Z^1(F, G_{ad})$ for 
$\omega$ such that $u=z(F) \in N_{G_{ad}}(\IH_{ad})$.
\subsection{Weakly unramified characters}

Let ${}^0 S := \mathcal{O}_{\KK}^{\times} \otimes X_*(\mbf{S})$ be the maximal bounded subgroup of $S$. We have $X_*(\mbf{S}) \simeq S/{}^0 S$ as (free) abelian groups. Let $G_1 = \langle {}^0 S,  G' \rangle$ be the group generated by ${}^0 S$ and the derived group $G'$ of $G$ (see \cite[Corollary 2.2]{Opdl}). Kottwitz \cite{Ko} defined a natural short exact sequence 
\begin{equation}\label{eq:kotsurk0} 
1\to G_{1} \to G \to \mathrm{Irr}(Z^\vee) \to 1,
\end{equation}
where $Z^\vee$ denotes the centre of $G^\vee$. Moreover, if we take 
invariants for $\mathrm{Gal}(\KK/K)$ we again obtain an exact sequence (cf.~\cite{Ko,HR}): 
\begin{equation}\label{eq:kotsurk0} 
1\to G^{\str}_1 \to G^{\str} \xrightarrow[]{w_G} \mathrm{Irr}(Z^\vee)^\theta \to 1,
\end{equation}
We say that a character $\chi$ of $G^{\str}$ is \emph{weakly unramified} if $\chi$ is trivial on the kernel $G^{\str}_1$ of the Kottwitz homomorphism $w_G: G^{\str}\to \mathrm{Irr}(Z^\vee)^\theta$. We have an identification 
$$
\Omega = \mathrm{Irr}(Z^\vee).
$$ 
Thus, weakly unramified characters of a semisimple reductive group 
form a finite abelian group $X_{wur}(G^{\str})$ which can be identified with $(\Omega^\theta)^*$ (the group of irreducible characters of $\Omega^\theta$).

\subsection{The relative root datum of $G$}\label{par:relRD}

Choose a maximal $K$-split torus $\mbf{S}_d$ of $\mbf{G}$ contained in $\mbf{S\subset G}$.
We will determine the root datum of $\mbf{G}$ with respect to $\mbf{S}_d$ (i.e.~the relative root datum), by the following method given by Springer \cite[Section 15.3.6]{Spr}.

We denote (somewhat awkwardly) the character and cocharacter lattices of $\mbf{S}_d$ by $X^\vee, X$ respectively. 
We see that $X=X_*(\mbf{S})^\theta$. Let $\mathrm{ann}_{X^*(\mbf{S})}(X)$ be the annihilator of $X$ in $X^*(\mbf{S})$. Then $X^\vee \cong X^*(\mbf{S})/\mathrm{ann}_{X^*(\mbf{S})}(X)$. 

Let $V: = \RR\otimes_{\ZZ} X^*(\mbf{S})$ and define a Euclidean inner product on $V$ which is invariant under $N_G(S)/S$.  Using this inner product we identify $V$ with its dual. Let $\mathtt{pr}: V \to \RR \otimes_{\ZZ} X^\vee$ be the natural projection. The image $\mathtt{pr}(\Sigma_0^\vee)$ of $\Sigma_0^\vee$ is a non-reduced root system. We define $R_0^\vee \subset \mathtt{pr}(\Sigma_0^\vee)$ to be the set of non-multipliable vectors. Then $R_0^\vee$ is a reduced root system 
(in $\RR\otimes_{\ZZ} X^*(\mbf{S}_d)$). Let $R_0 \subset X$ be the root system dual to $R_0^\vee$. Then the based relative 
root datum equals $(X^\vee, R_0^\vee,F_0^\vee, X, R_0,F_0)$. The complex algebraic torus $S^\vee \subset G^\vee$ has character lattice $X_*(\mbf{S})$. The complex algebraic torus $T$ with character lattice 
$X=X_*(\mbf{S})^\theta$ can be identified with 
the quotient $T:=S^\vee/(1-\theta) S^\vee$. 
Let $W_0=W(R_0)=W(\Sigma_0)^\theta$ denote the relative finite Weyl group acting on $T$.  \\

In this paper, the root datum $\calr^{\IM}:=(X, R_0, F_0, X^\vee, R_0^\vee, F_0^\vee)$ dual to the relative root datum obtained above is used more frequently. This is the root datum of the Iwahori--Hecke algebra $\ha^{\IM}$ of the quasi-split group $\mbf{G}$. The role played by $\ha^{\IM}$ will be explained in Sections \ref{sub:fduni} and \ref{sub:muSTM}.

\subsection{Classical groups}\label{par:Class} As was pointed out in the Introduction, this paper mainly focusses 
on the case of classical groups. By a \emph{classical group} we mean 
a $K$-form of a connected almost simple group whose root system is classical and whose 
relative root system (in the sense of \cite[Section 15.3]{Spr}) is classical (possibly non-reduced) or empty 
(if the group is $K$-anisotropic). In addition we require (as we do throughout this paper) that the group 
splits over an unramified extension. This definition excludes the quasi-split triality group ${}^3\textup{D}_4$
(which was treated in \cite{Fe2}). It is well known (by classification, see e.g. \cite[Sections 4.3, 4.4]{Tits} for an overview) 
that the only anisotropic cases in this list are those isogenous to $\textup{SL}_1(D)$, where $D$ is a rank $(l+1)^2$ 
(with $l\in\mathbb{N}$) central division algebra over $K$. This is an anisotropic inner twist of $\textup{SL}_{l+1}(K)$.

Observe that our notion of a classical group only depends on the central isogeny class. Every  
central isogeny class of classical groups in this sense has a representative which is the special automorphism 
group of a sesquilinear binary form (possibly $0$) of a division algebra over $K$, see e.g. \cite[Sections 4.3, 4.4]{Tits}.
For a list of such representatives for the unramified groups: See the list at the end of paragraph \ref{sec:DULLPCG}.

We allow arbitrary inner forms of a given unramified classical group (we do not restrict ourselves to pure inner forms). 
\section{Unipotent supercuspidal representations}\label{sec:uptrep}

By \cite{Lusztig-unirep} we know that the 
\emph{supercuspidal unipotent representations} $\pi$ of $G^{\str}$ are the irreducible summands of a compactly induced representation of a cuspidal unipotent representation $\sigma$ of a maximal parahoric subgroup 
$\PH^{\str} \subset G^{\str}$, and $\pi$ determines the $G^{\str}$-conjugacy class of $(\PH^{\str},\sigma)$. 
Here $\PH$ is an $\str$-stable maximal parahoric subgroup of $G$. \

The pair $(\PH^{\str},\sigma)$ is a \textit{type} in the sense of \cite{BK}. Attached to this \textit{type} is a Hecke algebra (via the last theorem in \cite{MP}, see also \cite{Mo}). Lusztig \cite[1.20]{Lusztig-unirep} explicitly described the structure of this Hecke algebra. In particular, when $\PH^{\str} \subset G^{\str}$ is maximal, the Hecke algebra of this \textit{type} is isomorphic to 
the group ring $\CC[\Omega^{\theta,\PH}]$, where $\Omega^{\theta, \PH}$ is defined as follows. Fix an $\str$-stable Iwahori subgroup $\mathbb{I}\subset G$. 
Consider the set of $G^{\str}$-conjugacy classes of $\str$-stable parahoric subgroups $\PH\subset G$. 
Each such class has representatives $\PH$ such that $\mathbb{I}\subset \PH$. 
 Using the isomorphism $\Omega \simeq N_{G}(\mathbb{I})/\mathbb{I}$, we see that the group $\Omega^\theta$  
acts on the set of such $G^{\str}$-conjugacy classes of $\str$-stable parahoric subgroups $\PH\subset G$. 
We define $\Omega^{\theta, \PH}\subset \Omega$ as the subgroup of elements fixed by $\theta$ and fixing $\PH$. 

Given a \textit{type} $(\PH^{\str}, \sigma)$, fix an extension of $\sigma$ to $N_{G^{\str}}(\PH^{\str},\sigma)$. The set of irreducible components 
of $\cInd_{\PH^{\str}}^{G^{\str}} \sigma$ is in canonical bijection with the set
$(\Omega^{\theta,\PH})^*$ of irreducible characters of $\Omega^{\theta,\PH}$. In particular this set is an orbit under the natural action of the group 
$(\Omega^\theta)^*$ of weakly unramified characters.

\subsection{Formal degrees of supercuspidal unipotent representations} \label{sub:fdsup} 

We want to single out the discrete unramified LLP $\varphi$ whose adjoint $\gamma$-factor 
$\gamma(\varphi,q)$ is equal (up to nonzero rational constant factors) to the formal degree of a unipotent 
supercuspidal representation of 
$G^{\str}$. We start by compiling the list of such supercuspidal unipotent formal degrees for classical groups. 

Following \cite{DeRe}, we normalize the Haar measure on the locally compact group $G^{\str}$ 
by the rule that (with $\overline{\PH}^{\str}$ the reductive quotient of $\PH^{\str}$) 
$$
\mathrm{vol}(\PH^{\str}) = v^{-\mathrm{rk}(\mbf{G})} |\overline{\PH}^{\str}|_{p'}
$$ 
for any parahoric subgroup $\PH^{\str}$ of $G^{\str}$. 
Here, $\mathrm{rk}(\mbf{G})$ is the absolute rank of $\mbf{G}$, 
and $|\overline{\PH}^{\str}|_{p'}$ is the largest factor of 
$|\overline{\PH}^{\str}|$ prime to $\mathrm{char}\,\frak{F}$. 
The factor $|\overline{\PH}^{\str}|_{p'}$ can be determined using the list in \cite[\S 2.9]{Carter}. Then, for each $\chi\in(\Omega^{\theta,\PH})^*$, 
the formal degree of the corresponding irreducible cuspidal unipotent 
summand $\pi_\chi$ of $\cInd_{\PH^{\str}}^{G^{\str}} \sigma$ equals (cf.~\cite{BK})\begin{equation}\label{eq:fdegdef}
\fdeg(\pi_\chi, q):=|\Omega^{\theta, \PH}|^{-1} \mathrm{vol}(\PH^{\str})^{-1} \deg(\sigma).
\end{equation}

Let us now concentrate on the isogeny classes of classical groups $\mbf{G}$ (unitary, special orthogonal and symplectic) other than the easy case $\mbf{G}=\mbf{PGL}_n$. 
A maximal $F_u$-stable parahoric subgroup $\PH$ 
corresponds to a maximal $u\theta$-stable subdiagram $\mathsf{J}$ in the affine Dynkin 
diagram $\tisfI$ of $G$ 
(where $u=z(F)\in Z^1(F, G_{ad})$ such that $\omega=[z]$). 
By Lusztig's work on cuspidal unipotent characters for finite groups of Lie type (see e.g.~\cite[\S 13.7]{Carter}), the subdiagram $\mathsf{J}$ 
must be one of the following six cases (we exclude the triality ${}^3 \DDD_4$):
\begin{enumerate}
\item[(i)] $\tisfI= \widetilde{ {}^2\AAA_n }, \mathsf{J}_{a,b} = {}^2\AAA_s \sqcup{}^2\AAA_t$, where 
$s=(a/2)(a+1)-1, t=(b/2)(b+1)-1$ for some $a, b \in \ZZ_{\geq0}$ and $s+t+1=n \geq 2$. Here ${}^2\AAA_x$ represents the empty diagram if $x\in \{-1, 0\}$.
We write $n=2l$ or $n=2l-1$ depending on the parity of $n$.

\item[(ii)] $\tisfI = \widetilde{ \BBB_l }$ and $\mathsf{J}_{a,b}  = \DDD_s\sqcup \BBB_t$ ($a$ even), or $\tisfI = {}^2\widetilde{\BBB_l}$ 
and $\mathsf{J}_{a,b}={}^2\DDD_s \sqcup\BBB_t$ ($a$ odd), 
where 
$s=a^2, t=b(b+1)$ for some $a,b\in \ZZ_{\geq0}$, 
and $s+t=l \geq 2$. Here $\DDD_x$ represents the empty diagram 
if $x=0$, and ${}^2\DDD_x$ represents the empty diagram if $x=1$. 

\item[(iii)] $\tisfI = \widetilde{\CCC_l}, \mathsf{J}_{a,b}  = \CCC_s \sqcup\CCC_t$, where $s=a(a+1), t=b(b+1)$ for some $a, b\in \ZZ_{\geq0}$ and  $s+t=l \geq 2$;

\item[(iv)] $\tisfI = \widetilde{\DDD_l}, \mathsf{J}_{a,b} = \DDD_s \sqcup \DDD_t$ ($a,b$ even) or 
${}^2\DDD_s \sqcup {}^2\DDD_t$ ($a,b$ odd) 
where $s=a^2, t=b^2$ for some $a, b \in \ZZ_{\geq0}$, with $a\equiv b \pmod{2}$ and $s+t=l \geq 4$. Or 
$\tisfI = \widetilde{{}^2\DDD_l}, \mathsf{J}_{a,b} = \DDD_s \sqcup {}^2\DDD_t$  ($a$ even) or 
$\mathsf{J}_{a,b} = {}^2\DDD_s \sqcup \DDD_t$  ($b$ even)
where $s=a^2, t=b^2$ for some $a, b \in \ZZ_{\geq0}$ with $a\not\equiv b\pmod{2}$ and $s+t=l \geq 4$;

\item[(v)] $\tisfI= {}^2\widetilde{\CCC_l}$,  
$\mathsf{J}_{a,b}  = {}^2\AAA_s \sqcup \CCC_t\sqcup  \CCC_t$, where $s=(a/2)(a+1)-1, t=b(b+1)$ for some 
$a, b\in \ZZ_{\geq0}$ and $s+2t+1=l \geq 2$;

\item[(vi)] $\tisfI = {}^2\widetilde{\DDD_l}$ ($l$ even) or ${}^4\widetilde{\DDD_l}$ ($l$ odd) 
or ${}^2(\widetilde{{}^2\DDD_l})$ ($l$ odd) or  ${}^4(\widetilde{{}^2\DDD_l})$ ($l$ even), where the 
inner left superscript refers to the order of $\theta$, a finite type $\DDD$-diagram automorphism, and the 
outer left superscript $2$ or $4$ indicates the order of the affine diagram automorphism $u\theta$, where 
$u\in\Omega$ is such that it does not commute with the unique nontrivial finite type $\DDD$-automorphism,  
$\mathsf{J}_{a,b} ={}^2\AAA_s \sqcup\DDD_t \sqcup\DDD_t$ ($b$ even, or equivalently $\textup{ord}(u\theta)=2$)
or ${}^2\AAA_s \sqcup{}^2\DDD_t \sqcup{}^2\DDD_t$ ($b$ odd, or equivalently $\textup{ord}(u\theta)=4$), 
where $s=(a/2)(a+1)-1, t=b^2$ for some $a, b \in \ZZ_{\geq 0}$ and 
$s+2t+1=l \geq 4$.  
\end{enumerate}

Cases (v) and (vi) are associated with unipotent types of the non-split inner forms of $\mbf{G}$ determined by $\omega \in \Omega_{ad}/(1-\theta) \Omega_{ad}$ where via Kottwitz's isomorphism, $\omega$ corresponds to the action of ${}^LZ_{ad}$ in the spin representations of ${}^LG_{ad}$. We call these cases \emph{extra-special} (the reason will become clear in next section).\\

Working out \eqref{eq:fdegdef} in these the cases yields the following expressions for the supercuspidal unipotent formal degrees for some rational constant factors $C_{a,b}$: 
\begin{equation}\label{eq:classicalfdeg}
C_{a,b}^{-1}\textup{vol}(\mathbb{P}_{\mathsf{J}_{a,b}})\fdeg(\pi_\chi,q)=
\begin{cases}
d_a^{ \{{}^2\AAA\}}(q) d_b^{\{ {}^2\AAA\}}(q) & \text{ in case (i)};\\ 
d_a^{\DDD}(q) d_b^{\BBB}(q) & \text{ in case (ii)}; \\
d_a^{\BBB}(q) d_b^{\BBB}(q) & \text{ in case (iii)}; \\
d_a^{\DDD}(q) d_b^{\DDD}(q) & \text{ in case (iv)}; \\
d_a^{\{ {}^2\AAA\} }(q) d_b^{\BBB}(q^2) & \text{ in case (v)}; \\
d_a^{\{ {}^2\AAA\} }(q) d_b^{\DDD}(q^2) & \text{ in case (vi)}.
\end{cases}
\end{equation}
Here $d_a^{\BBB}(q)$ denotes the degree of the cuspidal unipotent representation of  
$\mathsf{SO}_{2l+1}(\mathbb{F}_q)$ (or equivalently, of $\mathsf{Sp}_{2l}(\mathbb{F}_q)$) 
where $l=a(a+1)$, $d_a^{\DDD}(q)$ denotes the degree of the cuspidal unipotent representation of  
$\mathsf{SO}_{2l}(\mathbb{F}_q)$ (if $a$ even) or of $\mathsf{SO}^*_{2l}(\mathbb{F}_q)$ (if $a$ odd) with $l=a^2$,  
and $d_a^{\{ {}^2\AAA\} }(q)$ denotes the degree of the cuspidal unipotent representation of  
$\mathsf{SU}_{l+1}(\mathbb{F}_{q})$ with $l=(a/2)(a+1)-1$. 
These cuspidal unipotent degrees   
are listed in e.g. \cite[\S 13.7]{Carter}, and the cardinalities $|\mathsf{G}(\mathbb{F}_q)|$
(needed to compute $\textup{vol}(\mathbb{P}_{\mathsf{J}_{a,b}})$) are given in 
\cite[\S 2.9]{Carter}. 

These six cases cover the supercuspidal unipotent representations of the classical groups, except for the case of the weakly unramified characters of the anisotropic inner form of $\mbf{PGL}_{n+1}$.
By \eqref{eq:fdegdef} we see that in this case (see \cite[2.2.3]{Opdl}) the formal degree of a supercuspidal representation equals $(n+1)^{-1}[n+1]_q^{-1}$ 
(where $[n]_q:=\frac{v^n-v^{-n}}{v-v^{-1}}$ denotes a $q$-integer). From these explicit formulae we check that:
\begin{cor}\label{cor:sc} 
For every supercuspidal unipotent representation $\pi_\chi$ of an almost simple 
classical group over $K$, the function $\fdeg(\pi_\chi,q)^{-1}$ is the 
product of a nonzero rational constant, a power of $v$ (with 
$q=v^2$), and a polynomial in $q$ which is a product of cyclotomic factors. 
The parahoric group $\mathbb{P}_\mathsf{J}$ from which $\pi_\chi$ is induced 
is determined, up to isomorphisms, by the multiplicities of the cyclotomic factors of $\fdeg(\pi_\chi,q)^{-1}$
(in particular, the set $\{a,b\}$ such that $\mathsf{J}=\mathsf{J}_{a,b}$ in the cases above is determined by these multiplicities). 
Except for the anisotropic case, $\fdeg(\pi_\chi,q)^{-1}$ has only \emph{even} cyclotomic factors. 
\end{cor}

\section{Adjoint $\gamma$-factors of discrete unramified Langlands parameters}\label{sec:adgamma}

We will discuss discrete unramified local Langlands parameters and the notion of formal degree on the parameter side of the local Langlands correspondence. 

\subsection{Reduction to the Iwahori--spherical case}\label{sub:fduni}

Hiraga, Ichino and Ikeda \cite{HII} conjectured that the formal degree of a discrete series representation $\pi$ of the group of 
points of a reductive group over a local field is, up to a power of $v$ and a nonzero rational constant factor, 
equal to the so-called adjoint $\gamma$-factor $\gamma(\varphi_\pi)$ (see also \cite[Conjecture 7.1]{GR}) of the local Langlands parameter $\varphi_\pi$ of $\pi$. 
Even before the HII-conjecture, Reeder \cite{Re} had shown this for unipotent representations of split simple 
exceptional groups of adjoint type defined over a non-archimedean local field. 
In general, assuming that the LLP is given by Lusztig's parameterization of unipotent representations, the second mentioned author obtained: 
\begin{thm}(\cite{Re}, \cite[Theorem 4.11]{Opdl})\label{thm:HIIupt2}
Let $\mbf{G}^u$ be an inner form of an unramified group $\mbf{G}$ of adjoint type 
defined over a non-archimedean local field $K$ corresponding to 
$\omega=[u]\in H^1(F,G)$. 
Let $\pi$ be an irreducible unipotent discrete series representation of $G^{F_u}$. 
Let $\varphi$ denote the discrete unramified LLP 
for $\pi$ as defined in \cite{Lusztig-unirep}, 
and let $\rho$ be the local system on the $G^\vee$-orbit of $\varphi$ associated with $\pi$.
Let us write $\pi=\pi_{(\varphi,\rho)}$.
Then
$\textup{fdeg}(\pi_{(\varphi,\rho)},q)=\pm \frac{\textup{dim}(\rho)}{|A_\varphi|}v^{\textup{-dim}(\mbf{G})}\gamma(\varphi,q)$. 
\end{thm}

By Theorem \ref{thm:HIIupt2}, the formal degree $\textup{fdeg}(\pi_{(\varphi,\rho)},q)$ of $\pi_{(\varphi,\rho)}$ is, modulo nonzero rational constant factors, 
independent of $\rho$. In this case of groups of adjoint type we know by \cite{ReWmodel}, that there is a unique generic member in Lusztig's 
packet of unipotent representations associated with $\varphi$, namely $\pi_{(\varphi,\rho)}$ with $\rho$ equal to the trivial representation. This generic representation 
is an Iwahori--spherical representation of the quasi-split adjoint group $G^{F}$. 
In particular, modulo nonzero rational constant factors, 
the list of formal degrees of unipotent discrete series characters of the groups $G^{\str}$, 
coincides with the list of formal degrees of the subset of Iwahori--spherical discrete series representations of $G^F$.\\

Recall that the category of admissible irreducible Iwahori--spherical representations of $G^{F}$ is equivalent to the category of simple modules of the Iwahori--Hecke algebra $\ha^{\IM}$ of $G^{F}$ (cf.~\cite{BK}). Under this equivalence, the Plancherel measure 
of $G^{F}$ corresponds to the spectral measure of the trace of the canonical Hilbert algebra structure of $\ha^{\IM}$. 
This spectral measure was computed in \cite{Opd3} in terms of a rational function $\mu=\mu^{\IM}$ on 
the complex torus $T$ with $X^*(T)=X^*(\mbf{S})^\theta$ (see Section \ref{par:relRD}).
Let $\pi$ be an Iwahori spherical discrete series character of $G^F$, 
and let $\delta_\pi$ be the corresponding discrete series character of $\ha^{\IM}$.
In \cite{Opdl, Opd3} the formal degree of $\delta_\pi$ was computed, 
up to nonzero rational constant factors, as a ``residue" of $\mu^{\IM}$ at the character $\chi_\pi\in W_0\backslash T$ by which the centre of $\ha^{\IM}$ acts on $\delta_\pi$. By 
the above remarks, this residue of $\mu^{\IM}$ at $\chi_\pi$ is equal to the formal 
degree of $\pi$.\\

In order to connect this with the adjoint gamma factors of discrete unramified LLP of $G^F$
we will now explain the fundamental bijection between the set of equivalence classes $[\varphi]$ 
of discrete unramified LLP of $G^F$ and the set of central characters $\chi\in W_0\backslash T$ of 
$\ha^{\IM}$ which support discrete series characters of $\ha^{\IM}$. This requires first of 
all a closer look at $\ha^{\IM}$.

\subsection{The Iwahori--Hecke algebra and the Kac diagram}\label{sub:Kac} 

We will use the conventions of \cite[2.1, 2.3]{Opds} for affine Hecke algebras.
An affine Hecke algebra is defined by the based root datum $\calr$ 
and a parameter function $m_R$ on the set of affine roots associated 
with $\calr$. We will express these data for $\ha^{\IM}$
in terms of the Kac root system for $G^\vee\theta$ as discussed in \cite{Retorsion}. 
The root datum $\calr^{\IM}=(X,R_0,F_0,Y,R_0^\vee,F_0^\vee)$ of $\ha^{\IM}$ was 
described in Section \ref{par:relRD}. 
A $\theta$-orbit $\iota$ in the affine extension of the Dynkin diagram of $(\Sigma_0^\vee,\Delta_0^\vee)$ 
corresponds to a node of the affine extension of the Dynkin diagram of $(R_0^\vee,F_0^\vee)$, and the 
corresponding Hecke generator $T_\iota$ of $\ha^{\IM}$ satisfies the Hecke relation:
\begin{equation}
(T_\iota+1)(T_\iota-q^{|a_\iota|})=0
\end{equation}
where $a_\iota$ denotes the equivalence class of coroots in $\Sigma_0^\vee$ associated with $\iota$ 
as defined in \cite[\S 3.3]{Retorsion} (this is an orbit of (orthogonal) coroots under $\langle\theta\rangle$, 
except for one case in ${}^2A_{2n}$, where it is a union of two $\langle\theta\rangle$-orbits). 
This determines the parameter function $m_R^{\IM}$ on $\calr^{\IM}$. 
We will provide the list of these parameter values explicitly in all relevant cases below.\\

The parameters of $\ha^{\IM}$ can best be expressed in terms of the parameter 
function $m_R^\vee$ (\cite[2.1.5]{Opds}) of the spectral diagram for $\ha^{\IM}$, extended 
linearly to a function on the so-called Kac roots of $G^\vee \theta$ (\cite[(25)]{Retorsion}), as we will now see. 
In \cite[Definition 2.10]{Opds} a root system $R_m$ is introduced whose roots are certain multiples of those of $R_0$. 

Recall that $T=S^\vee/(1-\theta)S^\vee$ is 
a quotient of $(S^\vee)^\theta$ with kernel $(S^\vee)^\theta\cap (1-\theta)S^\vee$, which is a subgroup of the 
$f$-torsion subgroup of $(S^\vee)^\theta$, where $f$ denotes the order of $\theta$ 
(cf. \cite[Section 3]{Retorsion}). The roots of $R_m$, when lifted to $(S^\vee)^\theta$, 
are the roots called $\gamma_\iota$ in \cite[\S 3.3]{Retorsion}. In other words 
we have $R_m=\Phi_\theta$, in the notation of \cite{Retorsion}, viewed as 
root systems in the character lattice of $(S^\vee)^\theta$.
The {\emph{spectral diagram}} (\cite[2.3.3]{Opds}) 
of $\mathcal{H}^{\IM}$ is the diagram of the ordinary affine extension $R_m^{(1)}$ of $R_m$, 
with simple affine roots of the form $a^\vee_\iota$. 

Each node in the spectral diagram has a label given by a multiplicity function $m_R^\vee$ defined in 
\cite[(5)]{Opds}. 
Going through the list of cases, one easily verifies that 
$m_R^\vee(a^\vee_\iota)=f_\iota$ with $f_\iota$ as in \cite[Table 1]{Retorsion}.  
The Kac diagram of $G^\vee\theta$ (denoted as $\mathcal{D}(\frak{g}, \theta)$ in \cite[Table 1]{Retorsion}) 
is a twisted affine root system of characters on $(S^\vee)^\theta$ obtained from the spectral diagram 
$R_m^{(1)}$ by dividing its affine simple roots by their 
multiplicities $f_\iota$. Hence we have the following: 
\begin{prop}\label{prop:m}
The multiplicity function $m_R^\vee$ (extended linearly) has constant value $1$ 
on the affine simple Kac root system $\mathcal{D}(\frak{g}, \theta)$.
\end{prop}

Another important aspect of $\ha^{\IM}$ is its isogeny type. By \cite[(3)]{Opds} we see that 
the dual affine Weyl group $W^\vee$ associated with $\calr^{\IM}$ is of the form   
\begin{equation}\label{eq:dualaff}
W^\vee:=W_0\ltimes Y=W((\calr^{\IM, \max})^\vee)\rtimes\Omega_Y^\vee
\end{equation}
with $\calr^{\IM, \max}=(P(R_m), R_0, F_0, Q(R_m^\vee), R_0^\vee, F_0^\vee)$, the maximal 
possible extension of $\calr^{\IM}$ obtained by replacing $X$ by $P(R_m)$, and where
\begin{equation}\label{eq:OmY}
\Omega_Y^\vee=Y/Q(R_m^\vee).
\end{equation}
Observe that $W^\vee$ equals the extended affine Weyl group $\widetilde{W_\theta}$ of 
\cite[3.2]{Retorsion}, although our lattice $Y$ is the projection onto the space of $\theta$-invariants 
of the lattice $X^*(\mathbf{S})$. (This is a more general situation than that of \cite{Retorsion}, where 
only the case $X^*(\mathbf{S})=P(\Sigma_0^\vee)$ is being considered.) Notice that 
$W((\calr^{\IM, \max})^\vee)$ is the (unextended) affine reflection group of the (untwisted) affine extension 
of $R_m$. It follows from \cite[(2)]{Opds} that 
\begin{prop}\label{prop:OmY}
Let $\ha^{\IM, \max}$ be the maximally extended affine Hecke algebra with root datum 
$ \calr^{\IM, \max}$ and parameter function $m_R$. Then 
\begin{equation}
\ha^{\IM}=(\ha^{\IM, \max})^{\Omega_Y^\vee}.
\end{equation}
\end{prop}
\subsection{Residue points and discrete unramified LLP}\label{sub:RPandLLP}
Recall that an unramified Langlands parameter for $\mbf{G}$ is a homomorphism
$$
\varphi: \Frob^{\ZZ} \times \mbf{SL}_2(\CC) \to {}^L G=G^\vee \rtimes \langle \theta \rangle
$$
such that $\varphi$ is algebraic on the identity component $\{\mathrm{id}\}\times \mbf{SL}_2(\CC) $, 
and such that the image $\varphi(\Frob, \mathrm{id})$ is a semisimple element of $G^\vee\theta$. Two parameters $\varphi$ and $\varphi'$ are regarded as equivalent if they are $G^\vee$-conjugate. 

Since the element 
$\varphi(\Frob, \mathrm{id})=s\theta\in G^\vee\theta$ is semisimple, 
its connected centralizer $H:=C_{G^\vee}(s\theta)^\circ$ is a reductive group. Put 
 $u:=\varphi \big(\mathrm{id}, \left(\begin{smallmatrix} 1 & 1 \\ 0 & 1 \end{smallmatrix}\right) \big)$, then 
$u \in H$ is a unipotent element. It follows that the set of equivalence classes $[\varphi]$
of unramified Langlands parameters is in canonical bijection with the set of 
$\mathrm{Int}(G^\vee)$-orbits of pairs $(s\theta,u)$ 
with $s\theta\in G^\vee\theta$ semisimple, and $u\in C_{G^\vee}(s\theta)^\circ$  unipotent.  \\ 

On the parameter side,  $W_0\backslash T$ can be identified with the $\mathrm{Int}(G^\vee)$-orbits of semisimple elements of $G^\vee\theta$ by the following canonical bijection (cf.~\cite[Proposition 6.7]{Bo})
\begin{equation}\label{eq:Borelss}
\beta: \mathrm{Int}(G^\vee) \backslash (G^\vee\theta)_{ss}\stackrel{\sim}{\to} W_0\backslash T.
\end{equation}
Given a discrete unramified LLP $\varphi$ we put  
$$
\chi=\beta \big( G^\vee\cdot \varphi(\Frob, (\begin{smallmatrix} \mbf{v} &0 \\0 & \mbf{v}^{-1} \end{smallmatrix})) \big)
=W_0\respt\in W_0\backslash T.
$$ 
By \eqref{eq:phi(Frob)}, we can decompose $\respt$ as $\respt=cs$, where $s$ is an isolated 
torsion element of $G^\vee\theta$. 
We have:
\begin{prop}\label{cor:iso}
\begin{enumerate} 
\item[(i)] The above map $[\varphi]\mapsto\chi$ defines a 
canonical bijection between the set of $W_0$-orbits $\chi=W_0\respt$ of residual points of 
$\ha^{\IM}$ in $T$, and the set of equivalence classes $[\varphi]$ of unramified discrete 
Langlands parameters for $\mbf{G}$.
\item[(ii)] The central character of a discrete series character $\delta$ of $\ha^{\IM}$ 
is an orbit $\chi\in W_0\backslash T$ of residual points for $\ha^{\IM}$, and every 
such orbit of residual points is the central character of a nonempty finite set of discrete series 
characters of $\ha^{\IM}$.
\item[(iii)] Every equivalence class of discrete unramified Langlands parameters  
$[\varphi]$ for $\mbf{G}$ is the image of such an equivalence class $[\varphi_{ad}]$  
for $\mbf{G}_{ad}$ under the canonical isogeny. In this way, $[\varphi]$ can be identified 
the the orbit $I_G[\varphi_{ad}]$, where $I_G$ denotes the kernel of the canonical surjection 
$(\Omega^\theta_{sc})^*\to (\Omega^\theta)^*$. 
\item[(iv)] The set of equivalence classes of discrete unramified Langlands parameters for $\mbf{G}_{sc}$ can be 
identified with the set of orbits $(\Omega^\theta)^*[\varphi]$ of such equivalence classes for $\mbf{G}$ under 
the natural action of the group of weakly unramified characters 
$(\Omega^\theta)^*$ of $G^{\str}$. 
\end{enumerate}
\end{prop}
\begin{proof} 
(ii) is \cite[Lemma 3.31]{Opd3}. 

(i) follows from a direct comparison of the classification of equivalence classes $[\varphi]$ of discrete unramified 
LLP (using the results of \cite{Retorsion}) and the classification of the residual points of $\ha^{\IM}$ using 
\cite{OpdSol} and Subsection \ref{sub:Kac}.

An unramfied local Langlands parameter $\varphi$ is 
\emph{discrete} if the centralizer $C_{G^\vee}(\mathrm{Im}(\varphi))$ is finite.
It follows that the 
connected centralizer $H$ of $\varphi(\Frob, (\begin{smallmatrix} 1 &0 \\0 & 1 \end{smallmatrix}))$ 
in $G^\vee$ must be semisimple because of the existence of distinguished unipotent elements in $H$. 
In the terminology of \cite{Retorsion} $\sigma:=\varphi(\Frob, (\begin{smallmatrix} 1 &0 \\0 & 1 \end{smallmatrix}))$ is an \emph{isolated} torsion element. 
By \cite[\S 3.8]{Retorsion} we may assume that $\sigma=\sigma_\iota=\exp(v_\iota)\theta$, where $v_\iota$ is a 
vertex of the alcove $C$ of $W((\calr^{\IM, \max})^\vee)$ ($v_\iota$ corresponds canonically with a 
node of the Kac diagram of $G^\vee\theta$) 
by possibly replacing $\varphi$ by an equivalent LLP.
Even though we are in a more general situation than  \cite{Retorsion} (since $G^\vee$ need 
not be of adjoint type) one can check that the arguments required to reach the 
conclusions of \cite[\S 3.6]{Retorsion} still hold. 
It follows that the component group of $G^{\sigma_\iota}$ is isomorphic to 
(in our notations) the isotropy group $\Omega_{Y,\iota}^\vee\subset \Omega_Y^\vee$ of 
$\sigma_\iota\in \overline{C}$. It follows easily that the equivalence classes 
$[\varphi]$ of discrete unramified LLP are parameterized by the set of 
$\Omega_Y^\vee$-orbits $\Omega_Y^\vee v_\iota$ of vertices of $C$, and for each such 
orbit, the set of $\Omega_{Y,\iota}^\vee$-orbits of distinguished nilpotent orbits
of $\mathfrak{g}^{\sigma_\iota}$. On the other hand, consider the classification 
of the discrete series representations of $\ha^{\IM}= (\ha^{\IM, \max})^{\Omega_Y^\vee}$ 
(see Proposition \ref{prop:OmY}) 
of \cite[Theorem 8.7]{OpdSol} and their central characters. 
We have that $\Gamma$ (in the notation of \cite[Section 8]{OpdSol}) is equal 
to $\Gamma=\Omega_Y^\vee$. 
By Proposition \ref{prop:m} we see that the 
$\Omega_{Y,\iota}^\vee$-extended graded affine Hecke algebras $\mbf{H}_{\iota}$ in 
\cite[Theorem 8.7]{OpdSol} has underlying root system equal to that of $\mathfrak{g}^{\sigma_\iota}$, 
and \emph{equal parameters} $k=2\log(q)$. Then \cite[Appendix B]{Opd3} implies that 
the orbits of linear residual points for $\mbf{H}_{\iota}$ are the $\Omega_{Y,\iota}^\vee$-orbits 
of $W_{0,\iota}$-orbits of the weighted Dynkin diagrams of the distinguished nilpotent 
orbits of $\mathfrak{g}^{\sigma_\iota}$, multiplied by $k$. 
By these explicit descriptions we see that above map  $[\varphi] \mapsto\chi$ is a bijection, using 
that the distinguished nilpotent orbits of  $\mathfrak{g}^{\sigma_\iota}$
are classified by the $W_\iota$-orbits of the weighted Dynkin diagrams of $\mathfrak{g}^{\sigma_\iota}$.  

Now (iii) and (iv) follow easily, using (i), the fact that the map $[\varphi] \mapsto \chi$ is 
$(\Omega^\theta)^*$-equivariant, and using the covering map $T_{sc}\to T$ (whose kernel 
is naturally isomorphic to $I_G$).
\end{proof}
\begin{cor}\label{cor:LLP}
We may (an will) assign an equivalence class $[\varphi_\pi]$ of 
discrete unramified LLP to an Iwahori--spherical discrete series representation $\pi$ of $G^F$ 
by the condition that the image $\chi_\pi$ under the bijection of Proposition \ref{cor:iso}(i) 
is the central character of the discrete series character $\delta_\pi$ of $\ha^{\IM}$ corresponding 
to $\pi$.
\end{cor}
\subsection{Discrete unramified Langlands parameters for classical groups}\label{sec:DULLPCG}

We keep the notations in the previous subsections, where in this section $\mbf{G}$ will be an unramified classical group. More precisely, we will assume that $\mbf{G}$ is either $\mbf{SL}_{l+1}$, or otherwise 
a classical group such that the root datum root datum $\calr^{\IM}=(X,R_0,F_0,Y,R_0^\vee,F_0^\vee)$ of $\ha^{\IM}$
is of classical type with $X=\ZZ^l$ for some $l$. In this situation $T=T^l=(\CC^\times)^l$, with its natural 
coordinates $x_1,\dots,x_l$ given by the standard basis of $X$. 

Let $s\theta$ be a semisimple element of $G^\vee \theta$. 
By \cite[Lemma 3.2]{Retorsion} we may and will take $s\in (S^{\vee})^\theta$.
Consider its image $t\in T=T^l$, which is completely determined by its 
coordinates $t=(t_1,\dots,t_l)$. The $G^\vee$-orbit of $s\theta$ 
is completely determined by $Wt\subset T$, by Borel's well known result 
(\ref{eq:Borelss}).
Suppose that $s\theta$ is an isolated 
torsion element. Then, possibly after replacing $s\theta$ by an element 
conjugate to $s\theta$ under the conjugation action of $G^\vee$, we may 
assume that $s\theta=s_\iota\theta$, where $\iota$ corresponds to a node 
of the Kac diagram. Let $t_\iota\in T$ denote its image in $T$.  By definition 
it is clear that $\gamma_{\iota'}(t_\iota)=1$, where $\gamma_{\iota'}$ 
runs over the affine simple roots of $R_m^{(1)}$ other than $\gamma_\iota$, 
and $\gamma_\iota(t_\iota)=\zeta_\iota$  with $\zeta_\iota$ is a primitive 
$c_\iota$-th root of unity (using the notation of \cite[\S 3.3]{Retorsion}). 
Setting aside the trivial type $\textup{A}_l$-case, we see from 
\cite[Table 1]{Retorsion}) that $\zeta_\iota=\pm 1$. 
It follows easily that we may take $t_\iota$ of the form 
$t_\iota=(-1,\dots,-1,1,\dots,1)$. All the $G^\vee$-orbits of isolated 
torsion points have a unique representation of this form (but not necessarily 
all points $t\in T^l$ of this form represent a $G^\vee$-orbit of isolated 
torsion points. For example for $\mbf{SO}_{2l+1}$ there are $l+1$ 
orbits of isolated torsion points, but for $\mbf{SO}_{2l}$ there are $l-1$ 
such orbits). Finally for the type $\textup{A}$-case $\mbf{SL}_{l+1}$ the only orbit of 
isolated torsion elements is $\{\textup{Id}\}$. 

We have seen in the proof of Proposition 
\ref{cor:iso}(i) that the element $s\theta=\varphi(\Frob, \mathrm{id})\in G^\vee\theta$ 
has to be an isolated torsion element if $\varphi$ is a discrete unramified LLP of $G^F$. 
Consequently, $s$ corresponds to a ($\Omega_Y^\vee$-orbit of) node(s) in the Kac diagram of $G^\vee \theta$.
If $R_m$ is of type $\textup{B}_l$ or $\textup{D}_l$ then $\Omega_Y^\vee=C_2$, with the nontrivial 
element $\omega$ acting on $T^l$ as $\omega(x_1,\dots,x_l)=(x_1^{-1},x_2,\dots,x_l)$ 
in the type $\textup{B}_l$-case, and $\omega(x_1,\dots,x_l)=(x_1^{-1},x_2,\dots,x_l^{-1})$ in the 
type $\textup{D}_l$-case. If $R_m$ is of type $\textup{C}_l$ then $X=P$ and so $\Omega_Y^\vee=1$.

In the classical cases other than type $\textup{A}$ we may assume that 
$\varphi(\Frob, \mathrm{id})=s_\varphi\theta\in (S^\vee)^\theta\theta$ as above,   
corresponding to $t_\iota=(-1,\dots,-1,1,\dots,1)\in T$. Then 
$H$ is an almost direct product of at most two 
classical groups $H_-, \,H_+$ whose root system is the complement of the node corresponding 
to $s_\iota$ in the Kac diagram of $G^\vee\theta$. 
The equivalence classes of discrete unramified Langlands parameters $[\varphi]$ 
corresponding to the discrete torsion point $s_\iota\theta$ are in canonical bijection with 
the set of distinguished nilpotent orbits in the maximal semisimple subalgebra 
$\mathrm{Lie}(H):=\frak{h}=\frak{h}_- \oplus \frak{h}_+$ of 
$\frak{g}^\vee:=\mathrm{Lie}(G^\vee)$ 
(since in fact $\Omega_Y^\vee$ acts trivially on this set). \\

For  $\mbf{G}=\mbf{SL}_{l+1}$ the only isolated torsion element of $G^\vee$ is $s=1$. 
Thus $\frak{h} =\frak{g}^\vee= \frak{sl}_{n+1}$.
The classes of discrete unramified parameters $\varphi$ therefore correspond to
the distinguished unipotent classes $u$ of $\frak{sl}_{n+1}$. In this case only the regular 
unipotent orbit is distinguished, and the corresponding central character $\chi_\varphi=W_0\respt_\varphi$  
has the form $\respt_\varphi=c=(q^{-n/2}, q^{(2-n)/2}, \ldots, q^{(n-2)/2}, q^{n/2})$.

For a classical Lie algebra $\frak{h}$ other than type $\AAA$, the set of distinguished nilpotent orbits of $\frak{h}$ is in bijection with the set of partitions $\lambda\vdash N$ with distinct parts where  
$N$ is the dimension of the standard representation of $\frak{h}$ (cf.~\cite{OpdSol}). Moreover, all parts of $\lambda$ have the 
same parity which depends on the type of $\frak{h}$ (see below). 
In view of the above we conclude that, for classical groups other than type $\AAA$,  
an equivalence class $[\varphi]$ of discrete unramified local Langlands parameters  
is given by a pair of partitions $(\lambda_-, \lambda_+)$, where $\lambda_\pm$ is 
a partition with even, distinct parts of $2n_\pm$ if $\frak{h}_\pm=\frak{sp}_{2n_\pm}$, while 
$\lambda_\pm$ is a partition with odd, distinct parts of $N_\pm$ if $\frak{h}_\pm=\frak{so}_{N_\pm}$. 
\begin{definition} We define a ``\emph{parity}'' $\delta_\pm\in\{0,1/2,1\}$ depending on the type of $\frak{h}_\pm$ as follows: 
If $\frak{h}_\pm$ is symplectic (i.e.~if $\lambda_\pm$ has even parts) we put $\delta_\pm=1/2$. 
If $\frak{h}_\pm$ is orthogonal (hence $\lambda_\pm$ has odd parts) then 
we define $\delta_\pm\in\{0,1\}$ as the number of parts of $\lambda_\pm$ 
(or equivalently, of the sum $|\lambda_\pm|=N_\pm$) modulo $2$.  
\end{definition}

Suppose that the equivalence class of a discrete unramified Langlands parameter $\varphi$ of $G^\str$ is given 
by a pair of partitions $(\lambda_-,\lambda_+)$ as above. As explained above, 
$[\varphi]$ is also determined by the corresponding central character 
$\chi_\varphi:=W_0\respt_\varphi\in W_0\backslash T$ with $\respt_\varphi=c_\varphi s_\varphi\in T$. 
In order to compute the adjoint gamma factor $\gamma(\varphi,q)$ up to nonzero rational constant factors 
it is in fact enough to determine $\alpha(\respt_\varphi)$ for all roots $\alpha\in R_0$. 
We will now explain how to compute these values in terms of the pair $(\lambda_-,\lambda_+)$.
The isolated torsion element $s\theta=s_\varphi\theta$ corresponds to a $\Omega_Y^\vee$ orbit 
$\iota$ of vertices of the fundamental alcove $\overline{C}_\theta$ in the vector space $ \textup{Lie}(T)$, 
mapping to an element $t_\iota\in T=T^l$ of the form $t_\iota=(-1,\dots,-1,1,\dots,1)$. 

The ``infinitesimally real" part $c = c_\varphi\in T_v$ (with $T_v$ the vector subgroup $\mathbb{R}_+^l$ 
is determined as follows. 
Write $c=(c_-,c_+)\in \mathbb{R}_+^{l_-}\times \mathbb{R}_+^{l_+}$. 
Then $c_\pm$ is a real residual point denoted by 
\begin{equation}\label{eq:resptreal}
c_\pm=\respt_{\fb,\delta_\pm;\lambda_\pm} (\textrm{or\ sometimes\ simply\ } 
c_\pm=\respt_{\lambda_\pm} \textrm{\ if\ the\ values\ of\ }\fb, \delta_\pm\textrm{\ are\ clear}).
\end{equation}
for the root system of $\frak{h}_\pm$ with equal ``base" parameter 
$q^{\fb}$ (where $\fb\in\{1,2\}$ can be read off from the parameters of $\ha^{\IM}$ (see the list in this section and also the text below Remark \ref{rem:abandm})
and with $\delta_\pm$ the parity defined above.

The coordinates of $\respt_{\fb,\delta_\pm;\lambda_\pm} \in \mathbb{R}_+^{l_\pm}$ 
are of the form $q^{x\fb}$, with  
$x\in \ZZ_{\geq 0}$ (resp.~$x\in \ZZ_{\geq 0} + (1/2)$) when $\delta_\pm=1/2$ (resp.~$\delta_\pm\in\{0,1\}$)
(see \cite[\S 6]{OpdSol}). 
The $W_0$-orbit of $\respt_{\fb,\delta_\pm;\lambda_\pm}$ is therefore completely determined by 
the multiplicities $h_\pm(x)$ of $q^{x\fb}$ as coordinates of $\respt_{\fb,\delta_\pm;\lambda_\pm}$. 
\\

These multiplicities satisfy $h_\pm(x)=h_\pm(x+1)$ or $h_\pm(x)=h_\pm(x+1)+1$. (cf.~\cite{HO,OpdSol}) 
In the second case, we call $x$ a ``jump''. 
When $x>0$, the sequence of jumps is given by the sequence $(\lambda_\pm - \mbf{1})/2$ (see Section \ref{e-s}). Here $\lambda$ is the partition of the distinguished nilpotent element in $\frak{h}_\pm$ corresponding to the Langlands parameter $\varphi$ determined by $\respt$, 
presented as an increasing sequence of integers. If $\delta_\pm=1/2$ then 
$h_\pm$ is supported on $\ZZ_{\geq 0} + (1/2)$, and is determined by the above. 
If $\delta\in\{0,1\}$ then 
$h_\pm$ is supported on $\ZZ_{\geq 0}$. For $x>0$ the value $h_\pm(x)$ is determined 
as above, and we complete the determination of $h_\pm$ by  the rules  
 $h(0) = \lfloor h(1)/2 \rfloor$ if $\delta_\pm=1$ and 
 $h(0)=\lfloor (h(1)+1)/2 \rfloor$ if $\delta_\pm=0$. \\
  
We arrive at the following list of discrete unramified local Langlands parameters for the classical cases: 
\begin{enumerate}
\item[(1)] For $\mbf{G=SL}_{l+1}$, then $s=1$ and $\frak{h}=\frak{sl}_{n+1}$. Thus 
$\lambda = [l+1]$ has only one part, and 
$c=(q^{-n/2}, q^{(2-n)/2}, \ldots, q^{(n-2)/2}, q^{n/2})$; We have $\fb=1$, and $\ha^{\IM}$ is of type $\AAA_l[q]$.
\item[(2)] For $\mbf{G=SU}_{n}$, we distinguish two cases. If $n=2l+1$ is odd, then $\frak{h}_+ = \frak{so}_{2l_++1}$ 
and $\frak{h}_-=\frak{sp}_{2l_-}$ with $l_-+l_+=l$. 
If $n=2l$ is even, then $\frak{h}_+ = \frak{sp}_{2l_+}$ and $\frak{h}_-=\frak{so}_{2l_-}$
with $l_-+l_+=l$. In the first case $\lambda_+ \vdash 2l_++1$ is a partition with odd, distinct parts and $\lambda_- \vdash 2l_-$ is a partition with even, distinct parts. In the second case $\lambda_+ \vdash 2l_+$ is a partition with even, distinct parts and $\lambda_- \vdash 2l_-$ is a partition with odd, distinct parts. In both cases we have $|\lambda_-|+|\lambda_+|=n$.
Finally the base parameter equals $\fb=2$, and $\ha^{\IM}$ is of type $\CCC_l(1/2,1)[q^2]$ (first case) or 
$\BBB_l(1,1/2)[q^2]$ (second case) (notations as in \cite[3.2.1]{Opdl}). We note that $\BBB_l(1,1/2)[q^2]$
spectrally covers $\CCC_l(1/2,0)[q^2]$ (\cite[7.1.3--7.1.4]{Opdl}).
\item[(3)] For $\mbf{G=SO}_{2l+1}$, $\frak{h}_\pm$ are both symplectic, hence $\lambda_{\pm}$ both have even, distinct parts, and $|\lambda_-| + |\lambda_+| = 2l$. We have $\fb=1$, and $\ha^{\IM}$ is of type $\CCC_l(1/2,1/2)[q]$.
\item[(4)] For $\mbf{G=Sp}_{2l}$, one of $\frak{h}_\pm$ is of odd special orthogonal type, and the other one is of even special orthogonal type. So $\lambda_{\pm}$ both have odd, distinct parts, and $|\lambda_-| + |\lambda_+| = 2l+1$.
We have $\fb=1$, and $\ha^{\IM}$ is of type $\BBB_l[q]$, which spectrally covers $\CCC_l(0,1)[q]$.
\item[(5)] For $\mbf{G=SO}_{2l}$ or $\mbf{SO^\ast}_{2l}$.  
In the first case $\frak{h}_\pm$ are both of even special orthogonal type, while in the second case 
they are both of odd special orthogonal type. In both cases $\lambda_\pm$ 
have odd, distinct parts such that $|\lambda_-| + |\lambda_+| = 2l$. We have $\fb=1$, 
and $\ha^{\IM}$ is of type $\DDD_l[q]$ (first case) or $\CCC_l(1,1)[q]$ (second case).
We note that $\DDD_l[q]$ spectrally covers $\CCC_l(0,0)[q]$ (\cite[7.1.4]{Opdl}).
\end{enumerate}

\subsection{Application of spectral transfer maps}\label{sub:muSTM} 
Given a supercuspidal unipotent representation $\pi$ of $\mbf{G}$ we want to determine 
a discrete unramified LLP $\varphi_\pi$ of $\mbf{G}$  
such that $\gamma(\varphi_\pi, q) = C \fdeg(\pi, q)$ for some $C \in \QQ^\times$. In this section, we will explain how 
$\fdeg(\pi, q)$ and $\gamma(\varphi_\pi, q)$ are related to residues of 
$\mu$-functions of affine Hecke algebras, and how spectral transfer morphisms  (STMs) 
of affine Hecke algebras are used to simplify this task. 

For the definition, properties and examples of STMs, the reader is referred to \cite{Opds}. 
In \cite{Opdl}, the role played by STMs towards establishing a connection between formal degrees 
of unipotent discrete series representations of $\mbf{G}^u$ and adjoint gamma factors 
for discrete unramified Langlands parameters of $\mbf{G}$ is explained in detail. 
Three kinds of spectral transfer morphisms are relevant in this section, namely, spectral covering, 
translational and extra--special STMs. \\

Firstly, for classical cases other than type $\AAA$, there always exists a \emph{spectral covering} map 
(i.e. an STM between affine Hecke algebras of equal rank, which is a particularly simple instance of STMs) 
from the Iwahori--Hecke algebra $\ha^{\IM}$ of the adjoint group isogenous to $G^F$, to a unipotent affine Hecke algebra of the form 
$\CCC_l(\delta_-,\delta_+)[q^\fb]$ (see \cite[3.2.6, 3.2.7]{Opdl}). 
By the spectral correspondence result \cite[Theorem 6.1]{Opds} 
this implies that we can compute $\gamma(\varphi,q)$ as the residue of the 
$\mu$-function $\mu^l_{\delta_-, \delta_+}$ of $\CCC_l(\delta_-,\delta_+)[q^\fb]$ at the image $(-\respt_{\lambda_-}, \respt_{\lambda_+})$ of $\respt_\varphi$ under this spectral covering map. Indeed, by \eqref{eq:residue} and \cite[Theorem 6.1]{Opds} we have 
\begin{equation}\label{eq:gammamu}
\gamma(\varphi,q)=C\mu_{\delta_-,\delta_+}^{l;\{(-\respt_{\lambda_-},\respt_{\lambda_+})\}}(-\respt_{\lambda_-},\respt_{\lambda_+})
\end{equation}
for some $C \in \QQ^\times$, and $\respt_{\lambda_\pm}=\respt_{\fb,\delta_\pm;\lambda_\pm}$ is a residual point with positive coordinates as described in \eqref{eq:resptreal}. The relevant parameters $(\delta_-,\delta_+;\fb)$ have been given with the discrete unramified LLP in Section \ref{sec:DULLPCG}. 

Secondly, we look at the translational STM. Let $\CCC_d(m_-,m_+)[q^{\fb'}]$ be a unipotent affine Hecke algebra with parameter $(\emm, \emp; \fb')$. A translational spectral transfer morphism 
\begin{equation}\label{eq:stm}
\psi:\CCC_d(m_-,m_+)[q^{\fb'}]\leadsto \CCC_l(\delta_-,\delta_+)[q^\fb]
\end{equation}
is given by a morphism of the underlying algebraic tori $\psi_T: T^d\to T^l$ ($d \leq l$) with image\footnote{The image is \textit{a priori} a residual coset, cf.~\cite[\S 5.1]{Opds}.} $\textup{Im}(\psi_T)=\mbf{L}\subset T^l$. Suppose $(-\respt_{\lambda_-},\respt_{\lambda_+})\in\mbf{L}$. 
Then by \cite[Proposition 5.2, Theorem 6.1]{Opds} there exists a residual point 
$(-\respt_{\rho_-},\respt_{\rho_+})\in T^d$ for $\CCC_d(m_-,m_+)[q^{\fb'}]$ such that 
\begin{equation}\label{eq:respt}
\psi_T(-\respt_{\rho_-},\respt_{\rho_+})=
(-\respt_{\lambda_-},\respt_{\lambda_+}).
\end{equation}
In this situation there exists \cite[Theorem 6.1]{Opds} a constant $C' \in \QQ^\times$ such that:
\begin{equation}\label{eq:resSTM}
\mu_{\delta_-,\delta_+}^{l;\{(-\respt_{\lambda_-},\respt_{\lambda_+})\}}(-\respt_{\lambda_-},\respt_{\lambda_+})
=C'\mu_{m_-,m_+}^{d;\{(-\respt_{\rho_-},\respt_{\rho_+})\}}(-\respt_{\rho_-},\respt_{\rho_+})
\end{equation}
Combining with \eqref{eq:gammamu} we thus obtain, for some nonzero rational constant factor $C$, that
(after choosing $\psi_T$ appropriately in its equivalence class, cf.~\cite[Definition 5.9]{Opds}): 
\begin{equation}\label{eq:adjresmu}
\gamma(\varphi,q)=C\mu_{m_-,m_+}^{d;\{(-\respt_{\rho_-},\respt_{\rho_+})\}}(-\respt_{\rho_-},\respt_{\rho_+});
\textrm{\ with\ } 
\psi_T(-\respt_{\rho_-},\respt_{\rho_+})=(-\respt_{\lambda_-},\respt_{\lambda_+})=\psi_{0,T}(\respt_\varphi)
\end{equation}

An important special case of \eqref{eq:stm} is when $d=0$. In this case, $\CCC_0(m_-,m_+)[q^{\fb'}]$ is a 
direct summand of the 
the (rank $0$) Hecke 
algebra of a maximal cuspidal \textit{type} $(\PH^{\str},\sigma)$ of $G^{\str}$ 
corresponding to a maximal proper $u\theta$-stable subset $\mathsf{J}_{a,b}\subset \tisfI$ as in Section \ref{sub:fdsup}. 
The unique residual point of this rank $0$ Hecke algebra is denoted by  
$(-\respt_{\rho_-},\respt_{\rho_+})$, and its image 
$\psi_T(-\respt_{\rho_-},\respt_{\rho_+})=\mbf{L}=:\{(-\respt_{\lambda_-},\respt_{\lambda_+}) \}$ is a residual point 
in $T^l$. 
In this case, \eqref{eq:adjresmu} reduces to: 
\begin{equation}\label{eq:deggamma}
\gamma(\varphi,q)=C\textup{fdeg}(\pi_\chi) 
\end{equation}
where $\pi_\chi$ is any element of the finite set of irreducible cusipdal unipotent characters associated with $\mathsf{J}_{a,b}\subset \tisfI$.


\begin{rem}\label{rem:abandm}
The relation between the maximal subset $\mathsf{J}_{a,b}\subset \tisfI$ and the parameters 
$(m_-,m_+;\fb)$ is given by equation (34) (which we recover as \eqref{eq:sets} below) and the list of STMs in \cite[3.2.6]{Opdl}. Our main result,
Theorem \ref{thm:A} states that equation \eqref{eq:deggamma} for a given maximal subset 
$\mathsf{J}_{a,b}$ determines $(\lambda_-,\lambda_+)$ (up to 
obvious symmetry in case $\delta_-=\delta_+$). 
\end{rem}

It is thus useful to give the parameters $m_\pm$ and $\fb'$ explicitly. It is known \cite[(32)]{Opdl} that $m_\pm\in(\ZZ/4)_{\geq 0}$ and $\fb' \in\{1,2\}$ satisfy 
$\emp \pm \emm \in\ZZ/2$, and $\fb'=1$ if and only if both $\emp \pm \emm \in\ZZ$. 

The parameters $\{\emm, \emp\} $\footnote{We view $\{x,y\}$ as a multiset, so we still consider $\{x,y\}$ as a pair even if $x=y$.} by the following rules are determined in each case (i) to (vi) in Section \ref{sub:fdsup} by the following rules described in \cite[(34)]{Opdl}:
\begin{equation}\label{eq:sets}
\{ |\emp - \emm|, \emp+ \emm\} = \begin{cases}
\{ 1/2+a, 1/2+b\}  &\text{case (i)};\\
\{ 2a, 1+2b\}  &\text{case (ii)};\\
\{ 1+2a, 1+2b\} &\text{case (iii)};\\
\{ 2a, 2b\} &\text{case (iv)};\\
\{ 1/2+a, 1+2b\} & \text{case (v)}; \\
\{ 1/2+a, 2b\} &\text{case (vi)}.
\end{cases}
\end{equation}
This results in a partition of the set $\mathcal{V}$ of relevant parameters $(m_-,m_+)$, making it as a disjoint union of six subsets $\mathcal{V}^X$ with 
$\mathcal{X} \in \{\text{I, II, III, IV, V, VI}\}$, according to the cases (i) to (vi) in Section \ref{sub:fdsup}.\\

In order to define the corresponding $\mu$-function $\mu^l_{m_-, m_+}$ on $T^l$ as in \cite[Definition 3.2]{Opds}, we need to know the parameters $m_{\pm}(\alpha)$ for $\alpha\in R_0$, and we need to know 
the normalization factor of the $\mu$-function. As to the former, 
we note that here $R_0 \subset X^*(T^l)$ is of type $\BBB_l$, whose positive roots are $t_i$ (with $1\leq i\leq l$) 
and $t_it_j^{\pm 1}$ (with $1\leq i<j\leq l$). We thus define the $m_{\pm}(\alpha)$ by: 
\begin{equation}\label{eq:pars}
m_-(t_it_j^{\pm 1})=0, \, m_+(t_it_j^{\pm 1})=\fb'; \,m_-(t_i)=\fb' m_-, \,m_+(t_i)=\fb' m_+.
\end{equation}
The normalization factor of $\mu_{\emm, \emp}^l$ is given by 
$\tau_{m_-,m_+}(1):=(v^{\fb'}-v^{-\fb'})^{-l} d^\tau_{m_-,m_+}$ 
with 
$d^\tau_{m_-,m_+}$ given by \cite[(33)]{Opdl} (when the rank $l=0$ 
this equals $\fdeg(\pi_\chi,q)$ as defined by \ref{eq:classicalfdeg}). 
Observe that this is well-defined because $\{m_-,m_+\}$ determines $a$ and $b$ 
in the cases $\text{II, V}$ and $\text{VI}$, and determines $a$ and $b$ up to order in the other cases.\\

To introduce the extra-special STM we introduce some further notations for the cases $\text{V}$ and $\text{VI}$: 
if $m_{\pm} \in \ZZ \pm (1/4)$ and $m_{\pm} >0$ we write 
\begin{equation}\label{eq:esm}
m_{\pm} = \kappa_{\pm} + \frac{2\epsilon_{\pm} -1}{4}
\end{equation}
with $\epsilon_{\pm} \in \{0, 1\}$ and $\kappa_{\pm} \in \ZZ_{\geq 0}$. We define 
$\delta_{\pm} \in \{0, 1\}$ by $\kappa_{\pm}\equiv \delta_{\pm} \pmod{2}$. 
Observe that in case $\text{V}$ we have $\delta_-\not=\delta_+$ while in case $\text{VI}$ 
we have $\delta_-=\delta_+$.

Equation \eqref{eq:adjresmu} enables us to rewrite $\gamma(\varphi,q)$ in other ways, 
useful for the analysis of the difficult cases $\delta_\pm\in\{0,1\}$. We will show in Section \ref{e-s} that for every partition $\lambda\vdash 2l+\delta$ (with $\delta\in\{0,1\}$) with odd, distinct parts, 
there exists $m\in (\ZZ\pm 1/4)_+$, $d\in\ZZ_{\geq 0}$ and 
a spectral transfer map (a so-called \emph{extra-special} STM) 
\begin{equation}
\xi: \CCC_d(1/4,m)[q^2] \leadsto \CCC_l(0,\delta)[q]
\end{equation}
such that the (infinitesimally) real residual point $\respt_{\lambda}:=\respt_{1,\delta;\lambda}\in T^l_v$ is in the image 
of $\xi_T$ (where $T^l_v$ denotes the real vector subgroup of $T^l$). 
Therefore \cite[Proposition 5.2]{Opds} we can find a real residual point $r\in T^d_v$ of $\CCC_d(1/4,m)[q^2]$ such that $\xi_T(r)=\respt_{1,\delta;\lambda}$. The real residual points 
of Hecke algebras with generic parameters such as $\CCC_d(1/4,m)[q^2]$ are parameterized by partitions $\rho\vdash d$, where the coordinates of $r=\respt_{2,m;\rho}$ are 
simply given by $q^{2c(x)}$ with $c(x)$ running over contents of the $m$-tableau of shape 
$\rho$ (cf.~\cite{HO, OpdSol}).
In other words, for every odd, distinct partition $\lambda$ of $2l+\delta$ we can 
find a pair $(m,\rho)$ as above, with $\rho\vdash d$, and an extra-special STM 
such that 
\begin{equation}
\xi_T(\respt_{2,m;\rho})=\respt_{1,\delta;\lambda}. 
\end{equation} 
In particular, by \eqref{eq:resSTM}, we have the relation (for some nonzero rational constant $C$):
\begin{equation}\label{eq:es-STM}
\mu_{0,\delta}^{l;\{\respt_{1,\delta;\lambda}\}}
(\respt_{1,\delta;\lambda})
=C\mu_{1/4,m}^{d;\{\respt_{2,m;\rho}\}}(\respt_{2,m;\rho})
\end{equation}
The point of all this is that the left-hand side of \eqref{eq:es-STM} is, up to 
a nonzero rational constant, the adjoint gamma factor of an unramified Langlands 
parameter $\varphi$ (with $s_\varphi=1$), while the right-hand side is the residue 
at a residual point of the $\mu$-function of a Hecke algebra with generic parameters. 
The latter expression exhibits much less complicated cancellations, and it is on this 
side where we will analyze the combinatorial implications of the property
observed in Corollary \ref{cor:sc}.

The bijection $\lambda\leftrightarrow (m,\rho)$ between the set of odd distinct partitions, and 
the set of pairs $(m,\rho)$ with $m\in (\ZZ\pm 1/4)_+$ and $\rho$ a partition is made explicit 
in Section \ref{e-s}.

\subsection{The main theorem}\label{sec:mainthm}

We are now ready to state our main theorem. It generalizes the version in the Introduction in two aspects: the group is not necessarily simply connected, nor $K$-quasi-split. To be precise, let $\mbf{G}$ be isogenous to a classical group over $K$ which splits over an unramified extension.  
Let $\omega=[u]\in H^1(F, G_{ad})$, and denote $\mbf{G}^u$ the corresponding inner form of $\mbf{G}$. Given a maximal parahoric subgroup $\PH$ of $G$ such that $F_u(\PH)=\PH$, let $\sigma$ be a cuspidal unipotent representation of $\PH^{\str}$. We have a compactly induced representation $\cInd_{\PH^{\str}}^{G^{\str}} \sigma$. Let $\pi$ be an irreducible summand of this induced representation, with formal degree $\fdeg(\pi, q)$ given by \eqref{eq:classicalfdeg}.  
Our main Theorem is:
\begin{thm}\label{thm:A}
Let $\pi$ be a supercuspidal unipotent representation of the group $G^{\str}$ which occurs as an irreducible summand of $\cInd_{\PH^{\str}}^{G^{\str}} \sigma$.
There \emph{exists a unique} $(\Omega^\theta)^*$-orbit 
$(\Omega^\theta)^*[\varphi]$ of equivalence classes of discrete unramified local Langlands parameters 
$\varphi\in \Phi_{ur}^2(G^{\str})$ such that, for some nonzero rational constant $C$ (depending only on the centralizer of the image of $\phi$ and not on $K$),  
one has  
\begin{equation}\label{eq:fdegisgammafactor}
\fdeg(\pi, q)=C\gamma(\varphi, q)
\end{equation}
as rational functions in $q$ with $\QQ$-coefficients. 
\end{thm} 
We remark that if the discrete unramified Langlands parameter 
$\varphi$ as in Theorem \ref{thm:A} corresponds to the pair of unipotent partitions $(\lambda_-,\lambda_+)$ 
as in \S \ref{sec:DULLPCG}, then \eqref{eq:gammamu} implies that \eqref{eq:fdegisgammafactor} 
is equivalent to the following equality
\begin{equation}\label{eq:fdegmu}
\fdeg(\pi, q)=C'\mu_{\delta_-,\delta_+}^{l;\{(-\respt_{\lambda_-},\respt_{\lambda_+})\}}(-\respt_{\lambda_-},\respt_{\lambda_+})
\end{equation}
where $C'$ is a nonzero rational constant.

The point of the present paper is 
that \eqref{eq:fdegisgammafactor} completely characterizes $(\Omega^\theta)^*[\varphi]$. 
One immediate consequence of this characterization is that one can associate a discrete unramified LLP to a 
cuspidal unipotent character using this characterization for groups of arbitrary isogeny type. We formulate 
this for arbitrary almost simple groups (which is allowed since the necessary results for exceptional 
groups are known \cite{Re, Fe2}): 
\begin{cor}
Let $\mbf{G}$ be a connected absolutely almost simple group over $K$ which splits over $K_{ur}$, 
and let $\pi$ be a cuspidal unipotent representation of $G^{\str}$. The set of solutions $[\varphi]$ 
of \eqref{eq:fdegisgammafactor} is a unique $(\Omega^\theta)^*$-orbit of equivalence classes of 
unramified local Langlands parameters which is canonically associated with $\pi$.  
\end{cor}

\subsection{The list of solutions}

For the cuspidal unipotent representations as described in Section \ref{sub:fdsup}, we list below the classes of parameters $[\varphi]$ which are solutions of \eqref{eq:fdegisgammafactor} for the various classical groups (other than the easy type $\AAA$ case) explicitly. By Corollary \ref{cor:iso}, for each classical type (other than the easy case of type $\AAA$) it is enough to consider one 
specific group in the isogeny class of this type. Therefore, it is reduced to consider the associated affine Hecke algebra of the form $\CCC_l(\dem, \dep)[q^{\fb}]$ as described in Section \ref{sec:DULLPCG}. We remark that for the adjoint form of a classical 
group this unique  (up to the action of $(\Omega^\theta)^*$) parameter $\varphi$ 
is exactly the parameter which Lusztig \cite{Lusztig-unirep} associated 
somehow with the supercuspidal unipotent representation $\pi$. These are the equivalence classes of unramified 
Langlands parameters which support a cuspidal local system. We also give the corresponding $\omega \in H^1(F, G_{ad})\simeq \Omega_{ad}/(1-\theta)\Omega_{ad}$, 
indicating the class of the inner form $\mbf{G}$ of which $\pi$ is a representation, in terms of the parameter $\varphi$. 
The partitions are expressed as an increasing sequence of nonzero numbers.\\

(i) Let $\mbf{G}$ be of type $\widetilde{{}^2\AAA_n}$. 
The group $\Omega_{ad}/(1-\theta)\Omega_{ad}$ is trivial if $n=2l$ is even, and of order two if $n=2l-1$ is odd, 
and let $\omega\in \Omega_{ad}/(1-\theta)\Omega_{ad}$.
Let $\pi$ be a cuspidal 
unipotent representation of $G^{\str}$ 
associated with a subset of the form $\mathsf{J}_{a,b}\subset \tisfI$ as in Section \ref{sub:fdsup}.
Put $s=(a/2)(a+1)-1, t=(b/2)(b+1)-1$, then $s+t+2=n+1$ and $\omega$ is nontrivial iff 
$s$ and $t$ are both even (or equivalently, if $a,b$ are both equal to $1$ or $2$ modulo $4$).
Fix parameters $(m_-,m_+)\in\mathcal{V}^\text{I}$ such that  
$\{|m_--m_+|,m_-+m_+\}=\{a+1/2,b+1/2\}$ as in \eqref{eq:sets}, where $m_-\in\ZZ+1/2$. 
Let $\dem=1/2$ and let $\dep \in \{0, 1\}$ be the parity of $m_+$.
Then $\varphi$ corresponds to the pair of partitions $(\lambda_-,\lambda_+)$ where $\lambda_\pm$ are given by
$$
\lambda_-=[2, 4, \ldots, 2(\emm-3/2), 2(\emm-1/2)], \quad \lambda_+=[1, 3, \ldots, 2\emp-3, 2\emp-1].
$$
We have $|\lambda_-|+|\lambda_+|=n+1$. 

(ii) Let $\mbf{G}$ be of type $\widetilde{\BBB_l},\,(l \geq 2)$. 
The group $\Omega_{ad}/(1-\theta)\Omega_{ad}$ has order two, and let $\omega\in 
\Omega_{ad}/(1-\theta)\Omega_{ad}$. 
Let $\pi$ be a cuspidal 
unipotent representation of $G^{\str}$ associated with a subset of 
the form $\mathsf{J}_{a,b}\subset \tisfI$ as in Section \ref{sub:fdsup}, 
so that $l=s+t=a^2+b(b+1)$. Then $\omega$ is nontrivial iff $a$ is odd. 

Fix $(m_-,m_+)\in\mathcal{V}^\textup{II}$ such that $\{|m_--m_+|,m_-+m_+\}=\{2a,1+2b\}$ 
as in \eqref{eq:sets}. We put $\delta_\pm=1/2$. 
The pair of partitions $(\lambda_-,\lambda_+)$ is given by
$$
\lambda_\pm = [2, 4, \ldots, 2(m_{\pm} - 3/2), 2(m_{\pm} -1/2)].
$$
We have $|\lambda_-|+|\lambda_+|=2l$. 

(iii) Let $\mbf{G}$ be of type $\widetilde{\CCC_l},\, (l \geq 2)$. 
We consider here only the case $\omega=1$; for the non-split inner form 
(which is an extra-special case), see (v). 
Fix $(m_-,m_+)\in\mathcal{V}^\textup{III}$ such that $\{|m_--m_+|,m_-+m_+\}=\{1+2a,1+2b\}$
as in \eqref{eq:sets}, with $m_-$ even. We fix $\dem=0, \dep=1$. 
The partitions are given by
$$
\lambda_\pm = [1, 3, \ldots, 2m_\pm -3, 2m_\pm -1].
$$
We have $|\lambda_-|+|\lambda_+|=2l+1$.  

(iv) Let $\mbf{G}$ be of type $\widetilde{\DDD_l} \, (l \geq 4)$ or 
of type $\widetilde{{}^2\DDD_l} \, (l \geq 4)$. 
The group $\Omega_{ad}$ of 
the split group of type $\widetilde{\DDD_l}$ is isomorphic to $C_4$
(odd $l$) or $C_2\times C_2$ (even $l$), and has a nontrivial automorphism 
$\theta$ of order $2$ induced by the finite diagram automorphism associated 
with the quasi-split structure of type $\widetilde{{}^2\DDD_l}$. We write explicitly $\Omega_{ad}=\{1,\rho,\eta,\eta\rho\}$ such that $\Omega_{ad}^\theta=\{1,\eta\}$
(so $\rho^2=1$ if $l$ is even, and $\rho^2=\eta$ else). Then 
$(1-\theta)\Omega_{ad}=\Omega_{ad}^\theta$ and we denote
$\Omega_{ad}/(1-\theta)\Omega_{ad}=\{\overline{1},\overline{\rho}\}\simeq C_2$. 
In this item, we only consider those inner forms $G^{\str}$ for which 
the image in $\Omega_{ad}/(1-\theta)\Omega_{ad}$ is trivial. For the remaining 
(extra-special) cases, see (vi).

If $\mbf{G}$ is of type $\widetilde{\DDD_l}$ and $\omega=1$,   
consider $\mathsf{J}_{a, b} = \DDD_{a^2} \sqcup \DDD_{b^2}$ with $a$ and $b$ both even. 
Fix $(m_-,m_+)\in\mathcal{V}^\text{IV}$ such that
$\{|m_--m_+|,m_-+m_+\}=\{2a,2b\}$ as in \eqref{eq:sets}. 
We have $m_\pm \equiv 0 \pmod{4}$. 
If  $\omega=\eta$ then $\mathsf{J}_{a, b} = {}^2\DDD_{a^2} \sqcup {}^2\DDD_{b^2}$ 
with $a$ and $b$ both odd. We now have $m_\pm \equiv 2 \pmod{4}$. 
For the quasi-split groups of type $\widetilde{{}^2\DDD_l}\, (l \geq 4)$, we have $\Omega_{ad}/(1-\theta)\Omega_{ad} = C_2$ 
and we consider only the case $\omega=1$. In this case 
$\mathsf{J}_{a, b} = \DDD_{a^2} \sqcup {}^2\DDD_{b^2}$ with $a$ even and $b$ odd. We now 
take $m_-\equiv 2 \pmod{4}$ and $m_+\equiv 0 \pmod{4}$. 

In all these cases we have $\delta_\pm=0$, and 
the pair of partitions $(\lambda_-,\lambda_+)$ is given by
$$
\lambda_\pm = [1, 3, \ldots, 2m_\pm -3, 2m_\pm -1].
$$
We have $|\lambda_-|+|\lambda_+|=2l$.

(v) Let  $\mbf{G}$ be of type $\widetilde{\CCC_l} \,(l \geq 2)$ and let $\omega$ correspond to the non-trivial element 
$\rho\in \Omega_{ad}/(1-\theta)\Omega_{ad}$. This is an extra-special case. 
Fix $(m_-,m_+)\in\mathcal{V}^\text{V}$ such that 
$\{|m_--m_+|,m_-+m_+\}=\{a+1/2,1+2b\}$ as in \eqref{eq:sets}. 
Define $\kappa_\pm\in\ZZ_{\geq 0}$ and $\delta_\pm\in\{0,1\}$ by writing 
$m_{\pm} = \kappa_{\pm} + (2\epsilon_\pm -1)/4$, with $\kappa_\pm \equiv \delta_\pm \pmod{2}$. 
By possibly interchanging $m_-$ and $m_+$ we can choose $\kappa_-$ even and $\kappa_+$ odd, 
so that $(\dem, \dep)=(0, 1)$. 
The pair of partitions $(\lambda_-,\lambda_+)$ is given by: 
$$
\lambda_\pm = [1+2\epsilon_\pm, 5+2\epsilon_\pm, \ldots, 4(\kappa_\pm-1)+1+2\epsilon_\pm].
$$
(vi) The inner twists of $\widetilde{\DDD_l} \,(l \geq 4)$ corresponding to $\omega=\rho$ or $\rho\eta$ 
or the inner twist of $\widetilde{{}^2\DDD_l} \,(l \geq 4)$ corresponding to $\omega=\overline{\rho}$.
These cases are extra-special. 
Fix $(m_-,m_+)\in\mathcal{V}^\text{VI}$ such that 
$\{|m_--m_+|,m_-+m_+\}=\{a+1/2,2b\}$ as in \eqref{eq:sets}, and define $\kappa_\pm$ and 
$\delta_\pm\in\{0,1\}$ as in (v).
We now have $\delta_-=\delta_+$, and again the pair $(\lambda_-,\lambda_+)$ is given by:
$$
\lambda_\pm = [1+2\epsilon_\pm, 5+2\epsilon_\pm, \ldots, 4(\kappa_\pm-1)+1+2\epsilon_\pm].
$$

\section{Description of the extra-special algorithm}\label{e-s}

Let $\mathsf{P}_{\od}$ be the collection of all partitions (including the zero partition) with odd, distinct parts. 
Let  
$$
\mathsf{R} = \{(m, \rho) \mid m \in \ZZ \pm 1/4,\,m >0\, \text{ and }\rho \text{ a (possibly zero) partition}\}.
$$
We will define two operations $\E: \mathsf{P}_{\od} \to \mathsf{R}$ and $\D: \mathsf{R} \to \mathsf{P}_{\od}$ and prove they are inverse to each other. We refer to $\E$ as the \emph{extra-special algorithm}. 

We remark that the map $\D$ is similar, and indeed equivalent to the algorithm discussed in \cite[\S 4.4]{CK}. It was also shown in \cite{CK} that $\D$ is 
an injective map. \footnote{We thank Dan Ciubotaru for pointing out the reference \cite{CK} to us.}\\

The algorithm $\E: \lambda \mapsto (m, \rho)$ produces a number $m\in(\ZZ\pm 1/4)_+$, and an $m$-tableau, whose shape we call $\rho$.
The steps to produce $m$ and $T_m(\rho)$ from $\lambda\in \mathsf{P}_{\od}$ are as follows:  
\begin{itemize}
\item[(1)]  Write $\lambda$ as a non-negative integral sequence in decreasing order. 
Define $\jmath = (\lambda-{\bf 1})/2$, where $(\lambda-\mathbf{1})/2$ means substracting 1 from all nonzero parts of 
$\lambda$, and then dividing each part by 2.\\ 

We stress that we do not regard $\jmath$ as a partition, but as a tuple of non-negative integers, whose length is equal to the number of nonzero parts of $\lambda$.  

\item[(2)] Let $\kappa \geq 0$ be the excess number of parts of the dominant parity type (even or odd) of $\jmath$. Put $\epsilon =1$ if the dominant parity type is odd or if $\kappa=0$ (in which case we shall call the dominant parity type odd as well), 
otherwise put $\epsilon =0$. 
Let $m = \kappa + (2\epsilon -1)/4$. This gives us the required number $m\in(\ZZ\pm 1/4)_+$.

\item[(3)] Let $\jmath'=(\gamma_1, \ldots, \gamma_{\kappa})$ be the sub-sequence in $\jmath$ of the $\kappa$ smallest parts of dominant parity type.

\item[(4)] Removing $\jmath'$ from $\jmath$ and denote the remaining sub-sequence of $\jmath$ by $\jmath''$. Thus $\jmath''$ has an equal number parts of both parities. Arrange $\jmath''$ in $t$ pairs: 
$$
\jmath''= \big( (\alpha_1, \beta_1), \ldots, (\alpha_t, \beta_t)\big)
$$
with $\alpha_1 > \cdots > \alpha_t$ and $\beta_1 > \cdots > \beta_t$, where for all $i$, $\alpha_i$ is of dominant parity type and $\beta_i$ is of the other parity type.
  
\item[(5)]   
For every pair $(\alpha_i, \beta_i)$ we denote by $T_m(H(\alpha_i, \beta_i))$ the hook-shaped $m$-tableau whose 
\emph{hand} (the box at the end of its arm) is filled with $(\alpha_i - 1/2)/2$, and whose foot has filling $|\beta_i-1/2|/2$.\\

Note that we need to take the absolute value in the latter expression since it might happen that the smallest part of $\jmath$ is $0$. 
If $\epsilon=1$ then the part $0$ of $\jmath$ 
will appear as $\beta_t=0$ in the smallest pair $(\alpha_t,\beta_t)$ of $\jmath''$. Also, if $\epsilon=0$ then 
$\kappa>0$, and then the part $0$ of $\jmath$ will appear as the smallest part $\gamma_\kappa$ of $\jmath'$. In particular, we always have $(\alpha_i - 1/2)/2>0$ for all $i$.

Let $T_m(H)$ be the $m$-tableau obtained by nesting the hook shaped tableaux  $T_m(H(\alpha_i, \beta_i))$ 
in decreasing order.
Observe that all hooks $T_m(H(\alpha_i, \beta_i))$ contain a box with filling $m$ (namely the box at the corner) and a box
with filling $1/4$ (these two boxes coincide when $m=1/4$). We call such hooks \emph{$m$-hooks}.
Hence the leg of an $m$-hook has length at least $\kappa$, since its corner box has filling $m$ 
and the box with filling $1/4$ is precisely $\kappa$ boxes below that. 

\item[(6)] We add horizontal strips $S_i$ (which may be empty) to $T_m(H)$ (for $i=1,\dots , \kappa$).
If $\gamma_i < 2(m-i+1)+1/2$, then the strip $S_i$ is empty.
Otherwise $S_i$ is the horizontal $(m-i+1)$-tableau whose rightmost extremity has filling $|\gamma_i-1/2|/2$
(again, we need the absolute value because it might happen that $\gamma_\kappa=0$, namely if 
$\jmath$ contains $0$ as a part, and if in addition $\epsilon=0$).
These horizontal strips can be added to $T_m(H)$ in a unique way such that the union is an $m$-tableau
(so $S_1$ is placed at the ``armpit" of $T_m(H)$, and $S_{i+1}$ is placed just below $S_i$ for $i=1, \dots, \kappa-1$).\\

\noindent Observe that for all $i$ the parity type of $\gamma_i$ is $\epsilon$. The smallest possible value of $\gamma_i$ equals 
$\gamma_i=2(m-i+1)-3/2$, corresponding to $S_i$ being empty. In particular, the smallest possible value of $\jmath_{\kappa}$ is 
$\jmath_{\kappa} = \epsilon$. If all strips $S_i$ are empty, we have: 
\begin{equation*}
\jmath'=(\epsilon+2(\kappa-1), \epsilon+ 2(\kappa-2), \ldots, \epsilon)
\end{equation*}
\noindent Denote the union of the strips $S_i$ by $T_m(S)$. Then $T_m(S)$ is either an $m$-tableau or empty.
 
\item[(7)] Finally, $\rho$ is the partition whose Young diagram is the shape formed by the union of the $m$-hooks $H$ and strips 
$S_i$ in the way indicated above.
\end{itemize}

As mentioned above, 
we call the hook-shaped $m$-tableaux $T_m(H(\alpha_i, \beta_i))$ the \emph{$m$-hooks} of $T_m(\rho)$. 
Observe that the $m$-hooks of $T_m(\rho)$ are precisely the hooks in $T_m(\rho)$ which are $m$-tableaux, and contain $1/4$. 
Equivalently, a hook of $T_m(\rho)$ which is an $m$-tableau is an $m$-hook if and only if its leg length 
is at least $\kappa$.



\begin{ex} Let $\lambda = [21, 19, 13, 9, 5, 3]$ be an odd distinct partition. We have already ordered the parts of $\lambda$ in decreasing order. Then we have $\jmath= (10, 9, 6, 4, 2, 1)$. We now have 4 even numbers and 2 odd numbers. Thus the dominant parity type is even (hence $\epsilon =0$) and the excess number $\kappa=4-2=2$. We get $m = \kappa+(2\epsilon - 1)/4 = 7/4$.\\

Now $\jmath' = (4, 2)$. Removing $\jmath'$ from $\jmath$ we obtain $\jmath'' = \big ((10, 9), (6, 1) \big)$. Therefore we will have 2 hook-shaped $m$-tableaux and 2 horizontal strips. The Young tableau is as follows:

\begin{center}
\ytableausetup{mathmode, boxsize=2.2em}
\begin{ytableau}
 7/4 & 11/4 & 15/4 & 19/4 \\
 3/4 & 7/4 & 11/4 \\
 1/4 & 3/4 & {7/4}\\ 
 5/4 & 1/4 & {3/4}\\
 9/4 \\
 13/4\\
 17/4
\end{ytableau}
\end{center}

The two horizontal strips are the two right-most boxes in the third and fourth rows.
\end{ex}

We next describe the operation $\D$, namely how to recover the odd distinct partition $\lambda \in \mathsf{P}_{\od}$ from $(m, \rho) \in \mathsf{R}$.\\

Recall that $m \in (\ZZ \pm 1/4)_+$. Let $\kappa$ be the closest integer to $m$ and write $m = \kappa + (2\epsilon -1)/4$ with $\epsilon =0$ or $1$.
This uniquely determines a parity type $\epsilon$ and a nonnegative integer  $\kappa$. 
Define $\delta \in \{0, 1\}$ by $\kappa \equiv \delta \pmod{2}$. \\ 

The $m$-tableau $T_m(\rho)$ can be written as a disjoint union of nested $T_m(\rho)$-hooks which are themselves 
$m$-tableaux. If one of these hook shapes is an $m$-hook than all its predecessors are $m$-hooks too (since the 
condition for such hooks to qualify as $m$-hooks is that their leg lengths are at least $\kappa$). Hence 
$T_m(\rho)$ has a unique decomposition as the union $T_m(H) \cup T_m(S)$ of two $m$-tableaux (both possibly empty) such that  
$H$ is the largest $m$-tableau contained in $T_m(\rho)$ which is a union of $m$-hooks, 
and $S$ is the complement of $H$ in $T_m(\rho)$. By the above we see that $S$ is itself an 
$m$-tableau (or empty) without $m$-hooks, and that $S$ has at most $\kappa$ parts.  \\

We number the shapes of the the nested $m$-hooks in $T_m(\rho)$ in decreasing order as $H_1, \cdots, H_t$. 
For every $i$, $T_m(H_i)$ defines unique pair of nonnegative 
integers $(\alpha_i, \beta_i)$ such that $T_m(H_i)=T_m(H(\alpha_i, \beta_i))$. 
Indeed, if the hand of $T_m(H_i)$ has filling $A_i\in(\ZZ \pm 1/4)_+$ and its foot  
has filling  $B_i\in(\ZZ \pm 1/4)_+$, then $\alpha_i$ is the unique integer of parity type $\epsilon$ nearest to $2A_i\in(\ZZ+1/2)_+$ 
(which is easily seen to be $2A_i+1/2$), and 
$\beta_i$ is the unique integer of parity type $1-\epsilon$ nearest to $2B_i\in(\ZZ+1/2)_+$.
In particular, every such pair  $(\alpha_i,\beta_i)$ consists of nonnegative integers with opposite parity type.\\

Recall that $S$ itself is an $m$-tableaux (or empty) with at most $\kappa$ parts. 
Let $S_1, \cdots, S_{\kappa}$ denote the \emph{list} of rows of $S$ (where some of the rows, or even all of them, may be empty). 
Let the rightmost box of $S_i$ be filled with $C_i\in(\ZZ\pm1/4)_+$. 
If $S_i$ is empty, we define $C_i=|m-i|\in(\ZZ\pm1/4)_+$ (this is the filling of the rightmost box of the $(t+i)$-th row of 
$T_m(\rho)$, provided that $t\not=0$ (otherwise $S=T_m(\rho)$, in which case this row is empty by assumption)).
Define $\gamma_i$, for $i=1,\dots,\kappa$, as the unique nonnegative integer of parity type $\epsilon$ 
nearest to $2C_i\in(\ZZ\pm1/2)_+$.\\
 
This determines a set of pairs of nonnegative integers (of opposite parity) $(\alpha_i,\beta_i)$, and a 
set of nonnegative integers $\gamma_j$ uniquely. Observe that these integers are mutually distinct.\\


Now we form the \emph{descending} list $\jmath$ of the numbers $(\alpha_i$, $\beta_i)$ (for $i=1,\dots,t$) and the $\gamma_j$ (for $j=1,\dots,\kappa$). Observe that the length of $\jmath$ is at least $\kappa$ 
(namely, we have at least the 
numbers $\gamma_j$ for $i=1,\dots,\kappa$ in our list). 
Finally we define $\lambda:=2\jmath+{\bf 1}$. Observe that $\lambda$ has distinct, odd parts, 
and that $\delta$ (the parity of $\kappa$) is also the parity of the number of parts of $\lambda$.

We can also form the \emph{descending} list $\mathsf{e}$ of the numbers $(\alpha_i-1/2)/2$ the $(\beta_i -1/2)/2$ 
(for $i=1,\dots,t$) and the $(\gamma_j-1/2)/2$ (for $j=1,\dots,\kappa$). Observe that the list $\mathsf{e}$ may contain 
at most one negative number as an entry, namely $-1/4$. Then $\jmath=2\mathsf{e}+\mathsf{1/2}$, and  
the list $|\mathsf{e}|$ of absolute values of $\mathsf{e}$ is the 
list consisting of the fillings of arms (the $A_i$) and feet (the $B_i$) of the $m$-hooks $H_i$ of $T_m(\rho)$ 
(for $i=1,\dots, t$), combined with the list of the $C_j$ (for $j=1,\dots,\kappa$).


\begin{thm}
The operations $\E: \mathsf{P}_{\od} \to \mathsf{R}$ and $\D: \mathsf{R} \to \mathsf{P}_{\od}$ are inverse bijections. 
\end{thm}

We will not give the formal proof (which is straightforward, see \cite{Fe1}). 
Instead, we will give some illustrative examples.

\begin{ex} Examples of the operator $\D$:
\begin{enumerate}
\item Let $m=1/4$ and $\rho$ be zero. Then $\lambda$ is zero as well. On the other 
hand, if $m=3/4$ and $\rho$ is zero, then we have no $m$-hooks, but since $\kappa=1$ and $\epsilon=0$, 
we have one empty strip $S_1$, which yields $C_1=1/4$, and $\gamma_1=0$. Hence 
$\mathsf{e}=(-1/4)$, and $\lambda=[1]$.

\item Let $m=15/4$ and $\rho$ be the zero partition. Then $m=4-1/4$ gives $\kappa=4$ and $\epsilon  = 0$. 
We have no $m$-hooks, and $\kappa=4$ strips which are all empty. 
Hence $C_1=m-1=11/4$, $C_2=m-2=7/4$, $C_3=m-3=3/4$ and $C_4=|m-4|=1/4$, and thus 
$\gamma_1=6$, $\gamma_2=4$, $\gamma_1=2$, $\gamma_1=0$.
So $\mathsf{e}=(11/4, 7/4, 3/4, -1/4)$. We obtain the odd distinct partition $\lambda = [13, 9, 5, 1]$
(let us agree that we may also denote a partition in increasing order, using square brackets as delimiters).

\item For a singleton $\boxed{m}=\boxed{1/4}$ we have $\kappa=0, \epsilon =1$. 
Hence we have one hook, and no strips (even no empty ones!). 
We find that $\mathsf{e}=(1/4, -1/4)$; thus we get $\lambda = [3, 1]$.

\item Consider the following tableau:
	\begin{center}
	\ytableausetup{mathmode, boxsize=2.2em}
	\begin{ytableau}
    	5/4\\
	1/4
	\end{ytableau}
	\end{center}

Here $m=5/4=1+1/4$. Therefore $\kappa=1$ and $\epsilon = 1$. We have one $m$-hook, and one empty strip. 
Thus  $(A_1,B_1)=(5/4,1/4)$, and  $((\alpha_1-1/2)/2,(\beta_1-1/2)/2)=(5/4, -1/4)$. 
In addition the empty strip $S_1$ yields $C_1=|m-1|=1/4$, and thus $\gamma_1=1$. 
Thus we form the descending list $\mathsf{e} = (5/4, 1/4, -1/4)$ and recover the odd distinct partition $\lambda = [7, 3, 1]$.

\end{enumerate}
\end{ex}                             

\section{Discrete unramified Langlands parameters of even degree for classical groups}\label{sec:noodd}
In this section we will study the case of classical groups of positive $K$-rank. 
As observed in Corollary \ref{cor:sc}, in this case the formal degree of a cuspidal unipotent representation $\pi$ is, 
with our normalization of Haar measures, the product of a rational number and the value of a rational function 
in $q$ at $q=\mbf{q}$, where the reciprocal of this rational function is a product of \emph{even} cyclotomic polynomials.  
Thus a discrete unramified local Langlangds parameter $\varphi_\pi$ such that (\ref{eq:fdeg=gamma}) holds has to be  
``of even degree''. This turns out to be a very selective property. 
In the present section we will classify all discrete unramified Langlands 
parameters $\varphi$ which are of even degree. 

We first reduce the classification to real Langlands parameters of even degree in Section \ref{sub:redreal}, which 
come in the three flavours Odd Orthogonal, Symplectic and Even Orthogonal.  
The Odd Orthogonal case ($\delta=1/2$) is elementary, see Section \ref{sub:OO}.
The Even Orthogonal ($\delta=0$) and Symplectic ($\delta=1$) cases (see Section \ref{sec:countingdelta=0}) 
are more involved, requiring the technique of extra-special spectral transfer morphisms as discussed in the previous 
Section \ref{sub:muSTM}
to translate the problem to classifying real residue points of even degree for affine Hecke algebras with parameters 
of the form $m\in(\ZZ\pm 1/4)_+$ (see Section \ref{sub:oddextra}). 
The main classification results are Proposition \ref{cor:deltaishalf} and Proposition \ref{cor:delta01}.

We now first start in Subsection \ref{sub:mult} by 
fixing some notations for two natural bases of the free abelian subgroup of the multiplicative group 
of the field of rational functions in $v$ generated by the cyclotomic polynomials in $q=v^2$. 
\subsection{Multiplicity functions}\label{sub:mult}

We introduce some notations for counting the multiplicities of (odd) cyclotomic polynomial factors in the formal degrees. 
Recall the $n$-th cyclotomic polynomial  $\Phi_{n}(q) = \prod_{s | n} (1-q^s)^{\mu(n/s)}$.  Recall the following basic facts: 
\begin{itemize}
\item[(1)] The cyclotomic polynomials are distinct irreducible elements of $\ZZ[q]$. 
\item[(2)] $\Phi_n(q^2) = \Phi_{2n}(q)$ if $n$ is even, while $\Phi_n(q^2) = \Phi_{2n}(q) \Phi_n(q)$ if $n$ is odd. 
\item[(3)] A polynomial of the form $1+q^n\, (n\in\ZZ_{+})$ is a product of \emph{even} cyclotomic polynomials. 
\end{itemize}
Let $\mbf{K}$ be the fraction field of $\CC[v^{\pm}]$. Let $\mbf{M}_0$ be the subgroup of the multiplicative group 
$\mbf{K}^\times$ generated by $\QQ^\times$, $v$, and the set $\{\Phi_k(q): k \in \ZZ_+\}$. 
Then $\mbf{M}_0$ is the direct product
of $\QQ^\times$, of $v^\ZZ$, and of the free abelian group $\mbf{M}_c$ with basis $\{\Phi_k(q)\}_{k\in\ZZ_+}$.  
By M\"{o}bius inversion, 
$\{q^k-1\}_{k\in \ZZ_{+}}$ is also a basis of $\mbf{M}_c$. 
\begin{definition}\label{def:multi} 
Given $f\in \mbf{M}_0$, write $f=Cv^lf_c$ with $C\in\QQ^\times$, $l\in\ZZ$, and $f_c\in \mbf{M}_c$. 
\begin{itemize}
\item[(1)] For $n\in\ZZ_+$ we define $\cycl_f(n/2)$ as the coefficient of $\Phi_n(q)$ in the 
expansion of $f_c\in \mbf{M}_c$ with respect to the basis $\{\Phi_k(q)\}_{k\geq 1}$ of $\mbf{M}_c$ (i.e. 
the irreducible decomposition of $f_c$).
This defines a multiplicity function $\cycl_f:(\ZZ/2)_+\to\ZZ$.
If we say ``$f$ does not contain $\Phi_n$ as a factor", we mean $\cycl_f(n/2)=0$. 
\item[(2)] For $n\in\ZZ_+$ we define $\mult_f(n/2)$ as the coefficient of $q^n-1$ in the 
expansion of $f_c\in \mbf{M}_c$ with respect to the basis $\{q^k-1\}_{k\geq 1}$ of $\mbf{M}_c$. 
This defines a multiplicity function $\mult_f:(\ZZ/2)_+\to\ZZ$.
\end{itemize}
(We apologize to the reader for this convention of dividing the argument by $2$, but this turns out to be convenient in the context of this paper.) 
\end{definition}

If $k\in(\ZZ+1/2)_+$, by M\"{o}bius inversion we have 
\begin{equation}\label{eq:cycl}
\cycl_f(k)=\sum_{a\geq 1}\mult_f(ak).
\end{equation}
We suppress the subscript $f$ in $\mult_f$ and $\cycl_f$, if there is no danger of confusion.

\subsection{Reduction to real Langlands parameters}\label{sub:redreal}

A Langlands parameter $\varphi$ is called \emph{real} if $s\theta$ is $G^\vee$-conjugate to $\theta$. A residual point $r=cs \in T^d$ is called real if $s=1$. We have seen in Section \ref{sub:muSTM} that the residual points for 
a unipotent affine Hecke algebra of the form $\CCC_d(m_-,m_+)[q^\fb]$  
with $(m_-,m_+)\in\mathcal{V}$ are of the form 
$\respt_{(\lambda_-,\lambda_+)}=(-\respt_{\lambda_-},\respt_{\lambda_+})\in T^d$, 
with $\respt_{\lambda_\pm}=c_{\lambda_\pm}\in T^{d_\pm}$ a real residual point. 

From the definition and normalization of $\mu^d_{m_-, m_+}$ as discussed in Section \ref{sub:muSTM}
we see easily that 
\emph{modulo} even cyclotomic polynomial factors, rational constants and powers of $q$, we have a factorization: 
\begin{equation}\label{eq:fact}
f:=\mu_{\emm,\emp}^{d, \{\respt_{(\rho_-,\rho_+)} \} } (\respt_{(\rho_-,\rho_+)}) \equiv 
\mu_{\emp,\emm}^{d_-, \{\respt_{\rho_-}\} } (\respt_{\rho_-}) 
\mu_{\emm,\emp}^{d_+, \{\respt_{\rho_+}\} } (\respt_{\rho_+}) := f_- f_+.
\end{equation}
Also easy to see is the fact that, modulo even cyclotomic factors, powers of $q$ and nonzero rational factors, we have 
\begin{equation}
\mu_{m, m_\pm}^{d_\pm, \{\respt_{\rho_\pm}\} }(\respt_{\rho_\pm})\equiv 
\mu_{m',m_\pm}^{d_\pm, \{\respt_{\rho_\pm}\} }(\respt_{\rho_\pm})
\end{equation}
for any choice of $m'\in\ZZ/4$ such that the pair $(m',m_\pm)$ belongs to the same 
type $\mathcal{V}^X$ as $(m, m_{\pm})$ (which implies in particular that the base $\fb$ is 
the same for both parameter tuples).  
Therefore, without loss of generality, we may choose $m'$ as small as possible such that 
$(m',m_\pm)$ belongs to $\mathcal{V}^X$. 
For example, for the types $\text{V}$ and $\text{VI}$ we choose $m'=1/4$,  
so that the expressions to be analysed have the form $\mu_{1/4,m_\pm}^{d_\pm,{\{\respt_{\rho}\}}}(\respt_{\rho})$. \\

Clearly, if the multiplicity of all odd cyclotomic polynomials in both factors on the right-hand side of \eqref{eq:fact} is zero,
then the same thing is true on the left-hand side. 
Remarkably, the converse is also true in the following sense.
\begin{prop}\label{prop:odd} 
Let $\respt_{(\rho_-,\rho_+)}$ be a residual point for parameters $(\emm,\emp)$ of type \textup{I} to \textup{VI}.
If the support $\textup{Supp}(\cycl_{f_\pm})$ of $\cycl_{f_\pm}$ is not contained in $\ZZ$ for at least one of $f_-$ or $f_+$, let 
$p_\pm+1/2 \in \textup{Supp}(\cycl_{f_\pm})\cap (\ZZ+1/2)$ denote the maximal element. 
In this case we have:
$$
\cycl_{f_\pm}(p_\pm+1/2)>0,
$$ 
In particular the support of $\cycl_f$ is not contained in $\ZZ$ in this situation either, 
and if $p+1/2\in\textup{Supp}(\cycl_{f})\cap (\ZZ+1/2)$ is the maximal element, then $\cycl_f(p+1/2)>0$.
In other words, $\respt_{(\rho_-,\rho_+)}$ is of even degree if and only if 
$\respt_{\rho_-}$ and $\respt_{\rho_+}$ are both of even degree.
\end{prop}

Proposition \ref{prop:odd} yields a \emph{necessary} condition for the partitions $\lambda_\pm$  
such that $\psi_{0,T}(\respt_\varphi)=\respt_{(\lambda_-,\lambda_+)}\in T^l$ is a residual point for 
$\CCC_l(\delta_-,\delta_+)[q^\fb]$, where $\varphi$ is a discrete unipotent Langlands 
parameter satisfying \eqref{eq:fdegisgammafactor}. Indeed \eqref{eq:gammamu}, Corollary \ref{cor:iso} and 
Proposition \ref{prop:odd} imply that $\respt_{\lambda_\pm}$ need to be of even degree, 
i.e. that no odd cyclotomic factors can occur in 
$\mu_{\delta',\delta_\pm}^{l_\pm,{\{\respt_{\lambda_\pm}\}}}(\respt_{\lambda_\pm})$.
If $\delta_\pm\in\{0,1\}$ and $\fb=1$ we can further simplify this using \eqref{eq:es-STM}, to 
reformulate this as a condition on the pair $(m_\pm,\rho_\pm)$ corresponding to 
with $\lambda_\pm$ via the extra-special algorithm (we will also refer to this 
by saying that $(m_-,\rho_-)$  and $(m_+,\rho_+)$ are of even degree).
This necessary condition limits the list of eligible  
partitions $\lambda_\pm$ which may arise from a discrete unramified Langlands parameter $\varphi$ of a cusipdal unipotent representation considerably. \\

To prove Proposition \ref{prop:odd} we only need to consider the factors on the right-hand side of \eqref{eq:fact} individually. 
Thus we will omit the subscripts $\pm$ for the partitions and ranks for an individual factor. \footnote{The subscripts   
$\pm$ will be used a lot below, but with an entirely different meaning.} Using \eqref{eq:es-STM} and the results 
of Section \ref{e-s} we conclude that in order to prove Proposition \ref{prop:odd}, it suffices to show that:

\begin{prop}\label{prop:oddred} 
For all pairs $(m,\rho)$ with $\rho\vdash n$ a partition (possibly empty) and 
$m\in(\ZZ\pm 1/4)_+$ such that 
$\mu_{1/4,m}^{n, \{\respt_{\rho} \} }(\respt_{\rho})$ 
contains odd cyclotomic polynomial factors $\Phi_{2p+1}$, 
the multiplicity of $\Phi_{2p_m+1}$ with $p_m\in\ZZ_{>0}$ maximal such that 
$\Phi_{2p_m+1}$ has nonzero multiplicity in 
$\mu_{1/4,m}^{n, \{\respt_{\rho}\}} (\respt_{\rho})$ appears with \emph{positive} multiplicity.

The same statement is true for all $\delta=1/2$-unipotent classes $\lambda\vdash 2n$ (i.e. $\lambda\vdash 2n$ has even, distinct parts), and factors of the form $\Phi_{2p+1}$ in
$\mu_{1/2,1/2}^{n, \{\respt_{\lambda} \} }(\respt_{\lambda})$.
\end{prop}

The proof of Proposition \ref{prop:oddred}, hence of Proposition \ref{prop:odd}, 
as well as the classification of pairs $(m,\rho)$ 
(with $m\in(\ZZ\pm 1/4)_+$ and $\rho\vdash n\geq 0$) 
and of $1/2$-unipotent classes $\lambda$ for which no odd cyclotomic polynomials appeared in the residue of $\mu$-function as a factor, will be given in separate sections.

\subsection{Counting odd cyclotomic polynomials in the case $\delta=1/2$}\label{sub:OO}

Let $\lambda\vdash 2n$ be a partition with even, distinct parts. 
Let $\respt:=\respt_{1,1/2;\lambda}$ be the corresponding (infinitesimally) real 
residual point for $\mu_{1/2,1/2}^n$ as in \eqref{eq:resptreal}, i.e.~the
coordinates of $\respt$ are of the form $q^{x}$ for $x\in(\ZZ+1/2)_+$, where $q^x$ 
appears with multiplicity $h(x)$ defined by: $h(x+1)=h(x)+1$ if $x$ appears in 
the list of jumps $(\lambda-\mbf{1})/2$, and $h(x+1)=h(x)$ else.

\begin{lemma}\label{lem:deltahalf}
The real residual point $\respt$ for $\mu_{1/2,1/2}^n$ is of even degree 
if and only if there exists a  $p\in(\ZZ+1/2)_+$ such that $h$ is given by:
\begin{equation}\label{eq:form}
h(x)=\left\{
\begin{array}{ll}
p -x +1&\mathrm{if\; } 0< x\leq p, \\
0        &\mathrm{else}.\\
\end{array}
\right.
\end{equation}
Otherwise the highest odd cyclotomic factor $\Phi_{2j+1}$ of $\mu_{1/2,1/2}^{n, \{ \respt\}} (\respt)$ appears in its numerator.
\end{lemma}

\begin{proof}
Firstly, from the defining property of $h$ and the definition of $p$ we see that 
$$
h(p)=1.
$$

Assume that $\mu_{1/2,1/2}^{n, \{\respt\}} (\respt)$ contains no odd cyclotomic polynomial factors. 
Assume that $p=2k\pm 1/2$ is the maximum in the support of $h$, and that $0\leq i < k-1$.
A simple book-keeping using the expression of $\mu$-function 
(with $l$ replaced by $n$) determines $\mult(2p-2i)$ in terms of $h$ as follows (the inequality $2p-2i\geq p+3/2$ ensures that only the roots $(t_xt_y)^{\pm 1}$ (with possibly $x=y$) contribute to this multiplicity):
\begin{equation}\label{eq:mult}
\begin{aligned}
&\mult(2p-2i)=h(p-i)\left(h(p-i)-h(p-i-1)+1\right)\\
+&\sum_{x=1}^i h(p-i+x)\left(2h(p-i-x)-h(p-i-1-x)-h(p-i+1-x)\right)
\end{aligned}
\end{equation}
Since $2p-2i\geq p+3/2$, and since all factors of the form
$(1-q^k)$ of $\mu_{1/2, 1/2}^{n, \{\respt\}} (\respt)$ have order less than or equal to $2p+1$, it is clear that
$\mult(2p-2i)$ represents the multiplicity of the odd cyclotomic polynomial $\Phi_{2p-2i}$ as an irreducible factor of $\mu_{1/2, 1/2}^{n, \{\respt\}} (\respt)$. By assumption this multiplicity must therefore be equal to $0$.\\

For $i=0$, the equations $\mult(2p)=0$ and $h(p)=1$ imply that $h(p-1)=2$. Hence \eqref{eq:form} holds for all $x>p-2$. Now suppose by induction that \eqref{eq:form} holds for all $x>p-2j$ for some integer $1\leq j\leq k-1$.
Using this induction hypothesis we see that all summands of
$\mult(2p-2j)$ for $x\leq j-2$ vanish, so $\mult(2p-2j)=2j+2-h(p-2j-1)$. 
Hence we have $h(p-2j-1)=2j+2$, which implies, in view of the definition of $h(x)$ and $h(p-2j+1)=2j$, that 
$$
h(p-2j)=2j+1.
$$ 
Hence we find that \eqref{eq:form} holds for all $x>p-2j-2$ as well.
By induction this proves that $h$ satisfies \eqref{eq:form} if $p=2k-1/2$ for some nonnegative integer $K$, and it shows that \eqref{eq:form} is satisfied for all $x>1/2$ if $p$ is of the form $p=2k+1/2$. And we see that if $h$ does
not satisfy \eqref{eq:form} in this range of values for $x$ then
the numerator of $\mu_{1/2, 1/2}^{n, \{\respt\}} (\respt)$ has an odd cyclotomic polynomial.\\

Finally, for $p=2k+1/2$ we need to rule out the possibility
that $h(1/2)=p-1/2$. So assume that $h(1/2)=p-1/2$. 
We compute the multiplicity of the odd cyclotomic polynomial $\Phi_{p+1/2}=\Phi_{2k+1}$, but now $h(1/2)$ of the roots of the form $t_x t_y^{-1}$ also contribute in the denominator.
To compute this multiplicity, the easiest method is to compare $\mu_{1/2, 1/2}^{n, \{\respt\}} (\respt)$ with the analogous product 
$\mu_{1/2, 1/2}^{n, \{\respt'\}} (\respt')$,
where this time $\respt^\prime$ has one extra coordinate equal to $q^{1/2}$ compared to $\respt$. 
We already know that the multiplicity of $\Phi_{p+1/2}$ in
$\mu_{1/2, 1/2}^{n, \{\respt'\}} (\respt')$ is $0$ since $h^\prime$ (the multiplicity function
of $\respt^\prime$) does satisfy \eqref{eq:form}. 
The difference with the multiplicity of $\Phi_{p+1/2}$ in $\mu_{1/2, 1/2}^{n, \{\respt\}} (\respt)$ consists of two extra factors in the numerator (coming from a factor of the form $(1-q^{1/2}q^p)^2$) and 3 more in the denominator (one coming from $(1-qq^pq^{-1/2})$ and two from $(1-qq^{p-1}q^{1/2})$). 
Hence $\mu_{1/2, 1/2}^{n, \{\respt\}} (\respt)$ contains the factor $\Phi_{p+1/2}$, which violates our assumption. \\

To sum up, we have shown in all cases that if $h$ does not satisfy
\eqref{eq:form} then the highest odd cyclotomic polynomial $\Phi_{2j+1}$
(where we order the cyclotomic polynomials by the order of their associated roots)
of $\mu_{1/2, 1/2}^{n, \{\respt\}} (\respt)$ will be a factor of the numerator.
\end{proof}

As a consequence of Lemma \ref{lem:deltahalf}, we obtain the following:

\begin{prop}\label{cor:deltaishalf}
Suppose $n, r \in \ZZ_+$. In order that the expression $\mu_{1/2, 1/2}^{n, \{\respt\}} (\respt)$ contains no odd cyclotomic polynomials for a residual point $\respt = \respt_{\fb=1,1/2; \lambda}$, it is necessary and sufficient that the partition $\lambda$ is given by $\lambda=[2r, 2r-2, 2r-4, \ldots, 2] \vdash r(r+1)$ and $n=(r/2)(r+1)$ for some $r\in \ZZ_+$. In this case, this residual point represents a rank $0$ spectral transfer morphism $T_{1/2, r+1/2}^0 \to T_{1/2,1/2}^n$, and we have $$
\mu_{1/2, 1/2}^{n, \{\respt\}} (\respt)=\prod_{k\in\ZZ_+}(1+q^k)^{-h(k+1/2) - h(k-1/2)}
$$ 
with $h(\cdot)$ as in Lemma \ref{lem:deltahalf}.
\end{prop}

\subsection{Counting odd cyclotomic polynomials in the case  $m\in(\ZZ\pm 1/4)_+$}\label{sub:oddextra}

Given a pair $(m,\rho)$ with $m\in(\ZZ\pm 1/4)_+$ and $\rho \vdash r$ a partition, let us introduce some notations 
attached to the $m$-tableau $T_m(\rho)$ of $\rho$. 

We denote the entries of the upper-left, upper-right and lower-left cornered boxes of $T_m(\rho)$ by $m, p_+, p_-$ respectively. 
Note that $p_+\geq m$ and $p_+-m \in \ZZ$. 
Let $a_m \in [0, 1)$ be determined by $m- a_m \in \ZZ$. Then the $a_m$-diagonal inside the $m$-tableau indicates the change of congruence classes of the entries modulo $\ZZ$. So, all entries of $T_m(\rho)$ are in the same congruence class modulo $\ZZ$ if and only if $p_- - m\in\ZZ$. 
Denote the entry of the last box below $p_+$ by $r_+$, and the entry of the last box horizontal to the right of $p_-$ by $r_-$. 
Below is an example of a Young diagram with $p_+=15/4, m=3/4=a_m, p_-=9/4$ and $r_+=7/4, r_-=5/4$.
\begin{center}
\ytableausetup{mathmode, boxsize=2.2em}
\begin{ytableau}
 3/4 & 7/4 & 11/4 & 15/4  \\
 1/4 & 3/4 & 7/4 & 11/4 \\
 5/4 & 1/4 & 3/4 & 7/4\\ 
 9/4 & 5/4
\end{ytableau}
\end{center}
We will now study the \emph{odd} cyclotomic polynomial factors in the $q$-rational part of residues 
$\mu_{1/4,m}^{r, \{\respt_{\rho}\} } (\respt_{\rho})$ at a real residual point $\respt_{\rho}=\respt_{\fb=2,m;\rho}$.
Since we neglect even cyclotomic polynomials of the regularized value of $\mu_{1/4,m}^r$, after substitution of $(t_1,\dots,t_r)$ by 
$\respt_{\rho}$ (whose coordinates all are odd powers of $v$)
we can replace $\mu_{1/4,m}^r$ by  the following product (as usual, up to nonzero rational constants and powers of $v$):
\begin{equation}\label{eq:goodwill}
\begin{aligned}  
 &\frac{1}{(1-q)^{r} } \prod_{z=1}^{r} \frac{(1-t_z^2)^2}{(1-q^{2m}t_z) (1-q^{-2m} t_z)} \times \\
&\prod_{1 \leq i < j \leq r} \frac{(1-t_i t_j)^2(1-t_i t_j^{-1})^2}{(1-q^2 t_i t_j) (1-q^{-2} t_i t_j) (1-q^2 t_i t_j^{-1}) (1-q^{-2} t_i t_j^{-1})}.
\end{aligned}
\end{equation}
Recall that the coordinates of $\respt_{\rho}$ are of the form $q^{2c(b)}=v^{4c(b)}$ where $b$ runs over the 
boxes of $T_m(\rho)$ and $c(b)\in(\ZZ\pm 1/4)_+$ denotes the content of the box $b$. 
Since $\mu^r_{1/4,m}$ is $W_0^r$-invariant we may replace any number of coordinates $q^{2c(b)}$ by its reciprocal $q^{-2c(b)}$. 
Note however that $c(b)$ and $-c(b)$ are always in a different congruence class modulo $\ZZ$ in the current situation. This is a key fact in all that follows.\\

We choose Weyl group elements $w_\pm$ such that $w_\pm(\respt_{\rho})$ has all its coordinates of the form $v^{4x}$ with $x\pmod{\ZZ}=\pm m$. 
Then the multiplicity $h_\pm(x)$ of $v^{4x}$ in $w(\respt_{\rho})$ is \emph{independent} of the choice of $w_\pm$. 
This defines two multiplicity functions $h_{\pm }:\ZZ\pm m\to\ZZ_{\geq 0}$ which satisfy the obvious relation $h_-(x)=h_+(-x)$ for all $x$. 

The coordinates of $w_+(\respt_{\rho})$ are the contents of the boxes of 
$T_{\rho}(m,+)$, 
which is \textbf{defined} as the Young tableau of $\rho$ with its boxes above or on the $a_m$-diagonal filled like 
those of $T_\rho(m)$, 
but below the $a_m$-diagonal we multiply the contents of the boxes of $T_\rho(m)$ by  $-1$. 
Similarly, we \textbf{define} $T_{\rho}(m,-)$ by multiplying the content of the boxes of $T_\rho(m)$ above or on the $a_m$-diagonal by $-1$, leaving the boxes below the $a_m$-diagonal unchanged. 
Hence the multiplicity of the coordinate $q^{2x}$ in  $w_\pm(\respt_{\rho})$ 
is equal to the length of the $x$-diagonal in $T_{\rho}(m,\pm)$. 
\begin{rem}\label{rem:xpm} 
We introduce the following convention, which will only be used in this section. Given $m\in(\ZZ\pm 1/4)_+$, for $x\in\ZZ\pm 1/4$ we adopt the notation $x^\pm$ to denote the unique element $x^\pm\in \{-x,x\}$ such that $x^\pm\in \ZZ\pm m$.
For example, $m^\pm=\pm m$, and we always have $p_+^+=p_+$. 
\end{rem}
Now we analyze the multiplicity functions $\mult:=\mult_f$ and 
$\cycl(k):=\cycl_f(k)$ (cf.~Definition \ref{def:multi} and \eqref{eq:cycl}) for 
$f=\mu_{1/4,m}^{r, \{\respt_{\rho}\} }(\respt_{ \rho})$ and $k\geq (p_+ + p_-)/4$. 
Let $k \in \ZZ_{\geq 0} + 3/2$ (we use $3/2$ here, and not $1/2$, because this turns out to be more convenient, 
and since our normalization of Haar measures makes sure that the factor $\Phi_1=q-1$ 
has multiplicity $0$ in the formal degree of any discrete series character).
We will need to consider the functions $\mult(ak)$ for $a=1, 2, 4$. As seen below, we distinguish between various 
contributions, coming from factors associated with different types of roots:
\begin{itemize}
\item $\mult_+(k)$ and $\mult_-(k)$:  For the function $\mathrm{mult}_\pm(k)$
we take into consideration the contribution to $(1-q^{2k})$ 
from the $(1-t_i t_j)^2$-terms in the numerator, where $t_i$, $t_j$ 
are coordinates of $w_\pm(\respt_{\rho})$ (hence $t_i=q^{2x}$, $t_j=q^{2y}$ and 
$x,y$ are both in the congruence class of $m \pmod{\ZZ}$ or $-m \pmod{\ZZ}$), the factors 
$(1-q^2t_i t_j)$ and $(1-q^{-2}t_i t_j)$ in the denominator for such $t_i,t_j$,  
as well as the factors $(1-t_z^2)^2$. 

In this way, any (unordered) pair of boxes of $T_\rho(m,\pm)$ with contents $x$ and $y$ such that $x+y\geq 3/2$  
contributes $+2$ to $\mathrm{mult}_+(x+y)$ 
and $-1$ to $\mathrm{mult}_+(x+y+1)$ and $\mathrm{mult}_+(x+y-1)$. 
Moreover, every single box of $T_m(\rho)$ contributes  $+2$ to  $\mathrm{mult}_+(2x)$.
Note that the maximal entry $p_\pm$ occurs only once in $T_m(\rho)$.  
This implies in all cases easily that $\mult_\pm(k)=0$ for $k\geq 2p_\pm$.\\

\item $\mathrm{mult}_{+,-}(2k)$:  Consider factors of the form $(1-q^{4k})=(1-q^{2k})(1+q^{2k})$, where $4k$ equals twice an odd number.
In $\mathrm{mult}_{+,-}(2k)$, we count such factors of $\mu_{1/4,m}^{r, \{\respt_{\rho} \}} (\respt_{ \rho})$ 
arising from type $\DDD$-roots, via factors of the form 
$(1-t_i t_j)^2$ (in the numerator) 
or $(1-q^2t_i t_j)$ or $(1-q^{-2}t_i t_j)$  (both in the denominator), with $t_i$ a coordinate of $w_+(\respt_\rho)$,  
and $t_j$ a coordinate of $w_-(\respt_\rho)$. 

In such terms, the pair $\{i,j\}$ corresponds to a pair of boxes, one with entry  
$x^+$ of $T_m(\rho,+)$ and one with entry $y^-$ of  $T_m(\rho,-)$, thus 
in different congruence classes modulo $\ZZ$. 
(If $T_m(\rho)$ contains entires below the $a_m$-diagonal, then $x^+$ is on or above the $a_m$-diagonal,
and $y^-$ is below that diagonal in $T_m(\rho)$.) 
In the numerator terms we need that $2k=x^++y^-$ is an \emph{odd} integer. 

Then $t_i=q^{2x^+}$ is a coordinate of $w_+(\respt_{\rho})$, and 
$t_j=q^{2y^-}$ is a coordinate of $w_-(\respt_{\rho})$. 
The cardinality of the set of unordered pairs $\{i,j\}$ such that the corresponding coordinate pair $\{t_i,t_j\}$ of $\respt_{\rho}$ satisfies 
$\{t_i,t_j\}=\{q^{2x^+},q^{2y^-}\}$ (with $x^++y^-$ \emph{odd}) 
is $h_+(x^+)h_-(y^-)$. Each such pair contributes $+2$ to  
$\mathrm{mult}_{+,-}(x^++y^-)$. 
On the other hand, an unordered pair of such boxes with $x^++y^-$ \emph{even}  
contributes  $-1$ to $\mathrm{mult}_{+,-}(x^++y^-- 1)$ and to $\mathrm{mult}_{+,-}(x^++y^-+1)$. Notice that $\mult_{+,-}(2k)=0$ if $2k>p_++p_-^-$, for obvious reasons.\\

\item $\mathrm{mult}_{+,-}(4k)$: Similar as for $\mathrm{mult}_{+,-}(2k)$, but now 
counting the multiplicity of the factors 
of the form 
$(1-q^{8k})=(1+q^{4k})(1+q^{2k})(1-q^{2k})$, 
where $8k$ is four times an odd number. For factors in the numerator we are thus searching 
$x,y$ such that $4k=x^++y^-\equiv 2\pmod{4}$. 
In the denominator terms we should solve $4k=x^+ + y^- \pm 1\equiv 2\pmod{4}$.
(Since our overall assumption will be that $4k\geq (p_++p_-)$, 
$\mathrm{mult}_{+,-}(4k)$ is possibly nonzero only for the smallest values of $k$ in our range, depending on the congruence class of $p_+ + p_-$ modulo $4$.)\\

\item The terms $-h_\pm(k\mp m)$, coming from the denominator in the first line of 
\eqref{eq:goodwill}, as well as the possible contributions from this denominator to $(1-q^{2k})$
via a contribution to $(1-q^{4k})$ and $(1-q^{8k})$, 
in other words the terms $-h_\pm(2k\pm m)$ and $-h_\pm(4k\pm m)$.
\end{itemize}

To sum up, for $k\in(\ZZ+ 1/2)_+$ we have:
\begin{equation}\label{eq:multodd}
\mult(k):=(\mult_+(k)-h_+(k-m))+(\mult_-(k)-h_-(k+m))
\end{equation}
while for $n\in\ZZ_{+}$ we have:
\begin{equation}\label{eq:multeven}
\mult(n):=(\mult_{+,-}(n)-h_-(n-m))-h_+(n+m)
\end{equation}

We now compute each of these summands in terms of $h_-$ and $h_+$. 
Note that both $2p_+-k$ and $2p_--k$ are integers since $p_+, p_- \in \ZZ \pm 1/4$ and $k \in (\ZZ_+ + 1/2)_+$. Firstly, we look at the factors  
$$
\frac{(1-t_i t_j)^2}{(1-q^2 t_i t_j)(1-q^{-2} t_i t_j)}.
$$
To find out the multiplicity $\mult_+(k)$ of the factor $(1-q^{2k})$, we separate two cases, based on the parity of $2p_+-k$.\\

Case (i): $2p_+-k = 2i$ is even. Then
\begin{equation}
\begin{aligned}
&\mult_+(k)\\ 
=& 2\sum_{x=1}^i h_+(p_+-i+x) h_+(p_+-i-x) + 2\times \frac{1}{2}h_+(p_+-i)[h_+(p_+-i)-1]+2h_+(p_+-i)  \\
&-\sum_{x=1}^i h_+(p_+-i+x)h_+(p_+-i-x+1) - \sum_{x=0}^i h_+(p_+-i+x) h_+(p_+-i-x-1).
\end{aligned}
\end{equation} 

Because $t_i t_j =q^{2(c(i)+c(j))}$ with $i<j$, we need to count twice the multiplicities of the entries in the Young diagram whose sum is $k$ (corresponding to $(1-t_i t_j)^2$), and subtract once those with sum $k+1$ (corresponding to $1-q^2t_i t_j$) and similar for $k-1$ (corresponding to $1-q^{-2} t_i t_j$). 
For the first case we add the terms $2h_+(x)h_+(y)$ for $(x, y)$ a solution to $x+y=k$ with $p_+\geq x\geq y$ and $x, y\equiv p_+\pmod{\ZZ}$. 
The solutions to this equation yield the first 2 terms in the formula
(The second term corresponds to $x=y=k/2=p-i$). The third term corresponds to the factor $(1-t_z^2)^2$ in the $\mu$-function. 
The last two terms correspond to the factor $(1-q^2t_i t_j)(1-q^{-2} t_i t_j)$.\\

We introduce the $\Delta$-operator which is defined by 
\begin{equation}\label{Delta}
\Delta h(x) = 2h(x) - h(x+1)-h(x-1),
\end{equation}
and the ``jump function'' for all $x$:
\begin{equation}\label{jump}
J_\pm(x)=h_\pm(x)-h_\pm(x+1) \in \{-1, 0, 1\}.
\end{equation}
With these notations we can rewrite the formula of $\mathrm{mult}_+ (k)$ neatly (if $k=2p_+-2i$):
\begin{equation*}
\mult_+(k) = \sum_{x=1}^i h_+(p_+-i+x) \Delta h_+(p_+-i-x) \\
+ h_+(p_+-i)[1-J_+(p_+-i-1)].
\end{equation*}

Case (ii): $2p_+-k=2i-1$ is odd. In this case for the numerator we lose the terms corresponding to $(k/2, k/2)$ because of parity, 
and also $(1-t_z^2)^2$ does not contribute because of parity.
But for the denominator we have the contributions corresponding to $(x, y)$ 
with $x=y=(k+1)/2$ and $x=y=(k-1)/2$. So the formula becomes:
\begin{equation*}
\begin{aligned}
&\mult_+(k)\\
=&2\sum_{x=1}^i h_+(p_+-i+x) h_+(p_+-i-x+1) \\
-& \sum_{x=1}^i h_+(p_+-i+x) h_+(p_+-i-x) - \frac{1}{2}h_+(p_+-i)[h_+(p_+-i)-1]\\
-& \sum_{x=1}^{i-1} h_+(p_+-i+x+1) h_+(p_+-i-x+1)- \frac{1}{2}h_+(p_+-i+1)[h_+(p_+-i+1)-1]
\end{aligned}
\end{equation*}
Using $\Delta$ and $J_{\pm}$ we can rewrite this formula as (for $k =2p_+-2i+1$)
\begin{equation*}
\mult_+(k) = \sum_{x=1}^i h_+(p_+-i+x) \Delta h_+(p_+-i-x+1) + \frac{1}{2}[h_+(p_+-i+1)+h_+(p_+-i)][1-J_+(p_+-i)].
\end{equation*}
For all $k \in \ZZ_{\geq 0} +\frac{3}{2}$ we can formally write 
\begin{equation}\label{eq:mult+}
\mult_+(k) = \sum_{\substack{x+y=k, \; p_+\geq x>y\\ x\equiv y \equiv m (\ZZ)}} h_+(x) \Delta h_+(y) + R_+(k)=:M_+(k)+R_+(k),
\end{equation}
where the ``remainder'' $R_+(k)$ denotes the term which is not in the summation symbol of the above formulae. \emph{Observe that $R_+(k)\geq 0$}. \\

By virtue of symmetry we obtain the formula of $\mult_-(k)$ from the formula of $\mult_+(k)$ by replacing $+$ by $-$.
But it is important to observe here that the contributions of $(1-t_z^2)^2$ of the form $(1-q^{2k})$ with $k\equiv 2m\pmod{2\ZZ}$
contribute to $\mult_+(k)$, while those with $k\equiv -2m\pmod{2\ZZ}$ contribute to $\mult_-(k)$. 
Together we have used all such contributions coming from this factor of the $\mu$-function exactly once. Hence we may write
\begin{equation}\label{eq:mult-}
\mult_-(k) = \sum_{\substack{x+y=k, \; p_-\geq x>y\\ x\equiv y \equiv -m\pmod{\ZZ}}} h_-(x) \Delta h_-(y) + R_-(k), 
\end{equation}
with $R_-(k)\geq 0$. 

\begin{rem} 
Observe that $J_-(x) + J_+(-x-1) = 0$ for all $x\in\ZZ\pm 1/4$. 
\end{rem}

Next we look at the summand $\mathrm{mult}_{+,-}(2k)$. We need to consider the expression 
$$
\frac{(1-t_i t_j)^2}{(1-q^2 t_i t_j)(1-q^{-2} t_i t_j)}.
$$
where we now take $t_i$ and $t_j$ from opposite congruence classes modulo $\ZZ$. Then $c(i)+c(j)\in\ZZ$.
The contribution to an odd cyclotomic factor $(1-q^{2k})$ (with $2k$ odd) from terms of this kind comes from their contribution to the factor of the form $1-q^{2(c(i)+c(j))}$ with $c(i)+c(j)=2k$ or $4k$ (and not $k$, because $c(i)+c(j)\in\ZZ$ now). 
This readily yields: 
\begin{equation}\label{eq:multpm}
\mathrm{mult}_{+,-}(2k) = \sum_{\substack{a+b=2k\\ a \equiv m(\ZZ), b\equiv -m(\ZZ)}} h_+(a) \Delta h_-(b) = \sum_{\substack{a+b=2k\\ a \equiv m(\ZZ), b\equiv -m(\ZZ)}} \Delta h_+(a)\, h_-(b).
\end{equation}
The second equality holds by symmetry. \\

Finally, in our range $k\geq(p_++p_-)/4$, we need to subtract the multiplicity of the odd cyclotomic polynomials 
occurring in the factors 
$(1-q^{2k})$, $(1-q^{4k})$ and $(1-q^{8k})$ appearing in the denominator of the factors  
$$
\frac{1}{(1-q^{2m}t_z) (1-q^{-2m}t_z)}. 
$$
This yields 12 contributions to $\mult(k)$: 
$-h_\pm(k\mp m)$, $-h_\pm(2k\pm m)$ and $-h_\pm(4k\pm m)$.\\

Now we turn to the criterion for $\cycl(k)=0$. 

\begin{lemma}\label{lem:belowa_m}
Assume that $(m,\rho)$ (with $\rho$ a partition of the rank $r$) is such that for the odd cyclotomic factors of 
$f=\mu_{1/4,m}^{r, \{\respt_{\rho} \} }(\respt_{\rho})$
we have $\cycl(k)=0$ for all $k\in(\ZZ+1/2)_+$ and $k>p_+$. 
If $r=1$ then $m=1/4$; otherwise $T_m(\rho)$
contains boxes below the $a_m$-diagonal (so $p_-^-=p_-$ if $r>1$).
\end{lemma}

\begin{proof}
If $r=1$ then $f\equiv (q-1)^{-1}(q^{4m}-1)$ modulo even cyclotomic factors. Hence unless 
$4m=1$ we have $\cycl(2m)=1$ and $k=2m > p_+=m$, contradicting the assumption.\\

Assume that $r>1$, and that $T_m(\rho)$ has no entires below the 
$a_m$-diagonal. Put $k=p_++r_+\in (\ZZ+1/2)_+$.
Now $\mult_+(k)-h_+(k-m)\geq 1$ by \eqref{eq:mult+} (use that 
$\Delta(h_+)(r_+)\geq 1$ and that $\Delta(h_+)(m)\geq 1$, and that for all $y$ such that $r_+< y\leq p_+$ and $y\not= m$, we have 
$\Delta(h_+)(y)=0$). 
Since clearly $2k>p_++p_-^-$, we thus have $\mult_{+,-}(sk)=0$ for all 
$s\geq 0$. 
Also $2k\pm m>p_+>p_-^-$, so $h_\pm(2k\pm m)=0$, and  $k+m>p_-^-$ so $h_-(k+m)=0$. 
Hence \eqref{eq:cycl}, \eqref{eq:multodd} and \eqref{eq:multeven} imply that $\cycl(k)=\mult(k)>0$, contradicting the assumption.
\end{proof}

\begin{cor}\label{cor:atleast}
Assume $(m,\rho)$ as in Lemma \ref{lem:belowa_m}, and $r>1$. 
Then $p_++p_-\geq 1$. 
If $k\in(\ZZ+1/2)_+$ satisfies $\mult(2k)\not=0$, then $k\leq (p_++p_-+1)/2$ if $p_++p_-\in\ZZ$ is 
even, and $k\leq (p_+ + p_-)/2$ otherwise. 
\end{cor}

\begin{proof}
If $r>1$ then $p_-^-=p_-$, i.e. $p_-$ is in a different class modulo $\ZZ$ than $p_+$. It 
is immediate that $p_++p_-\geq 1$. Looking at \eqref{eq:multeven} and \eqref{eq:multpm},
we see easily that for $K$ above the indicated bounds, $\mult(2k)=0$ (use that 
$p_+\geq m$, and that $p_-^-=p_-$). 
\end{proof}

\begin{lemma}\label{lem:ppluspmin}
Assume that $(m,\rho)$ is as in Lemma \ref{lem:belowa_m}, and $r>1$. 
Then $p_+>p_-$.
\end{lemma}

\begin{proof}
We have already shown that $p_+$ and $p_-$ are in different classes modulo $\ZZ$.
Assume on the contrary that $p_->p_+$. If $r_-$ would be below the $a_m$-diagonal too, then 
\eqref{eq:mult-} implies that for $k=p_-+r_->p_->p_+$ we have $\mult_-(k) > h_-(k+m)$. 
But by the previous Corollary, $\mult(2k)=0$ because $p_-+r_->(p_++p_-)/2$ by our assumption.
By \eqref{eq:cycl}, \eqref{eq:multodd} and \eqref{eq:multeven} our assumption $\cycl(k)=0$  
implies that $\mult_+(k) < h_+(k-m)$. But from \eqref{eq:mult+} it then follows that $k'=p_+ + r_+\geq k$, 
because $k'$ is the largest argument such that $\mult_+(k') > h_+(k'-m) $. As we have seen, $\mult_+(k') > h_+(k'-m)$, and the considerations above concerning the other terms in $\cycl(k)$ hold \textit{a fortiori} for $k'>k$. Thus $\cycl(k')>0$, 
a contradiction.
Hence we conclude that $r_-$ has to be above the $a_m$-diagonal. 
This implies that the largest part $p_+-m+1$ of $\rho$ must be at least 
$p_-+a_m+1$, since this is now smaller than or equal to the length of 
the smallest part of $\rho$. Thus $p_+\geq p_-+a_m+m>p_-$, contradiction. 
\end{proof}

\begin{lemma}\label{lem:equal}
Assume that $(m,\rho)$ is such that $\cycl(k)=0$ for all $k\in(\ZZ+1/2)_+$
with $k\geq (p_++p_-)/4$. Then all parts of $\rho$ are equal, i.e.~the tableau $T_m(\rho)$ is rectangular.
\end{lemma}

\begin{proof}
Clearly we may assume without loss of generality that $r>1$. 
Assume first that $r_+$ is above the $a_m$-diagonal. Put $k_1^+:=p_++r_+$,
\footnote{The notation $k_1^\pm$ is not to be confused with the notation $x^\pm$ for $x\in\ZZ\pm 1/4$ as in Remark \ref{rem:xpm}. } 
then $\mult_+(k_1^+) > h_+(k_1^+-m)$.  
It follows from \eqref{eq:multeven} that the largest argument $k'$ for which 
$\mult(k')\not=0$ is either $k_1^+$ or $k_1^-:=p_-+r_-^-$ if $k_1^->k_1^+$. 
In any case, $k_1:=\max\{k_1^+,k_1^-\}\geq k_1^+=p_++r_+>(p_++p_-)/2$
(the inequality follows from Lemma \ref{lem:ppluspmin}), and $\mult(k_1)>0$.  

Our assumption $\cycl(k_1)=0$ and 
\eqref{eq:cycl} now force that $\mult(2k_1)\not=0$. Corollary \ref{cor:atleast} then implies that $k_1\leq (p_++p_-+1)/2$ 
(if $p_++p_-$ is even) or $k_1\leq (p_++p_-)/2$ (if $p_++p_-$ is odd). In the second case we reach a contradiction
with the above, so we conclude that $p_++p_-$ is even, and that $k_1=k_1^+=(p_++p_-+1)/2$. 
Then \eqref{eq:multeven} and \eqref{eq:multpm} imply that $\mult(2k_1)=(\mult(2k_1)-h_-(2k_1-m))\leq h_-(2k_1-p_+-1)=1$, so that 
$k_1^+=k_1^-$ is not allowed (since that would imply that $\mult(k_1)=2$, so that $\cycl(k_1)\geq 1$, a contradiction). Thus $k_1^-<k_1^+$. 

Now notice that $2k_1^+=2p_++2r_+=p_++p_-+1$, 
implying that $p_++2r_+=p_-+1$, so that $0<2r_+=1+p_--p_+<1$. 
It follows that $r_+=1/4$, and $p_+=p_-+1/2$. 
Suppose now that $r_-$ is still below the $a_m=1/4$-diagonal, then $r_-\geq 3/4$ and  
$k_1^-=p_-+r_-\geq p_++1/4=k_1^+$, contradicting our earlier conclusion that $k_1^-<k_1^+$. 
Thus $r_-$ must be above the $a_m=1/4$--diagonal, and in fact we must have $r_-=r_+=1/4$.
We finally conclude that $T_m(\rho)$ is a square diagram with $m=1/4$, and $p_+=n+1/4$, 
$p_-=n-1/4$ and $r_- = r_+ = 1/4$. This finishes the case where $r_+$ is above the diagonal.

Next, assume that $r_+$ is below the $a_m$--diagonal. 
Then $r_-$ is below the $a_m$-diagonal as well. \emph{This implies in particular that 
either $r_-=r_+$ (which is what we want to show) or otherwise $r_- - r_+\geq 2$}.

We have seen that the largest argument $k_1$ for which $\mult(k_1)>0$ is $k_1=\max\{k_1^+,k_1^-\}$, 
where this time (because of the congruence classes of $r_-$ and $r_+$ modulo $\ZZ$) $k_1^+=p_+-r_+$ and $k_1^-=p_-+r_-$. 
Since $\cycl(k_1)=0$ we must have that $\mult(2k_1)\not=0$, which 
implies by Corollary \ref{cor:atleast} that $k_1\leq (p_++p_-+1)/2$ 
(if $p_++p_-$ is even) and $k_1\leq (p_++p_-)/2$ (if $p_++p_-$ is odd) as before.
We note that $\mult(p_+ + p_- + 1)=1$ in the first case ($p_++p_-$ even), while $\mult(p_++p_-)=2$ in the second case.
Assume that $p_++p_-$ is even. Then $p_++p_-+(r_--r_+)=k_1^++k_1^-\leq 2k_1 \leq p_++p_-+1$. Thus $r_+=r_-$, as desired.

Next, assume that $p_++p_-$ is odd. 
Then $p_++p_-+(r_--r_+)=k_1^++k_1^-\leq 2k_1 \leq p_++p_-$.
Again it follows that $r_-=r_+$, and we are done.
\end{proof}

\begin{cor}\label{for:mr} 
In the notations of the proof of Lemma \ref{lem:equal}, put $r_-=r_+:=r$. 
We have $m\geq r$. Moreover, $r\equiv m \pmod{\ZZ}$ if and only if $m=r=1/4$.
\end{cor}

\begin{proof} 
It was shown in the proof of Lemma \ref{lem:equal} that 
$r\equiv m \pmod{\ZZ}$ implies that $m=r=1/4$. If $r\not\equiv m \pmod{\ZZ}$
then Lemmas \ref{lem:belowa_m} and \ref{lem:equal} imply 
that $p_-\equiv r \pmod{\ZZ}$, and by Lemma \ref{lem:equal} it follows that 
$p_+ - m = p_- -r$, or $m - r = p_+ - p_-$. The assertion now follows from Lemma \ref{lem:ppluspmin}.
\end{proof}

The following proposition is the technical heart of this section, and plays a main role in Section \ref{sec:countingdelta=0} in the classification of the partitions $\lambda$ for which $\mu_{0, \delta}^{\{\respt\}}(\respt)$ has no odd cyclotomic factors (see Proposition \ref{cor:delta01}). We shall prove it by a case-by-case check.

\begin{prop}\label{prop:mainrho}
Assume that $(m,\rho)$ is such that $\cycl(k)=0$ for all $k\in(\ZZ+1/2)_+$
with $k\geq (p_++p_-)/4$. Then $(m,\rho)$ is one of the following possibilities:
\begin{itemize}
\item[(a)] $m$ is arbitrary and $\rho$ is empty.
\item[(b)] $m=1/4$ and $\rho$ is a square diagram, so that $r_- = r_+=1/4$. 
\item[(c)] $m=3/4$, and $\rho$ is a rectangular diagram such that $r_- = r_+=1/4$. 
In this case 
we can write $p_+=n+3/4$, $p_- = n+1/4$ for some $n\in\ZZ_{\geq 0}$.
\item[(d)] $m=5/4$, and $\rho$ is a rectangular diagram such that $r_- = r_+=3/4$. 
In this case 
we can write $p_+=2n+5/4$, $p_- = 2n+3/4$ for some $n\in\ZZ_{\geq 0}$.
\item[(e)] $m=7/4$, and $\rho$ is a rectangular diagram such that 
$r_- = r_+ = 1/4$. In this case we can write $p_+=2n+7/4$, $p_-=2n+1/4$ for some $n\in\ZZ_{\geq 0}$.
\end{itemize} 

If $(m,\rho)$ does not belong to any of these cases, the largest $k_{(m,\rho)}\in\ZZ+1/2$ for which $\cycl(k_{(m,\rho)})\not=0$ should satisfy $k\geq(p_++p_-)/4$. In that case we will have $\cycl(k_{(m,\rho)})>0$.
\end{prop}

\begin{proof}
We have seen in the proof of Lemma \ref{lem:equal} that if $\rho$ is not empty and 
$r_+$ is above the $a_m$-diagonal, then we are in case (b). Hence we now assume that $\rho$ is not empty and $r:=r_+=r_-$ is below the $a_m$-diagonal. We are left with the task of proving that we are in one of the cases (c) to (e). \\

Since we are assuming $k\geq (p_++p_-)/4$, we see that $\cycl(k)=\mult(k)+\mult(2k)+\mult(3k)+\mult(4k)$. 
For the same reason, we see that  $\cycl(3k)=\mult(3k)=0$ (the latter equality holds by assumption).
Hence 
\begin{equation}\label{eq:cyclev}
\cycl(k)=\mult(k)+\mult(2k)+\mult(4k).
\end{equation}
We will see that $\mult(k)$ for $k\geq 1/2$ is given by the remarkably simple 
explicit formula \eqref{eq:mult}, and that $\mult(k)\geq 0$. 
In the proof below we analyze the implications of the 
requirement that $\mult(2k)+\mult(4k)$ cancels $\mult(k)$ for $k\geq (p_-+p_+)/4$.\\

Following the steps and notations of the proof of Lemma \ref{lem:equal}, we put 
$k_1=\max\{k_1^-,k_1^+\}$, with $k_1^\pm=p_\pm\mp r$. Then $k_1$ is the largest argument for which $\mult(k_1)\not=0$. 
Using \eqref{eq:mult+}, \eqref{eq:mult-} and \eqref{eq:multodd} 
the rectangular shape of $T_m(\rho)$ implies that for all $k\geq 1/2$: 
\begin{equation}\label{eq:mult.}
\mult(k)=h_+(k+r)+h_-(k-r)\geq 0.
\end{equation}

The simplicity of this formula is somewhat deceptive. One uses that 
 $\Delta(h_\pm)(y)=1$ if $y=\pm m$ or $y=\mp r$, 
 $\Delta(h_\pm)(y)=-1$ if $y=\pm (p_++1)$ or $y=\mp (p_-+1)$, 
 and  $\Delta(h_\pm)(y)=0$ otherwise.
For example, to see that 
$\mult_+(k)-h_+(k-m)=h_+(k+r)$ for all $k\geq 1/2$, one checks that for $k> 2m$ 
we have 
$$
M_+(k):=h_+(k-m)+h_+(k+r)-h_+(k+p_-+1)=h_+(k-m)+h_+(k+r)
$$ 
and $R_+(k)=0$; for $m-r<k\leq 2m$ we have 
\begin{equation*}
M_+(k) :=h_+(k+r)-h_+(k+p_-+1)=h_+(k+r),
\end{equation*} 
and $R_+(k)=h_+(m)=h_+(k-m)$; while for $1/2\leq k\leq m-r$ we have 
$M_+(k)=h_+(k+r)-h_+(k+p_-+1)$ and $R_+(k)=h_+(m)$, where we observe that 
$h_+(k-m)=h_+(m)-h_+(k+p_-+1)$. Similar observations apply to show that 
$$
\mult_-(k)-h_-(k+m)=h_-(k-r).
$$

The function $h_+(k+r)$ is zero for $k>k_1^+=p_+-r$, then equals the linear function 
$1+k_1^+-k$ for $m-r\leq k\leq k_1^++1$, and then is constant for $1/2\leq k\leq m-r$.

The function $h_-(k-r)$ is zero for $k>k_1^-=p_-+r$, then equals the linear function 
$1+k_1^--k$ for $2r\leq k\leq k_1^-+1$, and then is constant for $1/2\leq k\leq 2r$. \\

Let us now look at $\mult(2k)$. In the rectangular case, by \eqref{eq:multeven} and \eqref{eq:multpm}  we have:
\begin{equation*}
\mult(2k)=\mult_{+,-}(2k)-h_-(2k-m) -h_+(2k+m) = -h_-(2k-p_+-1)
\end{equation*} 
for all $k\geq (p_++p_-)/4$. Here we used that $h_+(2k+m)=0$ if 
$k\geq (p_-+p_+)/4$, because $(p_-+p_+)/2+m=p_++(m+r)/2>p_+$.
To describe the function $\mult(2k)$ in this range, we distinguish the following 
cases for later reference. 
	 \begin{enumerate} 
\item[(a)] $p_++p_-$ and $p_++r$ are both even.
Then the function $\mult(2k)$ is zero for $k> (p_++p_-+1)/2$, equals the linear function $2k-2-p_+-p_-$ for $k_2^a=:(p_++r+1)/2\leq k\leq  (p_++p_-+1)/2:=k_2$, 
and is constant for $(p_++p_-)/4\leq k\leq (p_++r+1)/2:=k_2^b$ (so $k_2^b=k_2^a$ here).
\item[(b)] $p_++p_-$ even and $p_++r$ odd.
Then the function $\mult(2k)$ is zero for $k> (p_++p_-+1)/2$, equals the linear function $2k-2-p_+-p_-$ for $k_2^a=:(p_++r+2)/2\leq k\leq  (p_++p_-+1)/2:=k_2$, 
and is constant for $(p_++p_-)/4\leq k\leq (p_++r)/2:=k_2^b$ (so $k_2^b=k_2^a-1$ here).
\item[(c)] $p_++p_-$ odd and $p_++r$ even.
Then the function $\mult(2k)$ is zero for $k> (p_++p_-)/2$, equals the linear function $2k-2-p_+-p_-$ for $k_2^a=:(p_++r+1)/2\leq k\leq  (p_++p_-)/2:=k_2$, 
and is constant for $(p_++p_-)/4\leq k\leq (p_++r+1)/2:=k_2^b$ (so $k_2^b=k_2^a$ here).
\item[(d)] $p_++p_-$ and $p_++r$ are both odd.
Then the function $\mult(2k)$ is zero for $k> (p_++p_-)/2$, equals the linear function $2k-2-p_+-p_-$ for $k_2^a=:(p_++r+2)/2\leq k\leq  (p_++p_-)/2:=k_2$, 
and is constant for $(p_++p_-)/4\leq k\leq (p_++r)/2:=k_2^b$ (so $k_2^b=k_2^a-1$ here).
	 \end{enumerate}

We will also need to consider the largest value $k_4\in\ZZ+1/2$ such that $\mult(4k_4)\not=0$,  
assuming that this $k_4$ satisfies $k_4\geq (p_++p_-)/4$.  By \eqref{eq:multpm} we have:
\begin{equation*}
\mult(4k)=(\mult_{+,-}(4k)-h_-(4k-m))=-h_-(4k-p_+-1)
\end{equation*}
Hence the summand $\mult(4k)$ is possibly nonzero only for the smallest values of $k$ in the range
$k\geq (p_++p_-)/4$, since $h_-(4k-p_+-1)\not=0$ implies that $4k\leq p_++p_-+1$. 
Considering that $4k_4 =2 \pmod{4}$ we make this more precise:  
Let $k_4$ be the largest $k\in (\ZZ+1/2)_+$ for which $\mult(4k_4)\not=0$.  
\begin{enumerate}
\item[(i)] If $p_++p_-=1 \pmod{4}$ then $k_4=(p_++p_-+1)/4$, and  $\mathrm{mult}(4k_4)=-1$;
\item[(ii)] If $p_++p_-=2 \pmod{4}$ then $k_4=(p_++p_-)/4$, and  $\mathrm{mult}(4k_4)=-h_-(p_--1)$;
\item[(iii)] If $p_++p_-=3 \pmod{4}$ then $k_4=(p_++p_--1)/4$, (and so is not in our range); 
\item[(iv)] If $p_++p_-=0 \pmod{4}$ then $k_4=(p_++p_--2)/4$, (and so is not in our range).
\end{enumerate}
We see that always $k_2>k_4$, and $k_4$ is either not in our range or it is the smallest argument in our range.\\

A function defined on $\ZZ+1/2$ is the restriction of a unique continuous piecewise linear function on $\RR$ whose intervals of linearity have end points in $\ZZ+1/2$. In view of the listed cases (a)--(d), the following definition makes sense.
Note that in all cases (a)--(d), the linear function $2k-2-p_+-p_-$ of $k$ is negative for all $k\leq k_2$.
\begin{definition}
\begin{enumerate}
\item[(1)] For $\mult(2k)$ we distinguish  
the following points $k_2, k_2^a, k_2^b\in(\ZZ+1/2)_+$: It is zero for $k>k_2$, 
it is negative and linear on $[k_2^a,k_2]$, and it is constant again on $[k_4, k_2^b]$, with $0\leq k_2^a-k_2^b\leq 1$.

\item[(2)] Consider the linear function $l_2(k):=2k-2-p_+ - p_-$ describing $\mult(2k)$ on the interval $[k_2^a, k_2]$.  We define $k_2^c\in\ZZ+1/2$ to be the largest number such that $\mult(2k_2^c)$ is strictly larger than the value of $l_2(k_2^c)$. Then $k_2^c=k_2^a-1$ in cases (a) and (c), while $k_2^c=k_2^b$ in cases (b) and (d). 

\item[(3)] We define the \emph{$\mult(2k)$-deficit $d_2=\mult(2k_2^c)-l_2(k_2^c)>0$ at $k_2^c$} to be the difference between $\mult(2k_2^c)$ and this linear expression evaluated at $k_2^c$.  Hence the deficit is $2$ for the cases (a) and (c), and $1$ for the cases (b) and (d).
\end{enumerate}
\end{definition}

It is immediate that the condition $\cycl(k)=0$ for all $k\geq (p_++p_-)/4$ forces that $k_1:=\max\{k_1^+,k_1^-\}=k_2$.
Notice that we also have $k_1^++k_1^-=2k_2-1$ in the cases (a) and (b), and $k_1^++k_1^-=2k_2$ in the cases (c) and (d).
Hence in the cases (a) and (b) we have $|k_1^+-k_1^-|=1$, while for the cases (c) and (d) we have $k_1^+=k_1^-$.
We see that in these cases, the linear expression representing $\mult(k)$ just below $k_1$, cancels 
the linear expression of $\mult(2k)$ just below $k_2$. 
\emph{We also see that if $k_1\not=k_2$, then $k_1>k_2$, and thus 
$\cycl(k_1)= \mult(k_1)>0$}. \\

We denote by $k_1^c$ the largest argument for which there is a positive deficit  
$d_1=\ell_1(k_1^c)-\mult(k_1^c)>0$ 
between $\mult(k)$ and the linear expression $\ell_1$ valid for 
$\mult(k)$ with $k_1^c<k\leq k_1+1$. 
From the analysis below \eqref{eq:mult.}, we see that 
$k_1^c=\min (m-r,2r)-1$, and 
that $d_1=2$ if $m-r=2r$, while $d_1=1$ else. 
We now have more possibilities to distinguish. 
\begin{enumerate}
\item[(a')] As in case (a), where we impose in addition that $k_1^+=k_1^-+1=k_1=k_2$.
This implies that $k_1^\pm=p_\pm\mp r=(p_++p_-\pm 1)/2$, thus $p_+-p_-=1+2r=m-r$
(recall that, from the rectangular shape of $T_m(\rho)$, we have $p_+-m=p_--r$), 
so that $m=3r+1$. 
Hence $k_1^c=\min(m-r,2r)=2r$, with deficit $d_1=1$.
Moreover, $k_2^c=(p_++r-1)/2=(p_++p_-+4r-1)/4$ with deficit $d_2=2$. 
Notice that $k_2^c-k_1^c=(p_+-3r-1)/2=(p_+-m)\geq 0$.
\item[(a'')] As in case (a), where we impose in addition that $k_1^-=k_1^++1=k_1=k_2$.
This implies that $k_1^\pm=p_\pm\mp r=(p_++p_-\mp 1)/2$. 
Thus $p_+-p_-=-1+2r=m-r$, so that $m=3r-1$. 
Hence $k_1^c=\min(m-r,2r)=2r-1$, with deficit $d_1=1$.
Moreover, $k_2^c=(p_++r-1)/2=(p_++p_-+4r-3)/4$ with deficit $d_2=2$. 
Observe that $k_2^c-k_1^c=(p_+-3r+1)/2=(p_+-m)\geq 0$.
\item[(b')] As in case (b), where we impose in addition that $k_1^+=k_1^-+1=k_1=k_2$.
This implies that $k_1^\pm=p_\pm\mp r=(p_++p_-\pm 1)/2$, thus $p_+-p_-=1+2r=m-r$, 
so that $m=3r+1$.  In this case  $k_1^c=\min(m-r,2r)=2r$, with deficit $d_1=1$.
Moreover, $k_2^c=(p_++r)/2=(p_++p_-+4r+1)/4$ with deficit $d_2=1$.
Hence $k_2^c-k_1^c=(p_+-3r)/2=(p_+-m+1)> 0$.
\item[(b'')] As in case (b), where we impose in addition that $k_1^-=k_1^++1=k_1=k_2$.
This implies that $k_1^\pm=p_\pm\mp r=(p_++p_-\mp 1)/2$. 
Thus $p_+-p_-=1+2r=m-r$, so that $m=3r-1$. In this case $k_1^c=\min(m-r,2r)=2r-1$ with deficit $d_1=1$.
Moreover, $k_2^c=(p_++r)/2=(p_++p_-+4r-1)/4$ with deficit $d_2=1$.
Hence $k_2^c-k_1^c=(p_+-3r+1)/2=(p_+-m+2)> 0$.
\item[(c')] As in case (c), where we impose in addition that $k_1^-=k_1^+=k_1=k_2$.
This implies that $k_1^\pm=p_\pm\mp r=(p_++p_-)/2$, thus $p_+-p_-=2r=m-r$, 
so that $m=3r$.  In this case  $k_1^c=2r-1$ with deficit $d_1=2$.
Moreover, $k_2^c=(p_++r-1)/2=(p_++p_-+4r-2)/4$ with deficit $d_2=2$.
Hence $k_2^c-k_1^c=(p_+-3r+1)/2=(p_+-m+1)> 0$.
\item[(d')] As in case (d), where we impose in addition that $k_1^-=k_1^+=k_1=k_2$.
This implies that $k_1^\pm=p_\pm\mp r=(p_++p_-)/2$, thus $p_+-p_-=2r=m-r$, 
so that $m=3r$. In this case $k_1^c=2r-1$ with deficit $d_1=2$.
Moreover, $k_2^c=(p_++r)/2=(p_++p_-+4r)/4$ with deficit $d_2=1$.
Hence $k_2^c-k_1^c=(p_+-3r+2)/2=(p_+-m+2)>0$.
\end{enumerate}
This shows that in all cases, $k_2^c\geq k_1^c$, and if $k_1^c=k_2^c$ then $d_1<d_2$.
This implies that \emph{always}
$\mult(k_2^c)+\mult(2k_2^c)>0$ 
if $k_2^c$ is in our range, i.e.~if $k_2^c \geq(p_++p_-)/4$. 
Therefore, $\cycl(k)=0$ for all $k\geq  (p_++p_-)/4$ could only happen if one of the following 
possibilities hold: Either $k_2^c< (p_++p_-)/4$, or else $(p_++p_-)/4\leq k_2^c \leq k_4 $, and 
$\cycl(k)=(\mult(k_2^c)+\mult(2k_2^c))+\mult(4k_2^c)=0$. We have already mentioned above 
that $k_4\geq(p_++p_-)/4$ implies that $k_4$ is the smallest element of $\ZZ+1/2$ that satisfies 
this inequality. In other words, the second possibility mentioned is equivalent to 
$(p_++p_-)/4\leq k_2^c = k_4 $.

If $k_2^c<(p_++p_-)/4$, by the expressions of $k_2^c$ in every case above, we must be in case (a'') or (c') with $r=1/4$. In the first case  
that would imply $m<0$, which is absurd. In the case (c') we get a solution, with 
$r=1/4$, $m=3/4$, $p_+=(2n-1)+3/4$, $p_-=(2n-1)+1/4$. This is the odd rank case of Theorem \ref{prop:mainrho}(c).

Next we analyze the case $(p_++p_-)/4\leq k_2^c = k_4$. 
One reads off from the above list of cases that 
this implies that we are either in the case (i) and (d'), 
with $r=1/4$, $m=3/4$, $p_+=2n+3/4$, 
and $p_-=2n+1/4$ (here $d_2=1=\mult(k_4)$),  
which is the case of Proposition \ref{prop:mainrho} (c), 
or 
in case (ii) and (a'), with $r=1/4$ and $m=7/4$, which gives Proposition \ref{prop:mainrho} (e), or in case (ii) and (a''), with $r=3/4$ and $m=5/4$, 
which is case Proposition \ref{prop:mainrho} (d). In the latter two cases we have 
$d_2=2$, and $\mult(k_4)=2$ if $k_1^c<k_2^c$, 
while $\mult(k_4)=1=d_1$ if $k_1^c=k_2^c$. 

Observe that in the course of the proof, we also checked that if a pair $(m,\rho)$ 
does not belong to one of the cases listed in Proposition \ref{prop:mainrho} (a)--(e), 
then there exist $k\in\ZZ+1/2$ such that $k\geq(p_++p_-)/4$ and $\cycl(k)\neq 0$, the largest $k_{(m,\rho)}$ of which satisfies $\cycl(k_{(m,\rho)})>0$. This concludes the proof of all assertions.
\end{proof}

\subsection{Counting odd cyclotomic polynomials for the cases $\delta\in\{0,1\}$}\label{sec:countingdelta=0}

As we have explained in Sections \ref{sub:muSTM} and \ref{sub:redreal}, 
the statements about odd cyclotomic factors of $\mu_{0,\delta}^{\{\respt_\lambda\}}(\respt_\lambda)$ follow from the results of the previous paragraph by application of the extra-special STMs. We now translate the results to the cases $\delta\in\{0,1\}$ 
using the extra-special bijection of Section \ref{e-s}.

\begin{prop}\label{cor:delta01} 
Suppose that $\delta\in\{0,1\}$. Let $\respt$ be a real residual point of even degree, i.e.,~such that $\mu_{0,\delta}^{n, \{\respt_\lambda \} }(\respt_{\lambda})$ has no odd cyclotomic factors. 
Then $\respt$ is of the form $\respt=\respt_\lambda$ 
with $\lambda\vdash \delta+2n$ a partition with odd, distinct parts, 
where $(\delta,\lambda)$ is one of the following cases:
\begin{enumerate}
\item[(a)] Choose $m\in(\ZZ\pm1/4)_+$ and define $\kappa\in\ZZ_{\geq 0}$ and $\epsilon\in\{0,1\}$ by 
writing $m=\kappa+(2\epsilon-1)/4$. Define $\lambda=[1+2\epsilon,5+2\epsilon,\dots,4(\kappa-1)+1+2\epsilon]$, and  
$\delta\in\{0,1\}$ by $\kappa\equiv \delta\pmod{2}$. Define $n$ by $2n+\delta=2\kappa^2+(2\epsilon-1)\kappa$.
Then $\respt_\lambda$ represents a rank $0$ (extra-special) spectral transfer morphism (STM) $T_{1/4,m}^{0} \to T_{0,\delta}^n$. In particular, modulo powers of $q$ and nonzero rational constants, the residue $\mu_{0,\delta}^{n, \{\respt_\lambda\} }(\respt_\lambda)$ equals the case (v) or (vi) on the right-hand side of \eqref{eq:classicalfdeg}, and for all these cases $\mu_{0,\delta}^{n, \{ \respt_\lambda\} }(\respt_\lambda)$ indeed has no odd cyclotomic factors.

\item[(b)] Let $\delta=0$ and $r\in\ZZ_{\geq 0}$. Put $\lambda=[1,3,\dots,4r+1,4r+3]\vdash 2n$ with $n=2(r+1)^2$. Then $\respt_\lambda$ represents a rank $0$ STM $T_{0,m}^0 \to  T_{0,0}^n$ with $m=2(r+1)+\delta = 2r+2$, and 
$\mu_{0,0}^{n, \{\respt_\lambda\} } (\respt_\lambda) = (d^\DDD_a(q))^2$ with $a=r+1$.  In particular, for all these
cases $\mu_{0,0}^{n, \{\respt_\lambda \} } (\respt_\lambda)$ indeed has no odd cyclotomic factors. (We regard the empty diagram with $\delta=0$ as belonging to case (a).)

\item[(c)] Let $\delta=1$ and $r\in\ZZ_{\geq 0}$. Put $\lambda=[1,3,\dots,4r+3,4r+5]\vdash 2n+1$ with $n=2(r+1)(r+2)$. 
Then $\respt_\lambda$ represents a rank $0$ STM  $T_{0,m}^0 \to T_{0,1}^n$ with $m=2(r+1)+\delta=2r+3$, 
and $\mu_{0,1}^{n, \{\respt_\lambda\} }(\respt_\lambda)=(d^\BBB_b(q))^2$ with $b=r+1$. 
In particular, for all these cases $\mu_{0,1}^{n, \{\respt_\lambda\} } (\respt_\lambda)$ indeed has no odd cyclotomic factors.

\item[(d)] Let $\delta=1$ and $r\in\ZZ_{\geq 0}$. Put $n=8(r+1)^2-1$ and $\lambda=[3,5,7,\dots,8r+5,8r+7]\vdash 2n+1$. In this case, modulo powers of $v$ and nonzero rational 
constants, we have  
\begin{equation}
\mu_{0,1}^{n, \{\respt\}} (\respt)=\frac{1-q^{4r+4}}{1-q} \mu_{0,0}^{{n+1}, \{ \respt' \} } (\respt')=\frac{1-q^{4r+4}}{1-q}(d^\DDD_a(q))^2
\end{equation} 
where $\respt'=\respt_{\lambda'}$ with $\lambda'=[1,3,5,\dots,8r+7]$ (then 
$\respt'=(v,\respt)$), and $a=2(r+1)$. 
In particular, $\mu_{0,1}^{n, \{\respt\} } (\respt)$ has no odd cyclotomic factors if and only if $r+1=2^s$ 
for some $s \geq 0$.

\item[(e)] Let $\delta=0$ and $r\in\ZZ_{\geq 0}$. Put $n=8(r+1)^2+1$ and 
$\lambda=[1,3,5,\dots,8r+5,8r+9]\vdash 2n$. Then $\respt_\lambda=(\respt_{\lambda'},q^{4r+4})$, 
where $\lambda'=[1,3,5,\dots,8r+7]$. Since $\respt_{\lambda'}$ represents a translation STM, we see that $\respt$ is the image under the translation STM $T_{0,m}^1 \to T_{0,0}^n$ of the residual point $t=q^{4r+4}$, where $m=4(r+1)$. Consequently, modulo powers of $v$ and nonzero rational constants, we have:
\begin{equation}
\mu_{0,0}^{n, \{\respt\} }(\respt) = \frac{1-q^{4r+4}}{(1+q^{4r+4}) (1-q)}(d^\DDD_a(q))^2
\end{equation}
with $a=2(r+1)$. 
In particular, $\mu_{0,0}^{n, \{\respt\} }(\respt)$ has no odd cyclotomic factors if and only if $r+1=2^s$ for some $s\geq 0$.
\end{enumerate}
Together with the results of Proposition \ref{prop:mainrho}, the above results imply that in all cases, if a residue of the form $\mu_{0,\delta}^{n, \{\respt\} }(\respt)$ or of the form $\mu_{1/4, m}^{r, \{\respt\} }(\respt)$ contains odd cyclotomic polynomials, then the highest odd cyclotomic factor which has nonzero multiplicity in the residue in fact has \emph{positive} multiplicity.
\end{prop}

\begin{proof}
Parts (a)--(c) follow directly from the results of Proposition \ref{prop:mainrho} by application of the extra-special bijection and STMs. The expression for residues in terns of unipotent degrees follows from the relation between the parameters $\{a,b\}$ and $\{\emm, \emp \}$ and the defining properties of STMs. 

Part (d) is obtained by a direct computation, using that the residual point $\respt'$ represents 
a standard STM, and is constructed from $\respt$ by adding one coordinate component $1$.
Using this, it is an easy direct computation to check the formula given in (d). Part (e) follow from an application of a translation STM (as explained there).

The last assertion follows from the expressions given in parts (a)--(e).
\end{proof}
\begin{rem}
This finishes the proof of Proposition \ref{prop:oddred}, and thus of Proposition \ref{prop:odd}.
\end{rem}

\section{Proof of the main theorem}\label{sec:pfmain}
In this Section we will finally put everything together and prove the main result Theorem \ref{thm:A}.
In order to do so we need to analyse which residue points of the form 
$\respt_{(\lambda_-,\lambda_+)}=(\respt_{\lambda_-},\respt_{\lambda_+})$,  
with $\respt_{\lambda_\pm}$ real residue points of even degree as listed in 
Propositions \ref{cor:deltaishalf} and \ref{cor:delta01},  
yield actual solutions to \eqref{eq:fdegmu} for some cuspidal unipotent character 
$\pi$ of an inner form of $\mbf{G}$. 
Here we are using the notation 
$\respt_{(\lambda_-,\lambda_+)}=(\respt_{\lambda_-},\respt_{\lambda_+})$ 
with $\respt_{\lambda_\pm}:=\respt_{\fb,\delta_-;\lambda_-}$ 
as discussed in Section \ref{sub:muSTM}, where are assuming that the parameters 
$(\fb;\delta_-,\delta_+)$ in \eqref{eq:deg} correspond to 
a classical group $\mbf{G}$, so with $\delta_\pm\in\{0,1/2,1\}$
such that  $(\delta_-,\delta_+)\in\mathcal{V}^\mathcal{X}$, where  
$\mathcal{X}\in\{\textup{I},\textup{II},\textup{III},\textup{IV}\}$ (as defined below Eq.~\eqref{eq:sets}). 

In the course of this analysis we need to be able to compute the multiplicities of even cyclotomic polynomials in the residue 
\begin{equation}\label{eq:deg}
\mathfrak{f}_{\lambda_-,\lambda_+}:=
\mu_{\delta_-,\delta_+}^{n,\{\respt_{(\lambda_-,\lambda_+)}\}}(\respt_{(\lambda_-,\lambda_+)})
\end{equation} 

Let $\mbf{M}^e$ denote the subgroup of $\mbf{M}_0$ generated 
by $\QQ^\times$ and $\{v^j+v^{-j} \mid j \in\ZZ_+ \}$. 
For the distinguished unipotent partitions $\lambda_{\pm}$ as listed in 
Proposition \ref{cor:deltaishalf} (for $\delta_\pm=1/2$) or in Proposition \ref{cor:delta01} (for $\delta_\pm\in\{0,1\}$),
we have shown that $\mathfrak{f}_{\lambda_-,\lambda_+} \in \mbf{M}^e$. 

To proceed, we introduce some notations. 

\begin{definition}\label{def:mult}
For $f \in \mathbf{M}^e$, let 
$m_f^e:\ZZ_{+} \to\ZZ$ be the function defined by the following rule: Modulo 
the subgroup of $\mathbf{K}^\times$ generated by powers of $v$ and nonzero rational constants, we have 
\begin{equation}
f\equiv\prod_{a\in\ZZ_+}(1+q^a)^{m_f^e(a)}.
\end{equation}
We define
\begin{equation*}
\mult^{(\lambda_-,\lambda_+)}_{(\delta_-,\delta_+)} ( \cdot )=m^e_{\mathfrak{f}_{\lambda_-,\lambda_+}} (\cdot) +
\diracdelta_{1,\delta_-\delta_+}
\end{equation*}
where $\mathfrak{f}_{\lambda_-,\lambda_+}$ is given by \eqref{eq:deg}.
\end{definition}
\begin{rem}
The Kronecker delta $\diracdelta_{1,\delta_-\delta_+}$ has to be included to make the formulae in the analysis below uniform in $(\delta_-,\delta_+)$. 
The origin of this term is easily explained.
If $(\delta_-,\delta_+)=(1,1)$, then the normalization factor $\tau_{1,1}^0(1)=(v-v^{-1})^n\tau_{1,1}(1)$
of the $\mu$-function (cf.~\cite[(33)]{Opdl}) equals 
$d^{\tau,\DDD}_1=(v+v^{-1})^{-1}$, while in all other cases $\tau_{1,1}^0(1)$ is equal to $1$.  
This is due to the one-dimensional $K$-anisotropic factor 
in a maximally $K$-split $K$-torus in the case $(\delta_-,\delta_+)=(1,1)$.
This factor gives rise to an additional factor $(v+v^{-1})$ in the volume of Iwahori 
subgroup of $G$ in this case. 
\end{rem}

In general, let $\delta\in\{0,1/2,1\}$ and let $\lambda\vdash 2n+\lfloor \delta\rfloor$ be a distinguished unipotent partition of type 
$\delta$, i.e. $\lambda$ has distinct parts of parity type $1-2\delta$.
Let $\respt_{\fb,\delta;\lambda}$ be the corresponding residual point with positive coordinates. 

\begin{definition}\label{def:h} 
For $x\in\QQ$, 
let $\tih_{\fb; \lambda}(x)\in\ZZ_{\geq0}$ denote the number of coordinates of $\respt:=\respt_{\fb,\delta;\lambda}$ 
equal to $q^{ x}$ or $q^{-x}$. 
If $\fb=1$ then 
we will usually drop it from the notation and simply write $\tih_\lambda$. Define $\Delta_y(\tih_\lambda)(x)=2\tih_\lambda(x)-\tih_\lambda(x-y)-\tih_\lambda(x+y)$. If $y=1$ we will often drop it from the notation, and simply write $\Delta$.
\end{definition}

Note that $\tih_{\fb; \lambda}(x)$ is independent of the choice of $\respt$ in its $W_0$-orbit.
We have $\tih_{\fb;\lambda}(\fb x)=\tih_{1;\lambda}(x)$, and 
$\tih_{\fb;\lambda}$ is an \emph{even} function supported on $\fb (\ZZ+\delta)$. Observe that $\tih_{\fb; \lambda}(0)$ equals \emph{twice} the number of coordinates of $\respt$ which are equal to $1$. \\

By an easy direct computations one sees that if $\delta\in\{0,1\}$ then, 
modulo powers of $q$ and nonzero rational constant factors, we have
\begin{equation}\label{eq:reldelta01}
\mu_{1,\delta}^{n, \{\respt\} } (\respt)=\prod_{x\in\ZZ_+}(1+q^x)^{\Delta_1(h_\lambda)(x)} \mu_{0,\delta}^{n, \{\respt\} } (\respt).
\end{equation}
This can be translated using notations in Definition \ref{def:h} (recall that $\delta$ also denotes the unique partition of 
$\delta\in\{0,1\}$)
\begin{equation}\label{eq:reldelta01add}
\mult^{(\delta_-,\lambda)}_{(\delta_-,\delta)} (\cdot )=\mult^{(0,\lambda)}_{(0,\delta)} (\cdot) +
\Delta_{\delta_-}(\tih_\lambda)(x).
\end{equation}

To give an explicit expression of $\mult_{(0, \delta)}^{(0, \lambda)}$, we introduce more notations. Recall the algebra of functions defined on $\QQ$ with finite support, with multiplication given by the Cauchy convolution $m_1*m_2(z)=\sum_{x+y=z}m_1(x)m_2(y)$. This is a commutative, unital associative algebra. We use the Dirac delta $\diracdelta_x$ to denote the basis element corresponding to $x\in\QQ$, so that $\diracdelta_0$ is the identity of convolution\footnote{Not to be confused with $\dem, \dep$ and $\delta$ indicating the parity type of a partition, which are not written upright. We apologize for this choice of notation.}. If $T_{x}m$ is the left translation of $m$, i.e.~$(T_{x}m)(y)=m(y-x)$, then $T_x m=\diracdelta_x*m$. We have the obvious rules $\Delta(m_1*m_2)=\Delta(m_1)*m_2=m_1*\Delta(m_2)$, also $\Delta(m)=\Delta(\diracdelta_0)*m$. 

\begin{prop}\label{prop:degwithdeltamin1}
Using the notations of Proposition \ref{cor:delta01}, we have the followings: 
\begin{enumerate}
\item[(a)] Let $m\in(\ZZ\pm 1/4)_+$, written as usual as $m=\kappa+(2\epsilon-1)/4$. Let $\delta\in\{0,1\}$ be the 
parity of $\kappa$. Put $p=2(\kappa-1)+\epsilon$, the largest jump in the jump sequence of $\lambda$. Thus 
$\lambda=\lambda_p$ and $\delta$ are both determined by $p$, and for $k\in\ZZ$ we have:
\begin{equation*}
\tih_\lambda(k)=\left\{
\begin{array}{ll}
2\lfloor (p+2)/4\rfloor &\text{\ if\ }k=0,\\
\max\{0,\lfloor (p+2-|k|)/2\rfloor\} &\text{\ else.\ }
\end{array}
\right.
\end{equation*}
We have $\mult^{(0,\lambda)}_{(0,\delta)}=-\tih_\lambda$, and 
\begin{equation*}
\Delta(\tih_\lambda)+\Delta_\delta(\diracdelta_0) + S_{p+1} = \diracdelta_0,
\end{equation*}
where 
\begin{equation*}
S_{p+1}(k)=\left\{
\begin{array}{ll}
(-1)^{p+1-k}&\text{\ if\ } |k|\leq p+1,\\
0&\text{\ else.\ }
\end{array}
\right.
\end{equation*}

\item[(b),(c)] The cases (b) ($\delta=0$) and (c) ($\delta=1$) can be treated uniformly. Let $p=2r+1+\delta$,  thus 
$\lambda=\lambda_p$ and $\delta$ are both determined by $p$, and for $k\in\ZZ$ we have:
\begin{equation*}
\tih_\lambda(k)=\left\{
\begin{array}{ll}
2\lfloor (p+1-\delta)/2\rfloor=p+1-\delta &\text{\ if\ }k=0,\\
\max\{0,p+1-|k|\}&\text{\ else.\ }
\end{array}
\right.
\end{equation*}
We have $\mult^{(0,\lambda)}_{(0,\delta)}=-2\tih_\lambda$, and 
\begin{equation*}
\Delta(\tih_\lambda)+\Delta_\delta(\diracdelta_0)=\Delta_{(p+1)}(\diracdelta_0)=-\diracdelta_{p+1}-\delta_{-p-1}+2\diracdelta_0.
\end{equation*}

\item[(d)] 
In this case $\delta=1$ and $p=2a+1$, with $a=2(r+1)$ a power of $2$. Hence 
$\lambda=\lambda_p$ is determined by $p$, and for $k\in\ZZ$ we have:
\begin{equation*}
\tih_\lambda(k)=\left\{
\begin{array}{ll}
p-1 &\text{\ if\ }k=0,\\
\max\{0,p+1-|k|\}&\text{\ else.\ }
\end{array}
\right.
\end{equation*}
We have $\mult^{(0,\lambda)}_{(0,\delta)}=-2\tih_\lambda+\diracdelta_1+\diracdelta_2+\diracdelta_4+\ldots+\diracdelta_a$, and 
\begin{equation*}
\Delta(\tih_\lambda)+\Delta_\delta(\diracdelta_0)=-\diracdelta_{p+1}-\diracdelta_{-p-1}+\diracdelta_{-1}+\diracdelta_{+1}.
\end{equation*}

\item[(e)] 
In this case $\delta=0$ and $p=2a$, with $a=2(r+1)$ a power of $2$. Hence 
$\lambda=\lambda_p$ is determined by $p$, and for $k\in\ZZ$ we have:
\begin{equation*}
\tih_\lambda(k)=\left\{
\begin{array}{ll}
p &\text{\ if\ }k=0,\\
\max\{0,p-|k|\}&\text{\ if\ }0<|k|<p,\\
1&\text{\ if\ }|k|=p,\\
0&\text{\ else.\ }
\end{array}
\right.
\end{equation*}
We have $\mult^{(0,\lambda)}_{(0,\delta)}=-2\tih_\lambda+\diracdelta_1+\diracdelta_2+\diracdelta_4+\ldots+\diracdelta_p$, and 
\begin{equation*}
\Delta(\tih_\lambda)=(-\diracdelta_{p+1}+\diracdelta_{p}-\diracdelta_{p-1})+(-\diracdelta_{-p+1}+\diracdelta_{-p}-\diracdelta_{-p-1})+2\diracdelta_{0}.
\end{equation*}
\end{enumerate}
\end{prop}

\begin{proof} 
We use the rules (cf.~\cite[Proposition 6.6]{OpdSol})   
describing a linear residual point $\xi$ for $\BBB_n$ with 
parameter $\delta\in\{0,1\}$ and $\mathsf{k}=\log q$,
in terms of a distinguished $\delta$-partition $\lambda$. 
The number $m_0$ of coordinates equal to $0$ is equal to $\lfloor (m_1+1-\delta)/2\rfloor$, where 
$m_1$ is the number of coordinates equal to $\pm k$. In terms of $\tih_\lambda$, this means that $\tih_\lambda(0)=\lfloor (\tih_\lambda(1)+1-\delta)/2\rfloor$. 
For positive integers $n$, $\tih_\lambda(n)$ equals the number of \emph{jumps} in the jump sequence $\jmath=(\lambda-\mathbf{1})/2$ which are larger than or equal to $n$. 
Since we extended the function $\tih_\lambda$ as an even function, this shows that $\tih_\lambda$ is 
determined by $(\delta,\lambda)$ and can be read of simply from the information provided in Proposition \ref{cor:delta01}. 
\end{proof}

A general residual point is of the form  
$\respt=(-\respt_{(\fb,\delta_-;\lambda_-)},\respt_{(\fb,\delta_+;\lambda_+)})$. Combining the results of Proposition \ref{cor:deltaishalf} (for $\delta_\pm=1/2$), Proposition \ref{cor:delta01} (for $\delta_\pm\in\{0,1\}$) and Proposition \ref{prop:degwithdeltamin1}, 
the ``residue'' $\mu_{\dem, \dep}^{n, \{\respt\} } (\respt)$ or equivalently the 
multiplicity $\mult^{(\lambda_-,\lambda_+)}_{(\fb;\delta_-,\delta_+)}$ is easy to determine if one of $\lambda_\pm$  
is the unique partition $\delta_\pm$ of $\delta_\pm$. 
When both $\lambda_{\pm}\not=\delta_\pm$ we also 
need to take into account the ``mixed" factors of this residue by looking at the expression 
\eqref{eq:goodwill}. 
These ``mixed'' factors arise from roots of the type $t_i^{\pm 1} t_j^{\pm 1}$ with $t_i^{\pm 1}$ 
comes from $-\respt_{(\fb,\delta_-;\lambda_-)}$ and $t_j^{\pm 1}$ comes from $\respt_{(\fb,\delta_+;\lambda_+)})$. In general, we have the following
\begin{prop}\label{prop:mixed} 
Suppose that $\fb\in\{1,2\}$. For all $k\in(\ZZ/\fb)_+$ we have: 
\begin{equation*}
\mult^{(\lambda_-,\lambda_+)}_{(\fb;\delta_-,\delta_+)}(\fb k)=\Delta(\tih_{\lambda_-}*\tih_{\lambda_+})(k)+
\mult^{(0,\lambda_+)}_{(\fb;\delta_-,\delta_+)}(\fb k)+\mult^{(0,\lambda_-)}_{(\fb;\delta_+,\delta_-)}(\fb k)
\end{equation*}

It is also easy to see that we always have: 
\begin{equation*}
\mult^{(0,\lambda_\pm)}_{(\fb;\delta_\mp,\delta_\pm)}(\fb k)
=(\Delta_{\delta_\mp}(\delta_0)*\tih_{\lambda_\pm})(k)+\mult^{(0,\lambda_\pm)}_{(\fb;0,\delta_\pm)}(\fb k)
\end{equation*}
\end{prop}
Combined with the results of Propositions \ref{cor:deltaishalf}, \ref{cor:delta01} and \ref{prop:degwithdeltamin1},
this formula enables us to express $\mult^{(0,\lambda_\pm)}_{(\fb;\delta_\mp,\delta_\pm)}(\fb k)$ 
in terms of $\tih_{\lambda_+}$ and $\tih_{\lambda_-}$.\\

We finally have enough tools to prove the main theorem, case by case. 

\subsection{Proof of the main Theorem for odd orthogonal groups}\label{sub:oddorth}
We consider $G=\mbf{SO}_{2n+1}$. This is of parameter type II, and this is by far the simplest case. 

\begin{proof}[Proof of Theorem \ref{thm:A} for $\mbf{SO}_{2n+1}$]
We have $\fdeg(\sigma, q)=d_a^{\DDD}(q) d_b^{\BBB}(q)$ (see \eqref{eq:classicalfdeg}), and a glance at the tables in \cite[\S 13.7]{Carter} readily shows that the multiplicity functions $m_{a,b}^e$ for  $d_{a,b}:=d_a^{\DDD}(q) d_b^{\BBB}(q)$ is:
\begin{equation}
m_{a,b}^e(k)=-(\max\{0,2a-k\}+\max\{0,2b+1-k\}), \quad k \in\ZZ_+.
\end{equation}
On the other hand, we have $\delta_\pm=1/2$ in this case, and $\fb=1$.
Proposition \ref{cor:deltaishalf} implies that $\lambda_\pm=[2,4,\dots,2r_\pm]$, 
and that the corresponding multiplicity functions $\tih_\pm$ are supported on 
$\ZZ+1/2$ and are given by $h_\pm(x)=\max\{0,p_\pm+1-|x|\}$ where $p_\pm=r_\pm-1/2=m_\pm-1$. 

\begin{rem}
We remind the reader that from here onwards we are using the notation $\lambda_\pm$, $m_\pm$, $p_\pm$ etc.~in the sense of Sections \ref{sec:DULLPCG} and \ref{sub:muSTM}, 
but \emph{not} as in Section \ref{sub:oddextra}. 
\end{rem}

Applying Lemma \ref{lem:deltahalf}, Propositions \ref{prop:degwithdeltamin1} and \ref{cor:deltaishalf} ,we have (with $\tih_\pm:=\tih_{\lambda_\pm}$ and $k\in\ZZ_+$):
\begin{equation*}
\begin{aligned}
&\mult^{(\lambda_-,\lambda_+)}_{(1/2,1/2)}(k)\\
=&\Delta(\tih_-*\tih_+)(k)
-\tih_-(k-1/2)-\tih_-(k+1/2)-\tih_+(k-1/2)-\tih_+(k+1/2)\\
=&(\Delta(\tih_-)*\tih_+)(k)
-\tih_-(k-1/2)-\tih_-(k+1/2)-\tih_+(k-1/2)-\tih_+(k+1/2)\\
=&-T_{-(p_-+1)}\tih_+-T_{+(p_-+1)}\tih_+
-\tih_-(k-1/2)-\tih_-(k+1/2)\\
=&-(\max\{0,p_-+p_++2-k\}+\max\{0,|p_+-p_-|-k\})
\end{aligned}
\end{equation*}
so that necessarily 
$\{2a,2b+1\}=\{p_++p_-+2,|p_+-p_-|\}=\{m_-+m_+,|m_+-m_-|\}$, as was to be proved (compare with \eqref{eq:sets}, and using the STM 
$\phi_T: T_r\to T_{r+m_+-1/2}$).
\end{proof}

It is instructive to carefully analyze the last step in the above computation.
If $k$ is sufficiently large then $m_{a,b}^e(k)$ and 
$\mult^{(\lambda_-,\lambda_+)}_{(1/2,1/2)}(k)$ are zero. 
Decreasing from $k$ to $1$ we meet three remarkable arguments 
of $\mult^{(\lambda_-,\lambda_+)}_{(1/2,1/2)}$, 
namely $k_1=p_-+p_++1$ (this is the first argument 
where $\mult^{(\lambda_-,\lambda_+)}_{(1/2,1/2)}(k) \neq 0$), 
$k_2=p_-+1/2$ (this is the first argument after $T_{p_-+1}\tih_+(k)$ reached its maximum) 
and $k_3=|p_+-p_-|-1$ (this is the argument where either $T_{-(p_-+1)}\tih_+$ starts increasing (if $p_-<p_+$) 
or where $T_{(p_-+1)}\tih_+$ stops decreasing (if $p_-> p_+$)). 
Now compare to $m_{a,b}^e(k)$.
We see that $k_1=\max\{2a-1,2b\}$. 
At $k_2$, we see that $-\tih_-(k-1/2)-\tih_-(k+1/2)$ exactly repairs the 
linear behaviour of $T_{-(p_-+1)}\tih_+(k)$ at $k_2<k\leq k_1$. 
Finally, for $k\leq k_3$ we obtain (in either case) an additional linear contribution, implying that $k_3=\min \{2a-1,2b\}$.

\subsection{Proof of the main Theorem for unitary groups}\label{sub:unit}

In this subsection we consider the case $G=\mbf{SU}_{2n}$ or 
$G=\mbf{SU}_{2n+1}$, which are in type $\textup{I}$. 

\begin{proof}[Proof of Theorem \ref{thm:A} for unitary groups]

We now have $\fdeg(\sigma, q)=d_s^{\{ {}^2\AAA \}}(q) d_t^{\{ {}^2\AAA \}}(q):=d_{s,t}$ (see \eqref{eq:classicalfdeg}). A glance at the tables in \cite[\S 13.7]{Carter} readily shows that the multiplicity function $m_{s,t}^e$ for $d_{s,t}$ is supported 
on odd integers, and is of the form:
\begin{equation}
m_{s,t}^e(2k)=-(\max\{0,s+1/2-k\}+\max\{0,t+1/2-k\}), \quad k \in(\ZZ+1/2)_+.
\end{equation}
On the other hand, we have $\delta_-=1/2$ and $\delta_+\in\{0,1\}$ in this case, 
and $\fb=2$, so 
Proposition \ref{cor:deltaishalf} implies that $\lambda_-=[2,4,\dots,2r_-]$, and 
$\lambda_+$ is as described in Proposition \ref{prop:degwithdeltamin1}. 
The corresponding multiplicity function $\tih_-$ is supported on $\ZZ+1/2$, and 
$\tih_+$ is supported on $\ZZ$. 
We know we can write $h_-(x)=\max\{0,p_-+1-|x|\}$ where $p_-=\emm -1$ and 
$\emm=r_-+1/2$, 
but at this point there are still various choices for $h_+=h_{\lambda_+}$ possible.
Applying Propositions \ref{prop:degwithdeltamin1} and \ref{cor:deltaishalf}, 
we have (with $\tih_\pm:=\tih_{\lambda_\pm}$ and $k\in(\ZZ/2)_+$):
\begin{equation*}\begin{aligned}
&\mult^{(\lambda_-,\lambda_+)}_{2;(1/2,\delta)}(2k)\\
=&\Delta(\tih_-*\tih_+)(k)
-2\tih_-(k)+\Delta_{\delta_+}(\tih_-)(k)+\mult^{(0,\lambda_+)}_{2;(1/2,\delta_+)}(2k)\\
=&\big( (\Delta(\tih_-)+\Delta_{1/2}(\diracdelta_0)\big)*\tih_+)(k)
-2\tih_-(k)+\Delta_{\delta_+}(\tih_-)(k)+\mult^{(0,\lambda_+)}_{2;(0,\delta_+)}(2k)\\
=&-T_{(p_-+1)}\tih_+(k)-T_{-(p_-+1)}\tih_+(k)
-2\tih_-(k)+\Delta_{\delta_+}(\tih_-)(k)+2\tih_+(k)+\mult^{(0,\lambda_+)}_{(0,\delta_+)}(k).
\end{aligned}\end{equation*}
Observe that $m_{s,t}^e(2k) \neq 0 \Rightarrow k\in\ZZ+1/2$. This implies that for all $k\in\ZZ$ 
we must have $2\tih_+(k)+\mult^{(0,\lambda_+)}_{(0,\delta_+)}(k)=0$, since all other terms in the expression are obviously equal to zero for $k\in\ZZ$. 
Checking the cases in Proposition \ref{prop:degwithdeltamin1} 
we see that this happens if and only if $\lambda_+$ 
is of the form described in cases (b) or (c).  
Hence we have $\lambda_+=[1,3,\dots,1+2p_+]$ with $p_+=m_+-1$ 
and $m_+=2r_++2+\delta_+$, and 
\begin{equation*}
\mult^{(\lambda_-,\lambda_+)}_{2;(1/2,\delta)}(2k)=
-T_{(p_-+1)}\tih_+(k)-T_{-(p_-+1)}\tih_+(k)
-2\tih_-(k)+\Delta_{\delta_+}(\tih_-)(k)
\end{equation*}
From here on the reasoning is completely analogous to the odd orthogonal case Section \ref{sub:oddorth}.
We see that necessarily 
$\{s-1/2,t-1/2\}=\{p_++p_-+1,|p_+-p_-|-1\}$ or $\{s+1/2,t+1/2\}=\{m_-+m_+,|m_+-m_-|\}$, as was to be proved (compare with \eqref{eq:sets}, and with Section \ref{sub:muSTM}). 
\end{proof}

\subsection{Proof of the main Theorem for symplectic and even orthogonal groups}\label{subsec:noteasy}

In this subsection we treat the cases that $\mbf{G}=\mathbf{Sp}_{2n}$, $\mathbf{SO}_{2n}$ or $\mathbf{SO}^*_{2n+2}$, 
which are of type \textup{III} (the first group) or of type \textup{IV} (the last two groups). \\

Before we start, let us recall the smallest unipotent partitions $\lambda$ which 
belong to the cases (a)--(e) of Corollary \ref{cor:delta01} and Proposition \ref{prop:degwithdeltamin1}. We always denote 
by $p$ the largest element of the support of the function $\tih_\lambda$, and we formally extend this 
to cases where $\lambda=\delta$ is the unique partition of $\delta$ (compatible with our formulae for residues). If $\delta=0$ and $\lambda=0$ then we put $p=-1$, and if $\delta=1$ and $\lambda=1$ we put $p=0$. 
These can be viewed as in case (a), or (b) (if $\delta=0$) or (c) (if $\delta=1$) respectively.  
If $p=1$ then $\lambda=[3]$, and this is a supercuspidal partition belonging to case (a), or $\lambda=[1,3]$, which is case (b). 
If $p=2$ then $\lambda=[1,5]$, which belongs to case (a), or $\lambda=[1,3,5]$, which belongs to case (c).
The smallest partition in case (d) is $[3,5,7]$ (with $p=3$), and the smallest partition in case (e) is $[1,3,5,9]$ (with $p=4$).

\begin{proof}[Proof of Theorem \ref{thm:A}]

Now the formal degree is of the form
\begin{equation*}
\fdeg(\sigma, q)=d_a^{\BBB}(q) d_b^{\BBB}(q) \text{ or } \fdeg(\sigma, q)=d_a^{\BBB}(q^2)d_b^{ \{ {}^2\AAA\}}(q)
\end{equation*}
for symplectic groups, or of the form 
\begin{equation*} 
\fdeg(\sigma, q)=d_a^{\DDD}(q) d_b^{\DDD}(q) \text{ or } \fdeg(\sigma, q)=d_b^{\DDD}(q^2)d_a^{ \{ {}^2\AAA \}}(q)
\end{equation*} 
for even orthogonal groups. A glance at the tables 
in \cite[\S 13.7]{Carter} readily shows that the multiplicity function $m_{a,b}^{e,ord,\textup{III}}$ for $d_{a,b}^{ord,\textup{III}}:=d_a^{\BBB}(q) d_b^{\BBB}(q)$
is of the form:
\begin{equation}
m_{a,b}^{e,ord,\textup{III}}(k)=-(\max\{0,2a+1-k\}+\max\{0,2b+1-k\}),\quad k \in\ZZ_+,
\end{equation}
while $m_{a,b}^{e,extra,\textup{III}}$ for $d_{a,b}^{extra,\textup{III}}:=d_a^{\BBB}(q^2) d_b^{\{ {}^2\AAA \}}(q)$ is of the form 
\begin{equation}
m_{a,b}^{extra,\textup{III}}(k)=-(\textup{max}_{int}\{0,(4a+2-k)/2\}+\textup{max}_{int}\{0,(2b+1-k)/2\}), \; k \in\ZZ_+,
\end{equation}
where $\textup{max}_{int}$ is {\bf{defined}} as the \emph{largest integer} of a finite set of rational numbers, or $0$ if the set contains no integer.

Similarly, for the parameter type $\textup{IV}$, the multiplicity function $m_{a,b}^{e,ord,\textup{IV}}$ for $d_{a,b}^{ord,\textup{IV}}:=d_a^{\DDD}(q) d_b^{\DDD}(q)$ is of the form:
\begin{equation}
m_{a,b}^{e,ord,\textup{IV}}(k)=-(\max\{0,2a-k\}+\max\{0,2b-k\}), \quad k \in\ZZ_+,
\end{equation}
while $m_{a,b}^{e,extra,\textup{IV}}$ for  $d_{a,b}^{extra,\textup{IV}}:=d_a^{\DDD}(q^2) d_b^{\{ {}^2\AAA \}}(q)$ is of the form 
\begin{equation}
m_{a,b}^{extra,\textup{IV}}(k)=-(\textup{max}_{int}\{0,(4a-k)/2\}+\textup{max}_{int}\{0,(2b+1-k)/2\}),\; k \in\ZZ_+.
\end{equation}
If $\respt=(-\respt_{\lambda_-},\respt_{\lambda_+})$ is a solution of \eqref{eq:fdegmu} then 
$\lambda_\pm$ must both belong to the cases listed in Proposition \ref{cor:delta01}, 
with $\delta_\pm$ determined by $G$. Using Propositions \ref{prop:degwithdeltamin1} and \ref{prop:mixed}, we will compute the multiplicity function 
$\mult^{(\lambda_-,\lambda_+)}_{(\delta_-,\delta_+)}(k)$ where $\lambda_\pm$ belongs to one of these  
cases (a)--(e) as listed in Proposition \ref{cor:delta01}, to study whether such combinations could provide solutions of equation \eqref{eq:fdegmu}. \\

(1) Let us first assume that $\lambda_\pm$ both belong to case (b) (if $\delta_\pm=0$) or to case (c) (if $\delta_\pm=1$). 
Note that it is strictly speaking not necessary to treat these cases, since we already know that they all represent standard STMs. Yet it is useful and instructive to do the calculations.

In the present range of parameters we need to compensate for the added factor 
$\delta_-\delta_+\diracdelta_1(k)$ 
in the multiplicity function $\mult^{(\lambda_-,\lambda_+)}_{(\delta_-,\delta_+)}(k)$
(see definition \ref{def:mult}).

Thus we have, using Proposition \ref{prop:degwithdeltamin1},
and analogous to the odd orthogonal cases, for all $k\in\ZZ_+$: 
\begin{align*}
&\mult^{(\lambda_-,\lambda_+)}_{(\delta_-,\delta_+)}(k)-\delta_-\delta_+\diracdelta_1(k)\\
=&\Delta(\tih_-*\tih_+)(k)
-2\tih_-(k)+\Delta_{\delta_+}(\tih_-)(k)
-2\tih_+(k)+\Delta_{\delta_-}(\tih_+)(k)
-\delta_-\delta_+\diracdelta_1(k)\\
=&\Delta(\tih_-*\tih_+)(k)
-2\tih_-(k)+\Delta_{\delta_+}(\tih_-)(k)
-2\tih_+(k)+(\Delta_{\delta_-}(\delta_0)*\tih_+)(k)
-\delta_-\delta_+\diracdelta_1(k)\\
=& \big( (\Delta(\tih_-)+\Delta_{\delta_-}(\delta_0) )*\tih_+ \big)(k)
-2\tih_+(k)-2\tih_-(k)+\Delta_{\delta_+}(\tih_-)(k)
-\delta_-\delta_+\diracdelta_1(k)\\
=&-T_{-(p_-+1)}\tih_+(k)-T_{+(p_-+1)}\tih_+(k)
-2\tih_-(k)-\delta_+\diracdelta_{(p_-+1)}(k)+\delta_+\delta_-\diracdelta_1(k)
-\delta_-\delta_+\diracdelta_1(k)\\
=&-T_{-(p_-+1)}\tih_+(k)-T_{+(p_-+1)}\tih_+(k)
-2\tih_-(k)-\delta_+\diracdelta_{(p_-+1)}(k)\\
=&-(\max\{0,p_-+p_++2-k\}+\max\{0,|p_+-p_-|-k\})
\end{align*}
so that this could be a solution of \eqref{eq:fdegmu}, but clearly only if the right-hand side 
is of the form $m_{a,b}^{e,ord,\textup{III}}(k)$ (if the parities of $p_-$ and $p_+$ are distinct)
or of the form $m_{a,b}^{e,ord,\textup{IV}}(k)$ (if the parities of $p_-$ and $p_+$ are equal).
In the first case we see that $\respt$ is a solution of \eqref{eq:fdegmu} if and only if 
$\{2a+1,2b+1\}=\{p_++p_-+2,|p_+-p_-|\}=\{m_-+m_+,|m_+-m_-|\}$, 
and in the second case if and only if
$\{2a,2b\}=\{p_++p_-+2,|p_+-p_-|\}=\{m_-+m_+,|m_+-m_-|\}$.
These solutions correspond exactly to the ordinary standard translation STMs
\cite[3.2.6]{Opdl}
 (compare with \eqref{eq:sets}, and with Section \ref{sub:muSTM}). \\

(2) The second case we will consider is that of the cuspidal extra-special STMs, that is, when both $\lambda_\pm$ belong to case (a).
(Again, strictly speaking this is not necessary, since all cases are known to represent cuspidal extra-special STMs.)
Using Propositions \ref{prop:degwithdeltamin1} and \ref{prop:mixed}, we compute, for $k\in\ZZ_+$:
\begin{align*}
&\mult^{(\lambda_-,\lambda_+)}_{(\delta_-,\delta_+)}(k)-\delta_-\delta_+\diracdelta_1(k)\\
=&\Delta(\tih_-*\tih_+)(k)
-\tih_-(k)+\Delta_{\delta_+}(\tih_-)(k)
-\tih_+(k)+\Delta_{\delta_-}(\tih_+)(k)
-\delta_-\delta_+\diracdelta_1(k)\\
=&(\Delta(\tih_-)*\tih_+)(k)
-\tih_-(k)+\Delta_{\delta_+}(\tih_-)(k)
-\tih_+(k)+(\Delta_{\delta_-}(\diracdelta_0)*\tih_+)(k)
-\delta_-\delta_+\diracdelta_1(k)\\
=&-(S_{p_-+1}*\tih_+)(k)
-\tih_-(k)+\delta_+(-S_{p_-+1}-\Delta_{\delta_-}(\diracdelta_0)+\diracdelta_0)(k)
-\delta_-\delta_+\diracdelta_1(k)\\
=&-(S_{p_-+1}*\tih_+)(k)
-\tih_-(k)+\delta_+(-S_{p_-+1}+\delta_-\diracdelta_1)(k)
-\delta_-\delta_+\diracdelta_1(k)\\
=&-(S_{p_-+1}*\tih_+)(k)
-\tih_-(k)-\delta_+S_{p_-+1}(k)
\end{align*}
Now observe that there exists a unique symmetric function $\tiH_+$ supported 
on $\ZZ+1/2$ and a remainder $R_{m_+}$ 
such that we have $\tih_+=(\diracdelta_{-1/2}+\diracdelta_{+1/2})*\tiH_++R_{\emp}\diracdelta_0$.
One checks easily that in fact $R_{\emp}=(-1)^{\epsilon_+}2(\lfloor \frac{p_++2}{4}\rfloor+\epsilon_+\delta_+)$
and $\tiH_+(k)=\textup{max}_{int}\{0,(2p_+-2|k|+3)/4\}$. 
We may continue the computation as follows:
\begin{align*}
&\mult^{(\lambda_-,\lambda_+)}_{(\delta_-,\delta_+)}(k)-\delta_-\delta_+\diracdelta_1(k)\\
=&-(S_{p_-+1}*\tih_+)(k)
-\tih_-(k)-\delta_+S_{p_-+1}(k)\\
=&-S_{p_-+1}*(\diracdelta_{-1/2}+\diracdelta_{+1/2})*\tiH_+(k)-R_{m_+}S_{p_-+1}(k)
-\tih_-(k)-\delta_+S_{p_-+1}(k)\\
=&-T_{p_-+3/2}\tiH_+(k)-T_{-(p_-+3/2)}\tiH_+(k)-(R_{m_+}+\delta_+)S_{p_-+1}(k)
-\tih_-(k)\\
=&-T_{p_-+3/2}\tiH_+(k)-T_{-(p_-+3/2)}\tiH_+(k)-(-1)^{\epsilon_+}2(\lfloor \frac{p_++2}{4}\rfloor+\delta_+)S_{p_-+1}(k)
-\tih_-(k)\\
=&\textup{max}_{int}\{0,(p_-+p_++3-k)/2\}+\textup{max}_{int}\{0,(|p_--p_+|-k)/2\}
\end{align*}
where the last line follows from carefully checking what happens at $k=p_-+2,\,p_-+1,\,p_-,\dots$ and at $|p_++p_-|-1,\,|p_++p_-|-2,\dots$ 
at the various congruence classes of $p_+$ modulo $4$.
Now use that in case (a), we have $p_\pm=2(\kappa_\pm-1)+\epsilon=2m_\pm-3/2$. Hence we see that 
$\respt$ is a solution of \eqref{eq:fdegmu} if and only if the right-hand side is of type $d_{a,b}^{extra,\textup{III}}$, in which case 
we need that $\{m_-+m_+, |m_--m_+|\}=\{2a+1,b+1/2\}$, or of type $d_{a,b}^{extra,\textup{IV}}$, in which case 
we need $\{m_-+m_+, |m_--m_+|\}=\{2a,b+1/2\}$. These are indeed the residual points representing the supercuspidal extra-special STMs.\\

In fact, these are the only possibilities. In what follows we will see that there is no other solution for \eqref{eq:fdegmu}.\\
 
(3) Let us now try to find solutions of \eqref{eq:fdegmu} with $\lambda_-$ is of type (b) or (c), and $\lambda_+$ of type (a).
We allow any pair $(\delta_-,\delta_+)\in\{0,1\}^2$ here. Similar to the previous computations, we have:
\begin{align*}
&\mult^{(\lambda_-,\lambda_+)}_{(\delta_-,\delta_+)}(k)-\delta_-\delta_+\diracdelta_1(k)\\
=&\Delta(\tih_-*\tih_+)(k)
-2\tih_-(k)+\Delta_{\delta_+}(\tih_-)(k)
-\tih_+(k)+\Delta_{\delta_-}(\tih_+)(k)
-\delta_-\delta_+\diracdelta_1(k)\\
=&((\Delta(\tih_-)+\Delta_{\delta_-}(\diracdelta_0))*\tih_+)(k)
-2\tih_-(k)+\Delta_{\delta_+}(\tih_-)(k)
-\tih_+(k)
-\delta_-\delta_+\diracdelta_1(k)\\
=&-T_{-(p_-+1)}\tih_+(k)-T_{p_-+1}\tih_+(k)
+\tih_+(k)
-2\tih_-(k)+\Delta_{\delta_+}(\tih_-)(k)
-\delta_-\delta_+\diracdelta_1(k)\\
=&-T_{p_-+1}\tih_+(k)+\tih_+(k)-2\tih_-(k)-T_{-(p_-+1)}\tih_+(k)
-\delta_+\diracdelta_{p_-+1}(k)
\end{align*}

If both $p_-,p_+\geq 1$ then the above expression has value $-1$ at $k=p_-+p_++1$ and value $-1$ at 
$k=p_-+p_+$, and from the form of the available cuspidal formal degrees we easily conclude that if 
$\respt$ solves \eqref{eq:fdegmu} in this case, then the right-hand side has to be of the form $d_{a,b}^{extra,\textup{III}}$,
or $d_{a,b}^{extra,\textup{IV}}$ where $a$ and $b$ are such that one of 
the corresponding extra-special parameters $(m_-^e,m_+^e)$ equals $1/4$. Indeed, if $\lambda_-^e,\lambda_+^e$ 
are the unipotent partitions of type (a) such that the image of the corresponding 
extra-special STM $\respt_{(\delta_-,\delta_+)}^{(\lambda^e_-,\lambda^e_+)}$ also 
solves of \eqref{eq:fdegmu} (with the same right-hand side),  
then the values of the two highest multiplicities we just discussed, imply 
that $p^e_++p^e_-+1$ and $|p^e_+-p^e_-|-2$ differ by one (see the formula for the multiplicities of the residue we derived 
in the case where $\lambda_-$ and $\lambda_+$ are both type (a)), from which we conclude that 
one of $p^e_\pm$ equals $-1$, or equivalently $m^e_\pm=1/4$. Let us say that $p_\pm^e=-1$, and then \textbf{define} $p_{max}^e=p_\mp^e$.  
Since $p_- + p_+ + 1 =p_-^e + p_+^e + 1$ we have $p_{max}^e=p_-+p_++1$.
On the other hand, looking at above formula, it is clear that 
for all $k> p_-+1$, the term $\tih_+(k)-2\tih_-(k)-\dep \diracdelta_{p_-+1}(k)-T_{-(p_-+1)}\tih_+(k)$ has to vanish, 
while at $k = p_- + 1$ this expression needs to be $0$,  $-1$, or $-2$ (depending on the congruence class of 
$p_+$ modulo $4$), and at $k=p_-$ it always needs to be $-1$
(we invite the reader to draw the graph of $T_{p_-+1}$ for $k=p_-,\,p_-+1,\,p_-+2$ and compare this 
with the graph of $-h_{\lambda_+^e}$ to see this). 
It follows that $p_-+1=p_+$ or $p_-=p_+$. 

In the first case we have $p_{max}^e=p_-+p_++1=2p_+$.
Now we compute the rank $n$ in terms of $(p_-,p_+)$ and of $(p^e_-,p^e_+)$. The first method gives $n=n_-+n_+$ with 
\begin{equation*}
n_-= \begin{cases}
m_-^2/2=(p_-+1)^2/2, & \text { type (b) } \\
(m_--1)(m_-+1)/2=p_-(p_-+2)/2, & \text{ type (c) }
\end{cases}
\end{equation*}
Hence $n_-\leq p_+^2/2$. For $n_+$ we have: 
\begin{align*}
n_+ &=\lfloor \kappa_+(2(\kappa_+-1)+2\epsilon_+-1)/2\rfloor =\lfloor(p_++2-\epsilon)(p_++1+\epsilon)/4 \rfloor\\
&=\lfloor (p_++1)(p_++2)/4\rfloor \leq (p_++1)(p_++2)/4.
\end{align*}
Thus $n\leq  p_+^2/2+(p_++1)(p_++2)/4=3p_+(p_++1)/4+1/2$.
The second method gives $n=n_-^e+n_+^e=n_+^e=\lfloor (2p_++1)(2p_++2)/4\rfloor\geq p_+(p_++3/2)$. We conclude that 
$p_+(p_++3/2)\leq 3p_+(p_++1)/4+1/2$, which implies $p_+\leq 1$, hence $p_-\leq 0$, a contradiction. 

In the second case we have $p^e_{max}=2p_++1$. 
We now find that 
\begin{equation*}
n \leq  (p_++1)^2/2+(p_++1)(p_++2)/4 =3p_+(p_++1)/4+1,
\end{equation*} 
and $n=n^e_+=\lfloor (2p_++2)(2p_++3)/4\rfloor\geq p_+(p_++5/2)+1$. We thus see that 
$p_+(p_++5/2)\leq 3p_+(p_++1)/4$, which implies $p_+\leq -6$, also a contradiction.

In other words, we have shown that \eqref{eq:fdegmu}
has no solutions when $\lambda_+$ is of type (a) and meanwhile $\lambda_-$ is of type (b) or (c). \\

(4) Let us now consider the cases where $\lambda_+$ is of type (a) and $\lambda_-$ is of type (d).
We use the notation  $\Sigma_-^d=\diracdelta_1+\diracdelta_2+\diracdelta_4+\ldots+\diracdelta_{a_-}$, 
where $a_-=2^{s+1}-1$. Also recall that $\delta_-=1$, and $p_-=2^{s+2}-1$ with $s\geq 0$ in case (d), 
and we will use the notation $r_-=2^s-1$.

Similar to the previous case (with $\lambda_-$ is of type (c), and $\lambda_+$ of type (a)) we have
for all $k> 0$:
\begin{align*}
&\mult^{(\lambda_-,\lambda_+)}_{(\delta_-,\delta_+)}(k)-\delta_-\delta_+\diracdelta_1(k)\\
=&\Delta(\tih_-*\tih_+)(k)
-2\tih_-(k)+\Sigma^d_-(k)+\Delta_{\delta_+}(\tih_-)(k)
-\tih_+(k)+\Delta(\tih_+)(k)
-\delta_+\diracdelta_1(k)\\
=&((\Delta(\tih_-)+\Delta(\delta_0))*\tih_+)(k)
-2\tih_-(k)+\Sigma^d_-(k)
+\Delta_{\delta_+}(\tih_-)(k)
-\tih_+(k)
-\delta_+\diracdelta_1(k)\\
=&-T_{-(p_-+1)}\tih_+(k)-T_{p_-+1}\tih_+(k)
+\tih_+(k)-\Delta(\tih_+)(k)\\
&\phantom{-T_{-(p_-+1)}\tih_+(k)-T_{p_-+1}\tih_+(k)+}
+\Sigma^d_-(k)-2\tih_-(k)+\Delta_{\delta_+}(\tih_-)(k)
-\delta_+\diracdelta_1(k)\\
=&-T_{p_-+1}\tih_+(k)+\tih_+(k)+S_{p_++1}+\Sigma^d_-(k)-2\tih_-(k)-T_{-(p_-+1)}\tih_+(k)\\
&\phantom{-T_{-(p_-+1)}\tih_+(k)-T_{p_-+1}\tih_+(k)+}
-\delta_+\delta_{p_-+1}(k)+(\delta_+-1)\diracdelta_1(k)
\end{align*}
Like in the previous case, we conclude first of all that the right-hand side of \eqref{eq:fdegmu} needs to be of the form 
 $d_{a,b}^{extra,\textup{III}}$,
or $d_{a,b}^{extra,\textup{IV}}$ where $a$ and $b$ are such that one of 
the corresponding extra-special parameters $(m_-^e,m_+^e)$ equals $1/4$. 
Let us use the same notations as in the previous case. We again have $\{p_-^e,p_+^e\}=\{-1,p_{max}^e\}$
for some $p_{max}^e\in\ZZ_+$. We see that $p_{max}^e=p_-+p_++1$. 
As before we need that 
$T(k):=\tih_+(k)+S_{p_++1}+\Sigma_-(k)-2\tih_-(k)-T_{-(p_-+1)}\tih_+(k)
-\delta_+\delta_{p_-+1}(k)+(\delta_+-1)\delta_1(k)=0$ for all $k>p_-+1$, from which we conclude that $p_+<p_-+1$
(because of the presence of the function $S_{p_-+1}$). As before, we need that $T(p_-)=-1$, implying that $p_+=p_--1$
this time. 

Hence in this situation we have $p_+^e=p_-+p_++1=2p_-$.
We have $n=n_-+n_+$ with $n_-=8(r_-+1)^2=(p_-+1)^2/2$, and 
$n_+=\lfloor \kappa_+(2(\kappa_+-1)+2\epsilon_+-1)/2\rfloor 
=\lfloor(p_++2-\epsilon)(p_++1+\epsilon)/4 \rfloor=\lfloor p_-(p_-+1)/4\rfloor$, hence $n_+\leq p_-(p_-+1)/4$. 
Taken together we have 
$$
n\leq (p_-+1)^2/2+p_-(p_-+1)/4=(p_-+1)(3p_-+2)/4.
$$
On the other hand, we have $n=n_+^e\geq p_-(p_-+3/2)$. Hence we conclude that 
$p_-(p_-+3/2)\leq (p_-+1)(3p_-+2)/4$, which implies that $p_-^2/4+p_-/4\leq 1/2$, hence $p_-\leq 1$.
But then $p_+\leq 0$, a contradiction.

Hence we have shown that \eqref{eq:fdegmu} has no solutions for $\lambda_+$ of type (a) and $\lambda_-$ of type (d). 
\\

(5) Next let us consider the cases with $\lambda_+$ of type (a) and $\lambda_-$ of type (e).
We now have $r_-=2^s-1$ for some $s\geq 0$, and $a_-=2(r_-+1)$, $p_-=2a_-=2^{s+2}\geq 4$. 
Moreover, $\delta_-=0$ now.
We write $\Sigma^e_-=\diracdelta_1+\diracdelta_2+\diracdelta_4+\ldots+\diracdelta_{p_-}$.
Similar to the previous cases (with $\lambda_-$ is of type (b), and $\lambda_+$ of type (a)) we have
for all $k> 0$:
\begin{align*}
&\mult^{(\lambda_-,\lambda_+)}_{(\delta_-,\delta_+)}(k)-\delta_-\delta_+\diracdelta_1(k)\\
=&\Delta(\tih_-*\tih_+)(k)
-2\tih_-(k)+\Sigma^e_-(k)+\Delta_{\delta_+}(\tih_-)(k)
-\tih_+(k)\\
=&(\Delta(\tih_-)*\tih_+)(k)
-2\tih_-(k)+\Sigma^e_-(k)
+\Delta_{\delta_+}(\tih_-)(k)
-\tih_+(k)\\
=&-T_{p_-+1}\tih_+(k)
+T_{p_-}\tih_+(k)-T_{p_--1}\tih_+(k)\\
&\phantom{is dit genoeg?}
-T_{-(p_-+1)}\tih_+(k)+T_{-p_-}\tih_+(k)
-T_{-(p_--1)}\tih_+(k)\\
&\phantom{is dit genoeg?Met dit?}
+\tih_+(k)+\Sigma^e_-(k)+\Delta_{\delta_+}(\tih_-)(k)-2\tih_-(k)
\end{align*}
If we assume that $p_+\geq 3$ then the maximum of the support of the sum $S_6$ of the first 6 terms (the translations of $\tih_+$, but not 
including the non translated term $\tih_+$) is $k=p_-+p_++1$, and the maximum of the support of the remaining terms 
$T_R:=\Sigma^e_-(k)+\Delta_{\delta_+}(\tih_-)(k)-2\tih_-(k)$ (excluding $\tih_+$) is 
equal to $p_-+1$ (if $\delta_+=1$) or $p_-$ (if $\delta_+=0$). Since $p_+=3$ corresponds to $\delta_+=0$, 
we see that $T_R=0$ on $p_-+p_++1,\,p_-+p_+,\,p_-+p_+-1,\,p_-+p_+-2$. The values of $S_6$ there are easily seen to be 
$-1,\,0,\,-2,\,-1$ respectively. Since $\tih_+$ is also $0$ at these arguments (since $p_-+p_+-2 > p_+$) these values are the 
multiplicities of the factors $(1+q^k)$ in $\mu_{\delta_-,\delta_+}^{(\lambda_-,\lambda_+)}$ for these values of $k$. 
If $\respt$ solves \eqref{eq:fdegmu} we conclude that the right-hand side $d_{rhs}$ of \eqref{eq:fdegmu} is of the form 
$d_{a,b}^{extra,\textup{III}}$,
or $d_{a,b}^{extra,\textup{IV}}$ where $a$ and $b$ are such that one of 
the corresponding extra-special parameters $(m_-^e,m_+^e)$ equals $3/4$. 
Hence we now know $d_{rhs}$ exactly in terms of $p_-$ and $p_+$.

Now consider the behaviour of $S_6-d_{rhs}$ at the arguments $k=p_-+1,\,p_-,\,p_--1,\,p_--2$. 
Looking at the three positively translated graphs of $\tih_+$ for all possible 
congruence classes of $p_+$ modulo $4$, we see that these values are $\delta_+,\,1-\delta_+,\,1+\delta_+,\,2$ respectively, 
and the value of $T_R$ at these arguments is $-\delta_+,\,\delta_+-1,\,-\delta_+-2,\,\Sigma^e_-(p_--2)-4$. The 
sum of these gives the values $0,\,0,\,-1,\,\Sigma^e_-(p_--2)-2$, and if we add $\tih_+$ to this we should get $0$ everywhere.
We conclude that $p_+=p_--1$ and $p_-=4$, otherwise this is not possible. However, then $p_+=3$ which 
corresponds to $\delta_+=0$, so that $(\delta_-,\delta_+)=(0,0)$. But $d_{rhs}=d^{extra,\textup{III}}_{a,b}$, with 
corresponding values $m_-^e=3/4$ and $m_+^e=9/4$, which yields $\delta_-^e=1$ and  $\delta_+^e=0$. 
One checks that $d_{rhs}$ is an extra-special supercuspidal unipotent degree for the symplectic group 
$\mbf{Sp}_{36}$, while the ``solution" $\respt$ of \eqref{eq:fdegmu} is a residual point for $\mbf{SO}_{28}$, so this is not really a solution.\\

For $p_+\leq 2$, the value $\tih_+(k)$ is zero at $k$ equal to $p_-+p_++1,\dots,\,p_--1$. If $p_+=2$ we find that $S_6$ has the 
values $-1,\,0,\,-2,\,0,\,-2$ there, while $\delta_+=0$ and $T_R$ has the values $0,\,0,\,0,\,-1,\,-2$, which yields the total $-1,\,0,\,-2,\,-1,\,-4$,
which are not the multiplicities of a cuspidal unipotent formal degree, and for $p_+=1$ (with $\delta_+=1$) 
we get for these values of $S_6$ the sequence $-1,\,1,\,-2,\,1$ while for $T_R$  we get $0,\,-1,\,0,\,-3$ which yields the total 
$-1,\,0,\,-2,\,-2$, which is also not a sequence of multiplicities of a cuspidal unipotent formal degree.
We may finally conclude that there are no solutions to \eqref{eq:fdegmu} of the form $\respt_{\lambda_-,\lambda_+}$ with 
$\lambda_-$ of type (e) and $\lambda_+$ of type (a). \\

(6) Next let us consider the cases with $\lambda_-$ of type (b), (c) and $\lambda_+$ of type (d).
This implies that $\delta_+=1$, and $p_+\geq 3$.  
We compute in a similar way to the case where both $\lambda_\pm$ are of type (b) or (c):
\begin{align*}
&\mult^{(\lambda_-,\lambda_+)}_{(\delta_-,\delta_+)}(k)-\delta_-\delta_+\diracdelta_1(k)\\
=&\Delta(\tih_-*\tih_+)(k)
-2\tih_-(k)+\Delta_{\delta_+}(\tih_-)(k)
-2\tih_+(k)+\Sigma_+^d(k)+\Delta_{\delta_-}(\tih_+)(k)
-\delta_-\delta_+\diracdelta_1(k)\\
=&((\Delta(\tih_-)+\Delta_{\delta_-}(\delta_0))*\tih_+)(k)
-2\tih_+(k)+\Sigma_+^d(k)-2\tih_-(k)+\Delta(\tih_-)(k)
-\delta_-\diracdelta_1(k)\\
=&-T_{-(p_-+1)}\tih_+(k)-T_{+(p_-+1)}\tih_+(k)
+\Sigma_+^d(k)-2\tih_-(k)-\diracdelta_{(p_-+1)}(k)\\
=&-(\max\{0,p_-+p_++2-k\}+\max\{0,|p_+-p_-|-k\})+R(k)
\end{align*}
where we see that $R(k)=0$ for $k>p_-+1$, and $R(p_-+1)>0$.

This does not represent a supercuspidal unipotent formal degree, and therefore there are no 
solutions of \eqref{eq:fdegmu} with $\lambda_-$ of type (b) or (c) and $\lambda_+$ of type (d).\\

(7) Next let us consider the cases with $\lambda_-$ of type (b), (c) and $\lambda_+$ of type (e).
This implies that $\delta_+=0$, and $p_+\geq 4$.  
We again compute in a similar way to the case where both $\lambda_\pm$ are of type (b) or (c):
\begin{align*}
&\mult^{(\lambda_-,\lambda_+)}_{(\delta_-,\delta_+)}(k)-\delta_-\delta_+\diracdelta_1(k)\\
=&\Delta(\tih_-*\tih_+)(k)
-2\tih_-(k)+\Delta_{\delta_+}(\tih_-)(k)
-2\tih_+(k)+\Sigma_+^e(k)+\Delta_{\delta_-}(\tih_+)(k)\\
=&((\Delta(\tih_-)+\Delta_{\delta_-}(\delta_0))*\tih_+)(k)
-2\tih_+(k)+\Sigma_+^e(k)-2\tih_-(k)\\
=&-T_{-(p_-+1)}\tih_+(k)-T_{+(p_-+1)}\tih_+(k)
+\Sigma_+^e(k)-2\tih_-(k)
\end{align*}
We see that the values at $k=p_-+p_++1,\,p_-+p_+,\,p_-+p_+-1,\,p_-+p_+-2$
are $-1,-1,-2,-3$ respectively. 
This does not represent a supercuspidal unipotent formal degree, and therefore there are no 
solutions of \eqref{eq:fdegmu} with $\lambda_-$ of type (b) or (c) and $\lambda_+$ of type (e).

(8) Next let us consider the cases with $\lambda_\pm$ both of type (d).
This implies that $\delta_\pm=1$, and $p_\pm\geq 3$. We may assume that $p_+\leq p_-$
We again perform similar computations:
\begin{align*}
&\mult^{(\lambda_-,\lambda_+)}_{(\delta_-,\delta_+)}(k)-\diracdelta_1(k)\\
=&\Delta(\tih_-*\tih_+)(k)
-2\tih_-(k)+\Delta(\tih_-)(k)
-2\tih_+(k)+\Sigma_+^d(k)+\Delta(\tih_+)(k)+\Sigma_-^d(k) -\diracdelta_1(k)\\
=&((\Delta(\tih_-)+\Delta(\delta_0))*\tih_+)(k) 
-2\tih_+(k)+\Sigma_+^d(k)-2\tih_-(k)+\Delta(\tih_-)(k) +\Sigma_-^d(k)-\diracdelta_1(k)\\
=&-T_{-(p_-+1)}\tih_+(k)-T_{+(p_-+1)}\tih_+(k)-2\tih_-(k)-\Delta(\tih_+)(k) 
+\Delta(\tih_-)(k)+\Sigma_+^d(k)+\Sigma_-^d(k)
\end{align*}
If this is a solution of \eqref{eq:fdegmu} then, looking at the values at the arguments $k=p_-+p_++2,\,p_-+p_++1,\,p_-+p_+,\,p_-+p_+-1$, 
we conclude that the right-hand side $d_{rhs}$ of \eqref{eq:fdegmu} should be of the form 
$d_{a,b}^{ord,\textup{III}}$ or $d_{a,b}^{ord,\textup{IV}}$, because 
these values are $0,\,-1,\,-2,\,-3$.  
But at $k=p_-+1$ the value of the above sum of terms is greater or equal to $1-p_++1-1=1-p_+$, 
while the multiplicity function of $d_{rhs}$ is less or equal to $-1-p_+$. Hence there 
are no solutions $\respt_{\lambda_-,\lambda_+}$ of \eqref{eq:fdegmu} with $\lambda_\pm$ of type (d).\\

(9) Next let us consider the cases with $\lambda_\pm$ both of type (e).
This implies that $\delta_\pm=0$, and $p_\pm\geq 4$. We may assume that $p_+\leq p_-$
We again perform similar computations:
\begin{align*}
&\mult^{(\lambda_-,\lambda_+)}_{(\delta_-,\delta_+)}(k)\\
=&\Delta(\tih_-*\tih_+)(k)
-2\tih_-(k)
-2\tih_+(k)+\Sigma_+^e(k)+\Sigma_-^e(k)\\
=&-T_{p_-+1}\tih_+(k)
+T_{p_-}\tih_+(k)-T_{p_--1}\tih_+(k)\\
&\phantom{is dit genoeg?}
-T_{-(p_-+1)}\tih_+(k)+T_{-p_-}\tih_+(k)
-T_{-(p_--1)}\tih_+(k)\\
&\phantom{is dit genoeg?met dit dan ook toch, en beetje}-2\tih_-(k)+\Sigma_+^e(k) +\Sigma_-^e(k)
\end{align*}
This cannot represent a solution of \eqref{eq:fdegmu} since, looking at the values of this sum at the arguments 
$k=p_-+p_++2,\,p_-+p_++1,\,p_-+p_+,\,p_-+p_+-1,\,p_-+p_+-2$, 
we get $0,\,-1,\,-1,\,-2,\,-3$, which does not correspond to a supercuspidal formal degree.
Hence there is no solution $\respt_{\lambda_-,\lambda_+}$ of \eqref{eq:fdegmu} with $\lambda_\pm$ of type (e). \\

(10) Finally, let us consider the cases with $\lambda_-$ of type (d), and $\lambda_+$ of type (e).
This implies that $\delta_-=1$, $p_-\geq 3$, and $\delta_+=0$, $p_+\geq 4$. 
We again perform similar computations: 
\begin{align*}
&\mult^{(\lambda_-,\lambda_+)}_{(\delta_-,\delta_+)}(k)\\
=&\Delta(\tih_-*\tih_+)(k) -2\tih_-(k)+\Sigma_-^d(k) 
-2\tih_+(k)+\Sigma_+^e(k)+\Delta(\tih_+)(k)\\
=& \big( (\Delta(\tih_-)+\Delta(\delta_0))*\tih_+ \big)(k) -2\tih_-(k)-2\tih_+(k)+\Sigma_+^e(k) +\Sigma_-^d(k)\\
=&-T_{-(p_-+1)}\tih_+(k)-T_{+(p_-+1)}\tih_+(k)-2\tih_-(k)-\Delta(\tih_+)(k) 
+\Sigma_+^e(k)
+\Sigma_-^d(k)
\end{align*}
If $p_->3$, this cannot represent a solution of \eqref{eq:fdegmu} since in this case, looking at the values of this sum at the arguments 
$k=p_-+p_++2,\,p_-+p_++1,\,p_-+p_+,\,p_-+p_+-1,\,p_-+p_+-2$, 
we get $0,\,-1,\,-1,\,-2,\,-3$, which does not correspond to a cuspidal formal degree.
If $p_-=3$ the values of the sum on 
$k=p_++5,\,p_++4,\,p_++3,\,p_++2,\,p_++1,p_+$ are as follows
$0,\,-1,\,-1,\,-2,\,-2,\,-4$, which also does not correspond to a cuspidal unipotent formal degree.
Hence there is no solution $\respt_{\lambda_-,\lambda_+}$ of \eqref{eq:fdegmu} with $\lambda_-$ of type (d) and $\lambda_+$ of type (e) either.\\

We have finally checked all possibilities, which completes the proof of Theorem \ref{thm:A} for symplectic and even orthogonal groups. 
\end{proof}

Together with the results of Sections \ref{sub:oddorth} and \ref{sub:unit} we have proved Theorem \ref{thm:A}.


\end{document}